\documentclass[10pt]{amsart}
\usepackage{amsfonts}
\usepackage{ifpdf,lscape}
\usepackage{graphicx,version}
\usepackage{indentfirst}
\usepackage{warpcol}
\usepackage{graphicx}
\usepackage{xcolor}
\usepackage{float}
\usepackage{subfigure}
\usepackage{appendix}
\usepackage{extarrows}
\usepackage{bm}
\usepackage{amssymb,amsmath,amsthm,mathrsfs}%
    \numberwithin{equation}{section}%
    \numberwithin{table}{section}%
    \numberwithin{figure}{section}
\allowdisplaybreaks

\newtheorem{lemma}{Lemma}[section]

\DeclareMathAlphabet{\mathmatrix}{OT1}{ptm}{b}{n}
\DeclareMathAlphabet{\mathvector}{OT1}{ptm}{bx}{it}

\def\SS{\mathbb{S}}
\def\PP{\mathbb{P}}
\def\NN{\mathbb{N}}
\def\ZZ{\mathbb{Z}}

\def\RR{\mathbb{R}}

\def\tr{\mathsf{T}}
\def\mathbf{\boldsymbol}
\def\alph{\boldsymbol{\alpha}}

\def\e{\mathrm{e}}
 \def\CP{{\mathcal P}}
  \def\CH{{\mathcal H}}

\def\ball{\mathbb{B}^d}
\def\sph{\mathbb{S}^{d-1}}
 \def\o{{\omega}}
 \def\s{\sigma}
\newcommand{\wh}{\widehat}

\def\diag{\operatorname{diag}}

\def\spn{\operatorname{span}}

 \ifpdf
 \DeclareGraphicsExtensions{.pdf,.png,.jpg}
 \else
 \DeclareGraphicsExtensions{.eps}
 \fi

\begin{document}
\title[Spectral and Spectral Element Methods for Schr\"{o}dinger Equations]{Efficient Spectral and Spectral Element Methods for  Eigenvalue Problems of  Schr\"{o}dinger Equations with an Inverse Square Potential}

\thanks{The first author was partially  supported by the National Natural Science Foundation of China (No. 91130014, 11471312 and 91430216). The second author was partially  supported by the US National Science Foundation (DMS-1419040) and the National Natural Science Foundation of China (No. 11471031 and 91430216).}

\author[H. Li]{HUIYUAN LI}
\address{State Key Laboratory of Computer Science/Laboratory of Parallel Computing,  Institute of Software, Chinese Academy of Sciences, Beijing 100190, China.}
\email{huiyuan@iscas.ac.cn}

\author[Z. Zhang]{Zhimin Zhang}
\address{Beijing Computational Science Research Center, Beijing 100193, China. Also at Department of Mathematics, Wayne State University, Detroit, MI 48202, USA}
\email{zmzhang@csrc.ac.cn,zzhang@math.wayne.edu}


\begin{abstract}
In this article, we study numerical approximation of eigenvalue problems of the Schr\"{o}dinger operator $\displaystyle -\Delta u + \frac{c^2}{|x|^2}u$. There are three stages in our investigation: We start from a ball of any dimension, in which case the exact solution in the radial direction can be expressed by  Bessel functions of fractional degrees. This knowledge helps us to design two novel spectral methods by modifying the polynomial basis to fit the singularities of the eigenfunctions. At the second stage, we move to circular sectors in the two dimensional setting. Again the radial direction can be expressed by Bessel functions of fractional degrees. Only in the tangential direction some modifications are needed from stage one. At the final stage, we extend the idea to arbitrary polygonal domains. We propose a mortar spectral element approach: a polygonal domain is decomposed into several sub-domains with each singular corner including the origin covered by a circular sector, in which origin and corner singularities are handled similarly as in the  former stages, and the remaining  domains are either a standard quadrilateral/triangle or a quadrilateral/triangle with a circular edge, in which the traditional polynomial based spectral method is applied. All sub-domains are linked by mortar elements (note that we may have hanging nodes). In all three stages, exponential convergence rates are achieved. Numerical experiments indicate that our new methods are superior to standard polynomial based spectral (or spectral element) methods and $hp$-adaptive methods. Our study offers a new and effective way to handle eigenvalue problems of the Schr\"{o}dinger operator including the Laplacian operator on polygonal domains with reentrant corners.
\end{abstract}





\maketitle

\section{Introduction}

 The Schr\"{o}dinger operator is extremely important in science
and there are several different  forms of this remarkable operator.
The Schr\"{o}dinger operator with the inverse square singular potential has attracted  quite a large interest in the recent literature owing to its fundamental role both in mathematics and in physics.
Mathematically, the inverse square potential possesses the same homogeneity or ``differential order" as the Laplacian, while it usually invokes strong singularities of  the Schr\"{o}dinger  eigenfunctions  and thus  cannot be treated as a lower order perturbation term \cite{KSWW75,CH06,FT06,FMT07}.
 On the other hand,  the inverse square potential represents an intermediate threshold between the regular potential and singular potential in nonrelativistic quantum mechanics \cite{Case,FLS71}.
Furthermore, the inverse square singular potential also arises in many other fields, such as nuclear physics, molecular physics, and quantum cosmology; we  refer to \cite{Case,FLS71} for a comprehensive overview.
Therefore, new tools and approaches  are urgently needed   for  such Schr\"{o}dinger operators  both in  analysis  and in numerics.
In addition, the geometry of the domain such as the presence of reentrant corners  also plays  a critical role which may reduce the regularity of the eigenfunctions.

In this article, we consider the eigenvalue problem of  the Schr\"{o}dinger equation with an inverse square potential:
\begin{align}
\label{eigenP}
\begin{cases}
 -\Delta u + \dfrac{c^2}{|x|^2} u = \lambda u, & \text{ in } \Omega, \\
 u = 0 , & \text{ on }  \partial\Omega,
 \end{cases}
\end{align}
 where
$\Omega$ is a bounded domain in $\mathbb{R}^d$ and  the origin $O$
is assumed to be in $ \overline{\Omega}$.
 Here we consider a Dirichlet boundary condition, but
other boundary conditions can be treated similarly.

Define the Sobolev spaces
\begin{align*}
W^1(\Omega) = H^1(\Omega) \cap L^2_{r^{-2}}(\Omega),
\qquad W^1_0(\Omega) = H^1_0(\Omega) \cap L^2_{r^{-2}}(\Omega),
\end{align*}
equipped with the norm
\begin{align*}
\|u\|_{W^1(\Omega)}  = \left(\|\nabla u\|^2 + \|u\|_{r^{-2}}^2\right)^{1/2}.
\end{align*}
Then the  variational  form of \eqref{eigenP} reads: Find
$\lambda\in \mathbb{R}$ and $ u\in W_0^1(\Omega)\setminus\{0\}$ such that
\begin{align}
\label{variationalP}
a(u,v):=  (\nabla u, \nabla v)_{\Omega} + c^2 (u,v)_{r^{-2},{\Omega}}  = \lambda(u,v)_{\Omega}, \quad v\in W_0^1(\Omega).
\end{align}
By  the Sturm-Liouville theory, there exists a sequence of eigenvalues
$$
     0 <  \lambda_1 < \lambda_2 \le   \dots \le \lambda_k  \le \dots \nearrow +\infty.
$$
It is well known that  that  $\lambda_k = \mathcal{O}( k^{2/d})$ for Laplacian eigenvalues ($c=0$)  as $k$ tends to infinity \cite{Weyl1911},  and this result is also valid for $c\neq 0$ by a Hardy-type inequality.
 Based on the variational form, Galerkin type numerical schemes can be designed. However,
low order methods have only limited  convergence rates, even if adaptive schemes are applied.
Readers are referred to  \cite{Reddien,LO14,LO15}  and the references therein  for this line of research.
Likewise, owing to the strong singularities of the underlying eigenfunctions arising from  both the singular potential
and the reentrant/obtuse corners of the domain,  classic  high order methods  including spectral/spectral element methods usually fail to achieve an exponential rate of convergence (see \S \ref{BallNE}  and refer  to \cite{BNZ05, GS07}).

The aim of this article is to  propose novel numerical methods for \eqref{eigenP}  with the intention  of reviving  spectral methods and spectral element methods.
A key idea is to use specially designed spectral basis functions to mimic the singular behavior of
eigenfunctions. We start from $\Omega$  a ball of dimension $d$, when
the radial component of an eigenfunction can be expressed explicitly by  Bessel functions of degree $\nu = \sqrt{(n+d/2-1)^2+c^2}$ together with the multiplier $r^{1-d/2}$.
Based on this knowledge, two classes of non-polynomial  Sobolev orthogonal basis functions can be designed to
incorporate the singularity $r^{\nu+1-d/2}$
to achieve an exponential rate of convergence. This idea is then extended to $\Omega$ being a sector,
and
to simplify the presentation, we concentrate on the two dimensional setting from now on. Our ultimate goal is for
$\Omega$  to be a polygonal domain, especially with reentrant corners. We propose a novel mortar
spectral element method: at each  singular corner including the origin, we attach a circular disc/sector,
on which a class of non-polynomial spectral basis functions are applied which depend on  the angle of the corner.
Other parts of  $\Omega$ are decomposed into  quadrilaterals/triangles, where some have one circular edge.
On these sub-domains, traditional spectral polynomial basis functions are used.  The two types of sub-domains are
linked smoothly by the mortar element idea. Again, we observe the exponential rate of convergence $\e^{-\sigma\sqrt{DoF}}$ with an almost uniform $\sigma$ for consecutive eigenvalues.
 Note that this convergence rate is superior to the optimal $hp$-version rate
$\e^{-\sigma\sqrt[3]{DoF}}$ in the literature \cite{GuiBab86a,GuiBab86b,GuiBab86c}, where $\sigma$ may vary from case to case depending on the singularity intensity of the eigenfunctions.

It is worth pointing out that the idea of inserting singularity terms into the basis functions was used in the literature,  at the cost of
destroying sparsity of the resulting algebraic matrix system. While this approach improves the rate of convergence to some extend, depending on
how many singularity terms are introduced  \cite{Grisvard85},  it cannot reach the exponential rate  of our methods, where
we target the entire singularity, not just a few leading terms. In this way, we are able to construct
orthogonal basis functions, which leads to very sparse (and sometimes diagonal) matrices.

In this paper, we only present our numerical algorithm and demonstrate its effectiveness by comparing it with state of the art methods. Related theoretical issues will be discussed in a separate work.

\section{Preliminary}

\subsection{Notation and conventions}
Let $ \Omega\subset \RR^d$ ($d\ge 1$) be a bounded domain and $ w$ be a generic weight function. Denote by $(u,v)_{  w, \Omega}=\int_{ \Omega}u(x)v(x) w dx$ and $\| \cdot \|_{ w, \Omega}$ the inner product and the norm of $L^2_{ w}( \Omega)$,  respectively. In addition, we use $H^s_{ w}( \Omega)$  and $H^s_{0, w}( \Omega)$ to denote the usual weighted Sobolev spaces, whose norms and seminorms  are denoted by $\| \cdot \|_{s,  w, \Omega}$ and  $|\cdot|_{s, w, \Omega}$, respectively.  In cases where no confusion would arise, $ w$ (if $  w\equiv1$) and  $ \Omega$ may be dropped from the notations.

Let $\NN$ and $\NN_0$  be the sets of the positive integers and  non-negative integers, respectively.
For any $k\in \NN_0$, we denote by $\PP_k( \Omega)$ the space of polynomials of total degree $\le k$ on $ \Omega$.

\subsection{Spherical Harmonics}
 Let $\CP_n^d$ denote the space of homogeneous polynomials of degree
$n$ in $d$ variables.
Harmonic polynomials of $d$-variables are polynomials in $\CP_n^d$ that satisfy the Laplace equation
$\Delta Y = 0$. Spherical harmonics are the restriction of harmonic polynomials on the unit sphere.
Let $\mathcal{H}_n^d$ denote the space of spherical harmonic polynomials of degree $n$. It is well--known
that
$$
    \dim \CP_n^d  = \binom{n+d-1}{n} \quad \hbox{and} \quad
      a_n^d: = \dim \mathcal{H}_n^d = \binom{n+d-1}{n} - \binom{n+d-3}{n-2}.
$$
If $Y \in \mathcal{H}_n^d$, then $Y(x) = r^n Y(\xi)$ in spherical--polar coordinates $x = r \xi$ with $|\xi|=1$. We call
$Y(x)$ a solid spherical harmonic. Evidently, $Y$ is uniquely determined by its restriction on the sphere.
We shall also use $\CH_n^d$ to denote the space of solid spherical harmonics.

Spherical harmonics of different degrees are orthogonal with respect to the inner product
$$
        ( f, g )_{\sph}: = \int_{\sph} f(\xi) g(\xi) d\s(\xi),
$$
where $d \s$ is the surface measure.  Further let $\{Y_\ell^n: 1 \le \ell \le a_n^d\}$ be the orthonormal (real) basis of $\CH_n^d$, $n\in \NN_0$, such that
\begin{align*}
 ( Y_\ell^n, Y_\iota^m)_{\sph}=\omega_d\delta_{n,m} \delta_{\ell,\iota},
 \quad 1\le \ell\le a_n^d, \, 1\le  \iota\le a_m^d,  \, m\ge 0, \, n\ge 0,
 \end{align*}
where $\o_{d}={2\pi^{\frac{d}{2}}}/{\Gamma(\frac{d}{2})}$ is the surface
area.

In spherical polar coordinates, the Laplace operator can be written as
\begin{equation} \label{eq:Delta}
   \Delta = \frac{d^2}{d r^2} + \frac{d-1}{r} \frac{d}{dr} + \frac{1}{r^2} \Delta_0,
\end{equation}
where $r = \|x\|$ and $\Delta_0$, the spherical part of $\Delta$, is the Laplace-Beltrami operator that has spherical
harmonics as eigenfunctions; more precisely, for $n = 0,1,2,\ldots$,
\begin{equation} \label{eq:LaplaceBeltrami}
        \Delta_0  Y = - n(n+d-2) Y, \qquad Y \in \CH_n^d.
\end{equation}

For more information regarding spherical harmonics, readers are referred to \cite{DaiXu2013,DX14}.

\subsection{Generalized Jacobi polynomials} Let $I=(-1,1)$.
The hypergeometric  representation for the classic Jacobi polynomials  $J^{\alpha_1,\alpha_2}_n(\zeta),\, \zeta\in I, n\in \NN_0$ with $\alpha_1,\alpha_2>-1$,
\begin{align}
\label{JHyper}
\begin{split}
 {J}^{\alpha_1,\alpha_2}_n(\zeta) =&\binom{n+\alpha_1}{n} {}_2F_1(-n, n+\alpha_1+\alpha_2+1;\alpha_1+1;\frac{1-\zeta}{2}),
 \quad -n-\alpha_1-\alpha_2\not \in\{ 1,2,\dots,n\}
 \end{split}
\end{align}
furnishes the extension  of $J^{\alpha_1,\alpha_2}_n(\zeta)$ to arbitrary  $\alpha_1$ and $\alpha_2$.
The restriction  $-n-\alpha_1-\alpha_2\not \in\{ 1,2,\dots,n\}$  is enforced such  that
the generalized Jacobi polynomial  $J^{\alpha_1,\alpha_2}_n(\zeta)$  is exactly of  degree $n$,
since a degree reduction occurs in \eqref{JHyper}  if and only if   $-n-\alpha_1-\alpha_2\in \{1,2,\dots,n\} $.

Denote by $\chi(x)$ a ``characteristic" function for negative integers such that $\chi(x)=-x$  if $x\in \ZZ^-$ and $\chi(x)=0$ otherwise.
 The generalized Jacobi polynomials $J^{\alpha_1,\alpha_2}_n(\zeta),\ n\ge \chi(\alpha_1)+\chi(\alpha_2)$
 defined by  \eqref{JHyper} with  $\alpha_1\in \ZZ^-$ and/or $\alpha_2\in \ZZ^-$  are exactly what were defined in \cite{LS09},  and also coincide, up to certain constants, with those defined in \cite{GSW2006}.

 For $\alpha_1,\alpha_2\in \ZZ^-\cup (-1,\infty)$,
 the generalized Jacobi polynomials $J_n^{\alpha_1,\alpha_2}, \, n \ge \chi(\alpha_1)+\chi(\alpha_2)$
 are
mutually orthogonal with respect to the weight function $w^{\alpha_1,\alpha_2}:=w^{\alpha_1,\alpha_2}
(\zeta) = (1-\zeta)^{\alpha_1} (1+\zeta)^{\alpha_2}$ on $I$ \cite{GSW2006,LS09}, i.e.,
\begin{align}
\begin{split}
\label{Jorth}
        (J_m^{\alpha_1,\alpha_2}, &J_n^{\alpha_1,\alpha_2})_{w^{\alpha_1,\alpha_2},I}= \gamma^{\alpha_1,\alpha_2}_{n}\delta_{m,n}
        \\
      := &\frac{2^{\alpha_1+\alpha_2+1}}
  {{2n+\alpha_1+\alpha_2+1}}\, \frac{\Gamma(n+\alpha_1+1)\Gamma(n+\alpha_2+1)}
  {\Gamma(n+1)\Gamma(n+\alpha_1+\alpha_2+1)}\,  \delta_{m,n}, \quad
        m,n\ge \chi(\alpha_1)+\chi(\alpha_2),
   \end{split}
\end{align}
where $\delta_{m,n}$ is the Kronecker delta. Moreover, the generalized Jacobi polynomials satisfy
the following differential recurrence  relation,
\begin{align}
\label{diff}
\partial_{\zeta}J_n^{\alpha_1,\alpha_2}(\zeta) = \frac{n+\alpha_1+\alpha_2+1}{2}J_{n-1}^{\alpha_1+1,\alpha_2+1}(\zeta),
\quad -n-\alpha_1-\alpha_2\not\in\{ 1,2,\dots,n\}.
\end{align}

 Of our great interest are those polynomials $J^{\alpha_1,\alpha_2}_n, n\in \NN_0$ with  $\alpha_1=-1$ and/or $\alpha_2=-1$.
At first, we directly obtain from \eqref{JHyper} that
\begin{align}
\label{GJacobir}
  &J_0^{-1,\alpha_2}(\zeta) =1, \quad  J_n^{-1,\alpha_2}(\zeta)
  =  \frac{n+\alpha_2}{n}  \frac{\zeta-1}{2}   J_{n-1}^{1,\alpha_2}(\zeta), \   n\ge 1, \quad \alpha_2>-1.
\end{align}
Meanwhile,  we  supplement the definition of $J^{-1,-1}_1$
and then  obtain the following  complete system,
\begin{align}
\label{Jm1m1}
   &J_0^{-1,-1}(\zeta) =1,  \quad  J_1^{-1,-1}(\zeta) =\zeta,  \quad J_n^{-1,-1}(\zeta)
  =  \frac{\zeta-1}{2}  \frac{\zeta+1}{2}
   J_{n-2}^{1,1}(\zeta),\  n\ge 2.
\end{align}
Such a supplementation preserves the symmetry properties of the classic Jacobi polynomials,
\begin{align}
\label{symmJ}
J^{\alpha_1,\alpha_2}_n(\zeta) = (-1)^n J^{\alpha_2,\alpha_1}_n(-\zeta) ,
\qquad n\in \NN_0, \  \alpha_1,\alpha_2 \in [-1,\infty).
\end{align}
For  more about the supplementation of $J^{\alpha_1,\alpha_2}_n$  for $-n-\alpha_1-\alpha_2\in \{1,2,\dots,n\} $,
please refer to \cite{LX14}.

\section{Novel spectral methods on an arbitrary ball}

Throughout this section, we assume that  $\Omega=\ball:=\big\{ x\in \mathbb{R}^d: |x|<1\big\}$ and then aim at  seeking the numerical solution to \eqref{eigenP}.
It is worthy to note that the classic spectral or spectral element methods for \eqref{eigenP}
possess only  limited algebraic convergence orders as shown in \S \ref{BallNE}.
 Here we propose two novel spectral methods for \eqref{eigenP}
with an exponential rate of convergence.

\subsection{Spectral-Galerkin method I}
Denote
$$\beta_n= \beta(n,c,d)=\sqrt{ c^2+(n+d/2-1)^2}.
$$
Inspired by the  classic spectral method  on a unit disk \cite{Shen97},  we define the ball functions
\begin{align*}
Q_{k,\ell}^{\alpha,n}(x) =J^{\alpha,2\beta_n}_k(2r-1)\, r^{\beta_n+1-d/2} Y^n_{\ell}(\xi),
\quad n , k\in \mathbb{N}_0, \, 1\le \ell \le a_n^d,
\end{align*}
where $(r,\xi)$ is the  spherical-polar coordinates such that $x= r\xi$   with $\|\xi\|=1$,
 and  $J^{\alpha,\beta}_k(\zeta)$ is the generalized Jacobi polynomial of degree $k$. 

\begin{lemma}
\label{SOBF} Denote $Q_{k,\ell}^{n}(x) =\dfrac{2k+2\beta_n}{k+2\beta_n} Q_{k,\ell}^{-1,n}(x)$.
Then
 $Q_{k,\ell}^{n}, \, k\in \mathbb{N}_0, \,
1\le \ell \le a_n^d, \, n\in \mathbb{N}_0 $,
form a Sobolev orthogonal basis in $W^1(\ball)$. More precisely,
\begin{align}
\label{Qortha}
\begin{split}
      (\nabla Q_{k,\ell}^{n}, &\nabla Q_{j,\iota}^{m})_{\ball} + c^2 (Q_{k,\ell}^{n},Q_{j,\iota}^{m})_{r^{-2},\ball}
    \\
     =\, & \omega_d  \delta_{m,n} \delta_{\ell,\iota} \delta_{k,j}
\left[(2k+2\beta_n) (1-\delta_{k,0})  + (\beta_n-d/2+1) \delta_{k,0}   \right].
\end{split}
\end{align}
Moreover,
\begin{align}
\label{Qorthb}
     (Q_{k,\ell}^{n},Q_{j,\iota}^{m})_{\ball} :=\, & \omega_d\delta_{n,m}\delta_{\ell,\iota}\times
     \begin{cases}
     \frac {( k+\beta_n )  ( k^2 +2 k\beta_n + 4 \beta_n^2-1) }{ ( k+\beta_n -1)  ( k+\beta_n+1) ( 2 k+2 \beta_n-1
 )  ( 2k+2 \beta_n+1) },
  & k=j\ge 1, \\[0.3em]
  \frac{1}{2(\beta_n+1)}, & k=j=0,\\[0.3em]
  - \frac { ( 2\beta_n-1) ( 2 \beta_n+1  ) }{   ( 2 k+2 \beta_n-1  )  ( 2 k+2 \beta_n+1  ) ( 2k+2\beta_n+3  )   },
     & j = k+1,\\[0.3em]
  - {\frac {    ( k+1  )  ( k+2 \beta_n+1  ) }{2 ( k+\beta_n+1  )   ( 2 k+2 \beta_n+1  )    ( 2 k+2 \beta_n+3  ) }},
   & j = k+2,\\[0.3em]
  - \frac { ( 2\beta_n-1) ( 2 \beta_n+1  ) }{   ( 2 j+2 \beta_n-1  )  ( 2 j+2 \beta_n+1  ) ( 2j+2\beta_n+3  )   },
     & k = j+1,\\[0.3em]
  - {\frac {    ( j+1  )  ( j+2 \beta_n+1  ) }{2 ( j+\beta_n+1  )   ( 2 j+2 \beta_n+1  )    ( 2 j+2 \beta_n+3  ) }},
   & k = j+2,\\[0.3em]
   0, & \text{otherwise}.
     \end{cases}
\end{align}
\end{lemma}

The proof is postponed to Appendix \ref{AppA}.

Define the approximation space
\begin{align*}
W_{K,N}=
\mathrm{span}\big\{ Q_{k,\ell}^{n}:  1\le \ell\le a_n^d,\, 1 \le k\le K, \, 0\le n \le N \big\}.
\end{align*}
The spectral-Galerkin approximation scheme to \eqref{eigenP}  reads:
 to find $u_{K,N}\in W_{K,N}$ such that
\begin{align}
\label{GalBF}
a (u_{K,N},v)= (\nabla u_{K,N}, \nabla v)_{\ball} + c^2 (u_{K,N},v)_{r^{-2},{\ball}}
  = \lambda_{K,N} (u_{K,N},v)_{\ball}, \qquad v\in W_{K,N}.
\end{align}
Assume
$$
u_{K,N}(x) =\sum_{n=0}^N \sum_{\ell=1}^{a_n^d}   \sum_{k=1}^K \widehat u_{k,\ell}^n Q_{k,\ell}^n(x),
$$
and
denote
$$
 \widehat u =   (  \widehat u^0_{1}, \widehat u^0_{2}, \dots,\widehat u^0_{a_0^d} ,
 \widehat u^1_{1}, \widehat u^1_{2}, \dots,\widehat u^1_{a_1^d}, \dots,
 \widehat u^N_{1}, \widehat u^N_{2}, \dots,\widehat u^N_{a_N^d}     )^{\tr}
 ,
 \quad \widehat u^n_{\ell} = (\widehat u^n_{1,\ell}, \widehat u^n_{2,\ell},\dots, \widehat u^n_{K,\ell}).
$$
Then the  discrete problem  \eqref{GalBF} is equivalent to the following
algebraic eigen system
\begin{align}
\label{AlgBF}
\diag(A_{\ell}^n) \widehat u  = \lambda_{K,N}  \diag( B_{\ell}^n) \widehat u,
\end{align}
where, in view of Lemma \ref{SOBF}, the stiffness matrices $A_{\ell}^n = [a( Q^{n}_{k,\ell}, Q^m_{j,\iota})  ]_{0\le k,j \le K}$ are diagonal;
and  the mass matrices $B_{\ell}^n=[( Q^{n}_{k,\ell}, Q^m_{j,\iota})_{\ball}  ]_{0\le k,j \le K}$  are penta-diagonal.
Thus  \eqref{AlgBF} can be decoupled into a series of algebraic eigen systems,
which can be solved in parallel,
\begin{align*}
A_{\ell}^n  \widehat u^n_{\ell}  = \lambda_{K,N}^{\ell,n}   B_{\ell}^n \widehat u^n_{\ell},
\quad 1\le \ell\le a_n^d,\, 0\le n \le N .
\end{align*}

\subsection{Spectral-Galerkin method II}
Our second novel method uses basis functions imitating the ball polynomials \cite{LX14},
\begin{align*}
P_{k,\ell}^{\alpha,n}(x) =J^{\alpha,\beta_n}_k(2r^2-1)\, r^{\beta_n+1-d/2} Y^{\,n}_{\ell}(\xi),
\quad
k\in \mathbb{N}_0,\,  1\le \ell\le a_n^d,\,  n\in \NN_0.
\end{align*}
In particular, each $P_{k,\ell}^{\alpha,n}(x)$ is
reduced to the ball polynomials $J^{\alpha,n+d/2-1}_k(2r^2-1)\, Y^{\,n}_{\ell}(x)$
in \cite{LX14} whenever $c=0$.

\begin{lemma}
\label{SOBP} Denote $P_{k,\ell}^{n}(x) =\dfrac{2k+\beta_n}{k+\beta_n} P_{k,\ell}^{-1,n}(x)$.
Then
 $P_{k,\ell}^{n}, \, k\in \mathbb{N}_0, \,
1\le \ell \le a_n^d, \, n\in \mathbb{N}_0 $,
form a Sobolev orthogonal basis in $W^1(\ball)$. More precisely,
\begin{align}
\label{Portha}
\begin{split}
      (\nabla &P_{k,\ell}^{n}, \nabla P_{j,\iota}^{m})_{\ball} + c^2 (P_{k,\ell}^{n},P_{j,\iota}^{m})_{r^{-2},{\ball}}
    \\
     =\, & \omega_d  \delta_{m,n} \delta_{\ell,\iota} \delta_{k,j}
\left[2(2k+\beta_n) (1-\delta_{k,0})  + (\beta_n-d/2+1) \delta_{k,0}   \right].
\end{split}
\end{align}
Moreover,
\begin{align}
\label{Porthb}
     (P_{k,\ell}^{n},P_{j,\iota}^{m})_{\ball} :=\, & \omega_d\delta_{n,m}\delta_{\ell,\iota}\times
     \begin{cases}
 \frac {1}{ 2 k+\beta_n+1 } +  \frac {1-\delta_{k,0}}{ 2 k+\beta_n-1 }  ,
  & k=j, \\
  - \frac{1}{ 2(2k+\beta_n+1) },
     & j = k+1,\\
  - \frac{1}{2( 2j+\beta_n+1)},
     & k = j+1,\\
   0, & \text{otherwise}.
     \end{cases}
\end{align}
\end{lemma}
The proof of the above lemma is postponed to Appendix \ref{AppA}.

Define the
approximation space
\begin{align*}
V_{K,N}=
\big\{ P_{k,\ell}^{n}:  1\le \ell\le a_n^d,\, 1 \le k\le K,\,  0\le n \le N \big\}.
\end{align*}
Then
approximation scheme for \eqref{eigenP} reads, to find $u_{K,N}\in V_{K,N}$ such that
\begin{align}
\label{GalBP}
a (u_{K,N}, v)= (\nabla u_{K,N}, \nabla v)_{\ball} + c^2 (u_{K,N}, v)_{r^{-2},{\ball}} = \lambda_{K,N} (u_{K,N}, v)_{\ball},
 \qquad v\in V_{K,N}.
\end{align}
Assume
$$
u_{K,N}(x) =\sum_{n=0}^N \sum_{\ell=1}^{a_n^d} \sum_{k=1}^K  \widehat u_{k,\ell}^n P_{k,\ell}^n(x).
$$
Then the  discrete problem  \eqref{GalBP} is equivalent to the following
algebraic eigen system,
\begin{align}
\label{AlgBP}
\diag(A_{\ell}^n) \widehat u  = \lambda_{K,N}  \diag( B_{\ell}^n) \widehat u.
\end{align}
In light of Lemma \ref{SOBP},  the stiffness matrices $A_{\ell}^n$ are diagonal;
and  the mass matrices $B_{\ell}^n$  are tridiagonal.
Thus the \eqref{AlgBP} can be also decoupled into a series of algebraic eigen systems,
which can be solved independently,
\begin{align*}
A_{\ell}^n  \widehat u^n_{\ell}  = \lambda_{K,N}^{\ell,n}   B_{\ell}^n \widehat u^n_{\ell},
\quad 1\le \ell\le a_n^d,\, 0\le n \le  N .
\end{align*}

\subsection{Numerical experiments}
\label{BallNE}
We now present some numerical results
using Sobolev-orthogonal basis functions to  Schr\"odinger equations
on the unit ball $\ball$  to demonstrate effectiveness of our proposed methods.
To make a comparison, we shall also show
numerical results by an adaptive finite element method  with graded meshes \cite{LO15}
and those by classic spectral methods on the disk/ball \cite{Shen97, LX14}.

We first note that the finite element method (FEM) has a low accuracy
and thus does not fit well for solving the Schr\"{o}dinger equation \eqref{eigenP}, even if
variants of adaptive techniques are applied.
We excerpt  from \cite{LO15} the errors of the adaptive FEM  with various degrees of freedom (DoF)  in Table
\ref{tab:FEM},  which verifies  our observation.

A  heuristic spectral method  inspired by  \cite{Shen97} utilizes the technique of  separation of variables by assuming   the  eigenfunction  $u=\wh u^n_{\ell}(r)  Y^n_{\ell}(\xi) $.  As a result, \eqref{eigenP}
is transformed into a singular equation in  $r$
as indicated
in \eqref{BeigenS} in the subsequent subsection.
Then  one
adopts  the following generalized Jacobi polynomials as basis functions
to solve  the reduced 1-D equation,
\begin{align*}
&J^{-1,d-3}_k(2r-1) , \qquad  2\le k\le K \ \text{ if }\  c^2+n^2\neq 0 , d=2\ \text{ and }\
1 \le k \le K \text{ if otherwise}.
\end{align*}
This scheme leads to an algebraic eigen system with a  tri-diagonal stiffness matrix and
a  penta-diagonal mass matrix \cite{Shen97}.

 Figures \ref{fig:errSMdisk} and \ref{fig:errSMball} depict the convergence behaviours of this spectral method for the 4 smallest Schr\"{o}dinger eigenvalues  for $c=1/2$ and $c=2/3$  on the unit disk and the unit ball.
 Without a mechanism to capture the singularity of eigenfunctions
 induced by the singular potential $r^{-2}$, this method has only a limited convergence rate
 instead of a spectrally high rate.
 It is observed specifically that the computational eigenvalues converge at an algebraic rate  $\mathcal{O}(K^{-4\beta_n})$, where $n$ is the degree of the spherical component of the eigenfunction. In particular, the eigenvalues corresponding to $n=0$ have the lowest  convergence rate  and poor accuracy even for very large $K$.

\begin{table}[h!]
\caption{Approximation  errors of the first, second, and sixth  Schr\"{o}dinger eigenvalues with $c=1/2$ on the unit disk
by the  finite element method  with $L$ levels of graded meshes in \cite{LO15}.
}
\hspace*{\fill}
\begin{tabular}{c||c|c|c|c|c|c|c}
$L$   & 0      & 1 &   2 & 3 & 4 & 5 & 6
 \\ \hline
DoF  & 48   & 224 & 961 & 3968 & 16129 & 65025 & 261121
\\ \hline
$\lambda_{1}$ & 9.467e-1 & 2.429e-1 & 5.690e-2 & 1.631e-2 & 3.957e-3 & 1.026e-3 & 2.637e-4
\\ \hline
$\lambda_{2}$ & 2.371 & 5.769e-1 & 1.433e-2 & 3.576e-2 & 8.938e-3 & 2.234e-3 & 5.586e-4
\\ \hline
$\lambda_{6}$ & --  & 3.892 &9.629e-1 & 2.493e-1 & 5.898e-2 & 1.510e-2 & 3.844e-3
\end{tabular}\hspace*{\fill}
\vspace*{0.5em}
\label{tab:FEM}
\end{table}

\begin{figure}[h!]
\hspace*{\fill}
\begin{minipage}[h]{0.35\textwidth}
\includegraphics[width=1\textwidth]{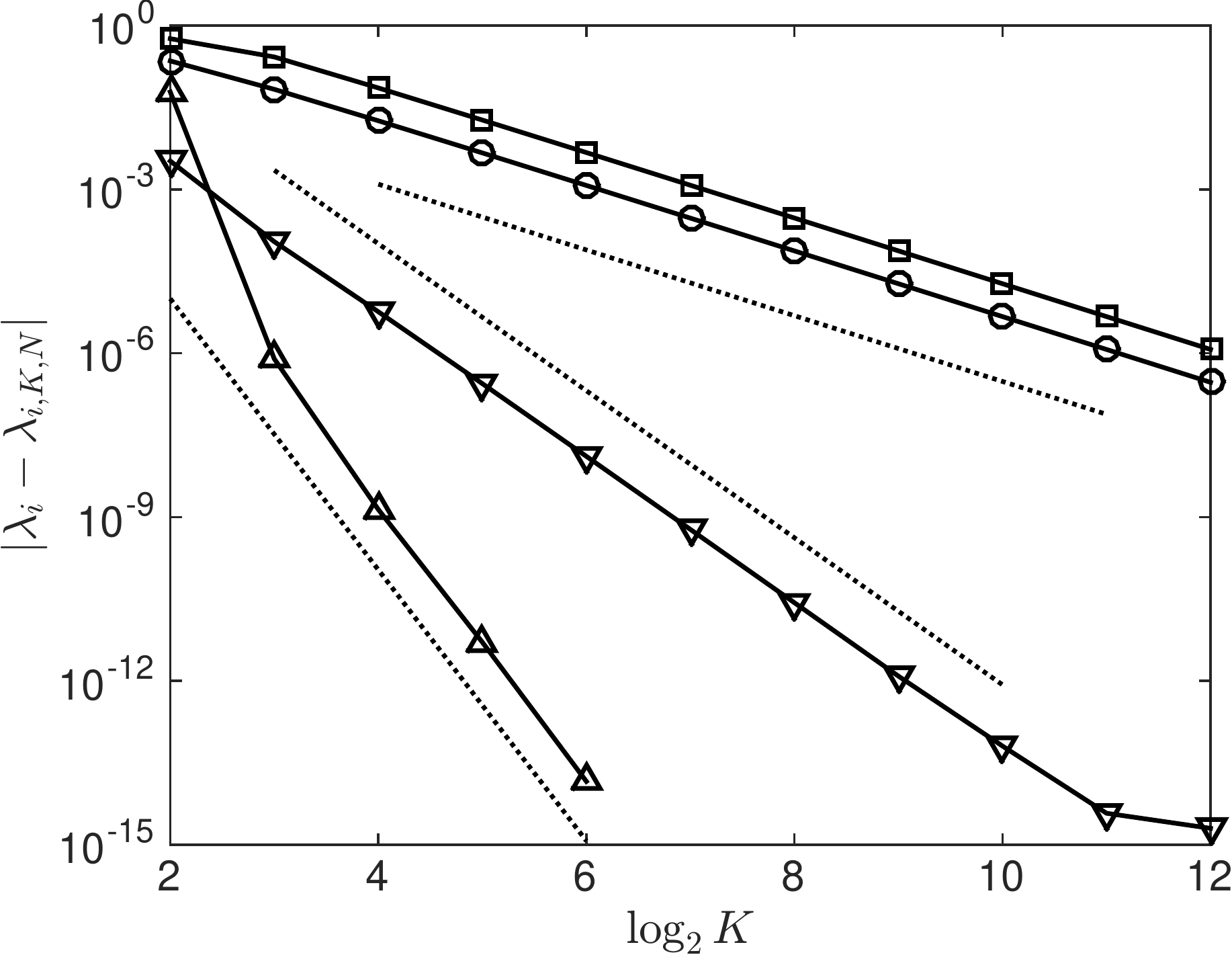}
\hspace*{\fill} (a). $c=1/2$. \hspace*{\fill}
\end{minipage}\hspace*{\fill}%
\begin{minipage}[h]{0.35\textwidth}
\includegraphics[width=1\textwidth]{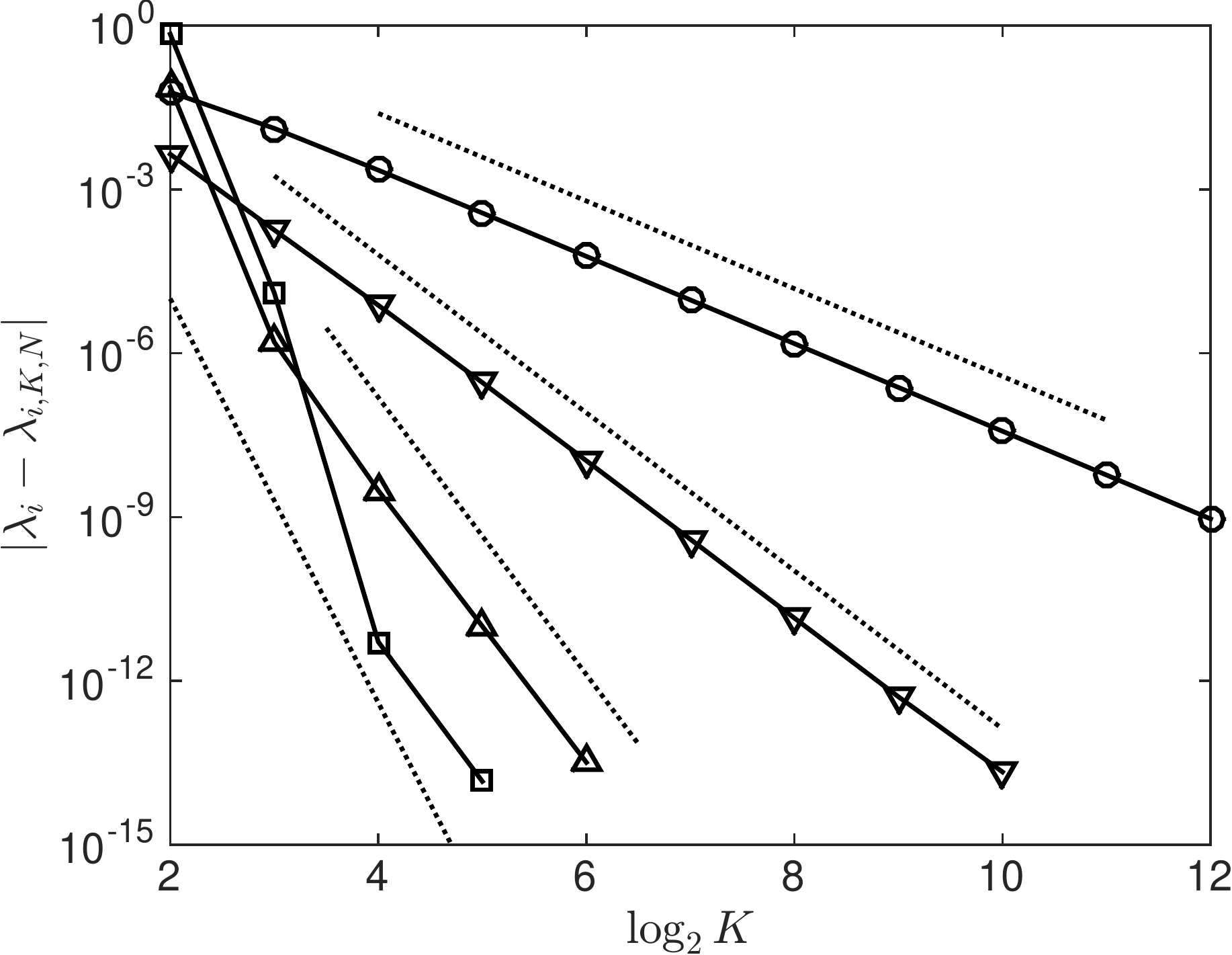}
\hspace*{\fill} (b). $c=2/3$.  \hspace*{\fill}
\end{minipage}\hspace*{\fill}
\caption{Approximation errors $|\lambda_i-\lambda_{i,K,N}|$  versus $K$ by the classic spectral method   inspired by  \cite{Shen97} on the unit disk.
 $\circ: \lambda_1$ ($n=0$);  $\triangledown:\lambda_2=\lambda_3$ ($n=1$);  $\vartriangle:\lambda_4=\lambda_5$ ($n=2$); $\square: \lambda_6$ ($n=0$)/$\lambda_6=\lambda_7$ ($n=3$).
Dashed lines:  $y = \sigma K^{-4\beta_n}$.}
\label{fig:errSMdisk}
\end{figure}
\begin{figure}[h!]
\hspace*{\fill}
\begin{minipage}[h]{0.35\textwidth}
\includegraphics[width=1\textwidth]{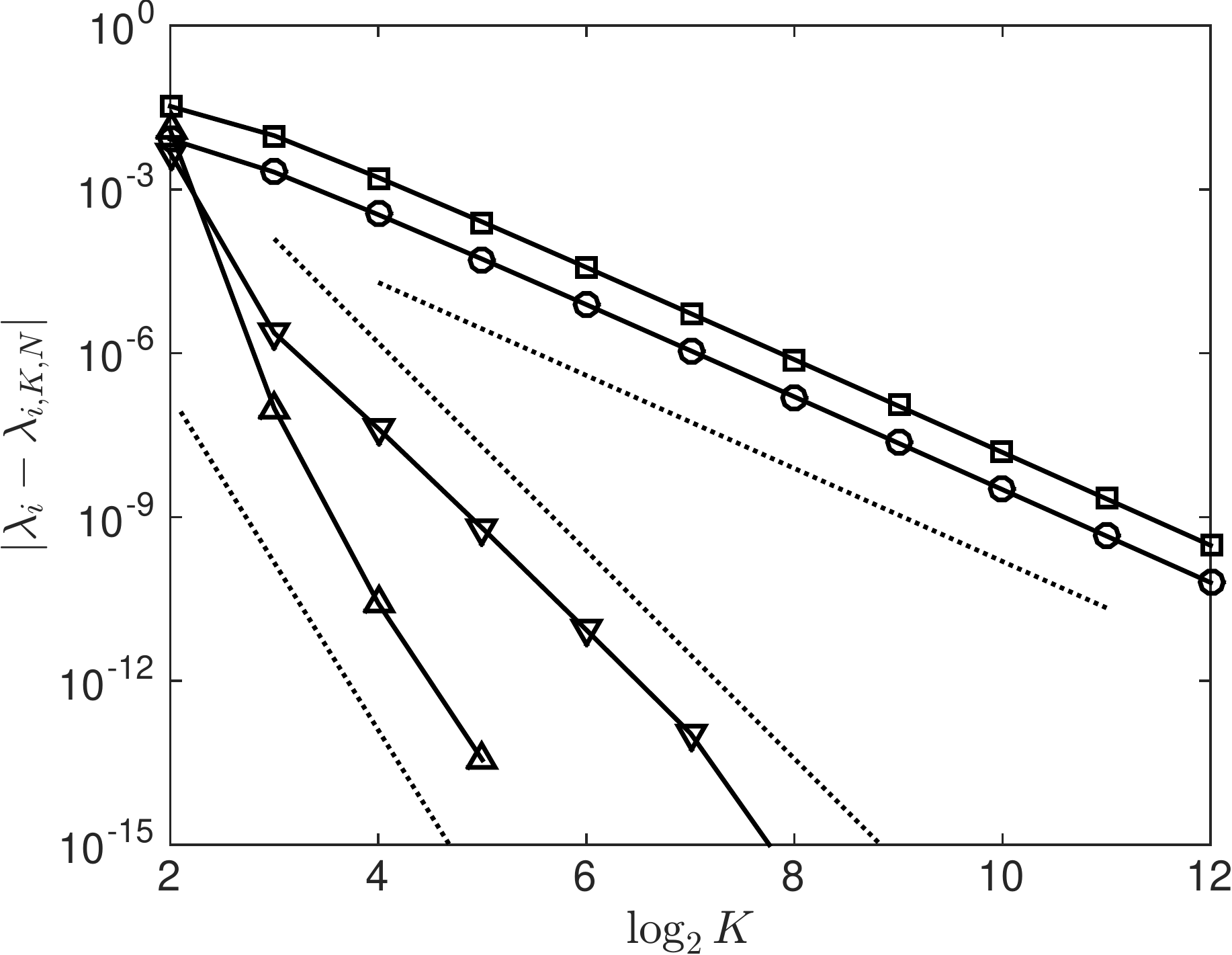}
\hspace*{\fill} (a). $c=1/2$. \hspace*{\fill}
\end{minipage}\hspace*{\fill}%
\begin{minipage}[h]{0.35\textwidth}
\includegraphics[width=1\textwidth]{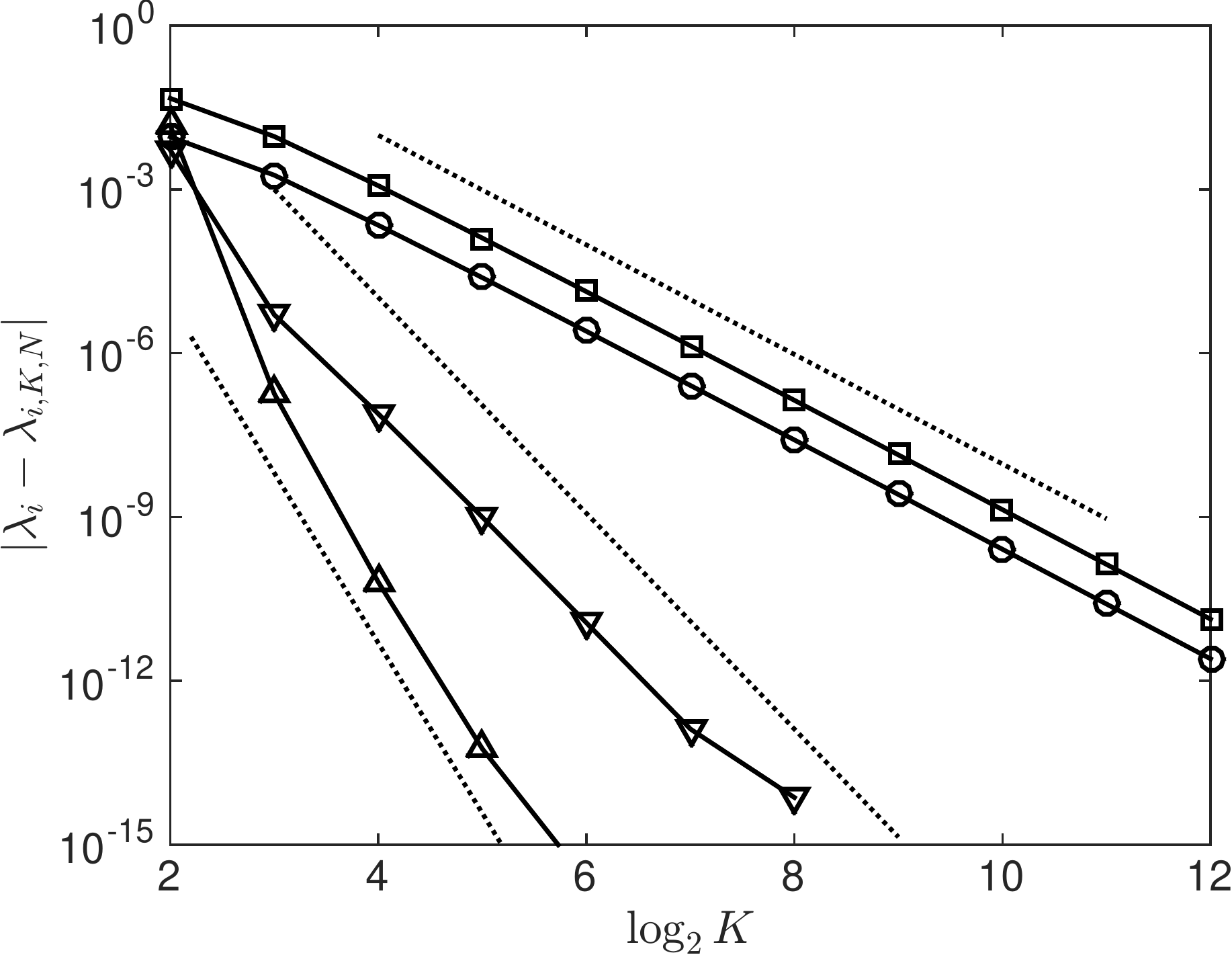}
\hspace*{\fill} (b). $c=2/3$.  \hspace*{\fill}
\end{minipage}\hspace*{\fill}
\caption{Approximation errors $|\lambda_i-\lambda_{i,K,N}|$  versus $K$ by the classic spectral method  inspired by  \cite{Shen97} on the unit ball.
 $\circ: \lambda_1$ ($n=0$),  $\triangledown:\lambda_2=\lambda_3=\lambda_4$ ($n=1$), $\vartriangle:\lambda_5=\dots=\lambda_9$ ($n=2$), $\square:  \lambda_{10} $ ($n=3$). Dashed lines:  $y = \sigma K^{-4\beta_n}$.}
\label{fig:errSMball}
\end{figure}

The polynomial spectral method  \cite{LX14}  for \eqref{eigenP} utilizes
the orthogonal ball polynomials as basis functions,
\begin{align*}
J^{-1,n+d/2-2}_k&(2r^2-1) Y^n_{\ell}(x) ,
\\
& 2\le k\le K \ \text{ if }\  c^2+n^2\neq 0 , d=2\ \text{ and }\
1 \le k \le K \text{ if otherwise}.
\end{align*}
Once again, this method leads to  a series of  independent algebraic eigenvalue problems with the tri-diagonal stiffness matrix and
 the penta-diagonal mass matrix.

The approximation errors  of the polynomial spectral method are plotted in Figures \ref{fig:errPSMdisk} and \ref{fig:errPSMball} in log-log scale for both $c=1/2$ and $c=2/3$ in $d=2,3$ dimensions.
We clearly see that the polynomial spectral method converges at a rate of $\mathcal{O}(K^{-2\beta_n})$,
which is only the half order of the classic spectral method inspired by \cite{Shen97}.
This even worse accuracy and convergence rate  confirm
the singularity of type  $r^{\rho}$ of the Schr\"{o}dinger eigenfunctions, which will be specified in \S \ref{Why}.

\begin{figure}[h!]
\hspace*{\fill}
\begin{minipage}[h]{0.35\textwidth}
\includegraphics[width=1\textwidth]{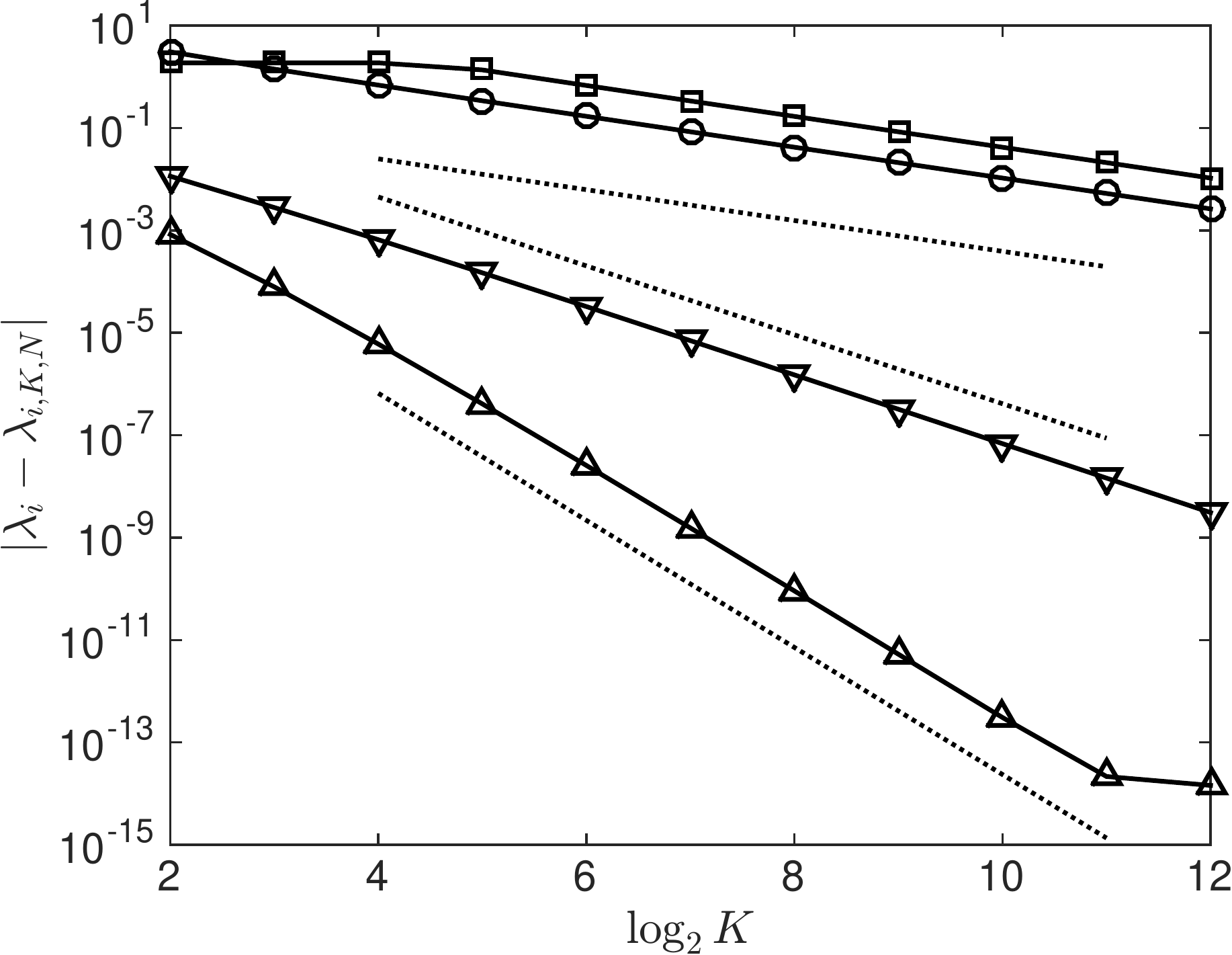}
\hspace*{\fill} (a). $c=1/2$. \hspace*{\fill}
\end{minipage}\hspace*{\fill}%
\begin{minipage}[h]{0.35\textwidth}
\includegraphics[width=1\textwidth]{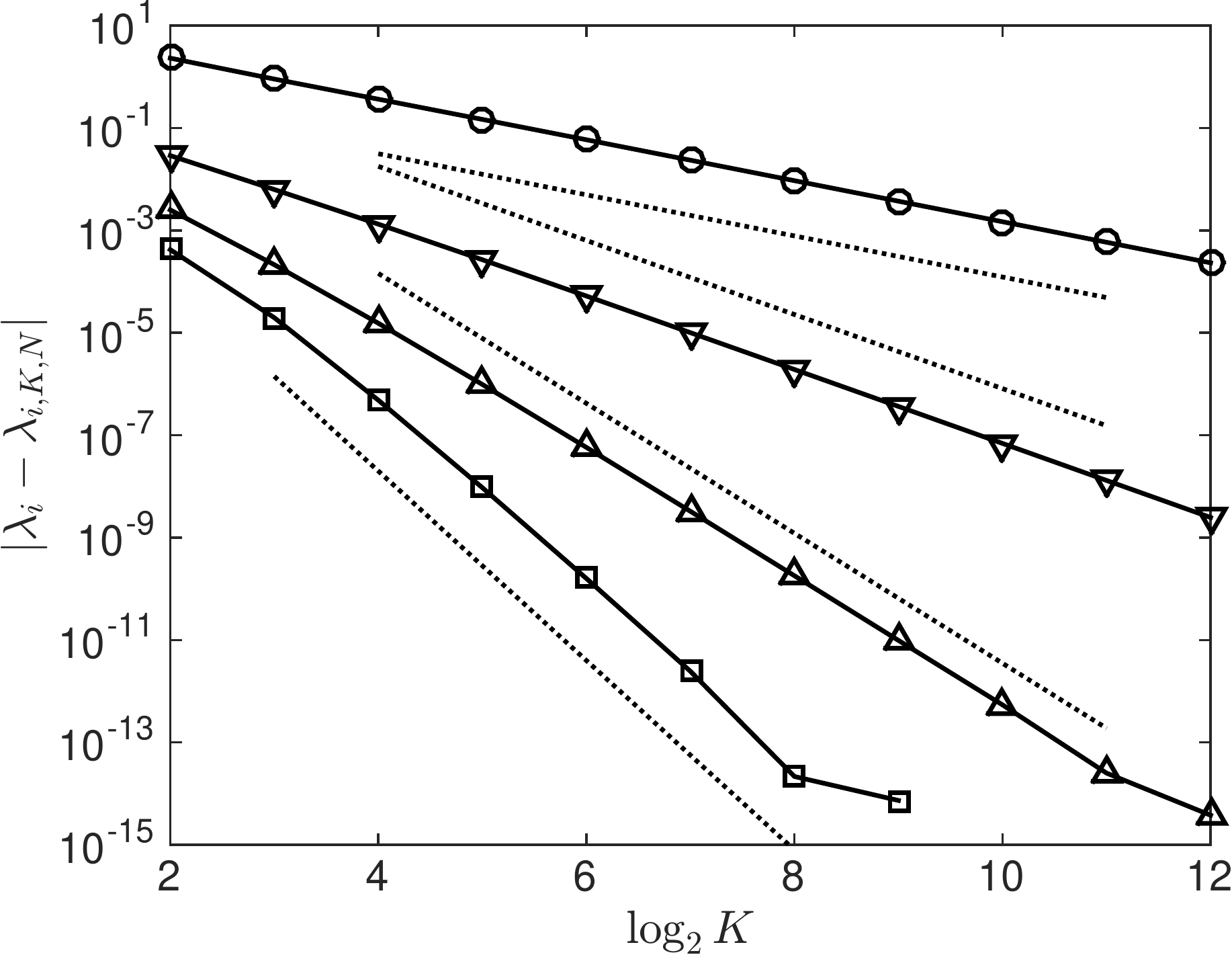}
\hspace*{\fill} (b). $c=2/3$.  \hspace*{\fill}
\end{minipage}\hspace*{\fill}
\caption{Approximation errors $|\lambda_i-\lambda_{i,K,N}|$  versus $K$ by the polynomial spectral method on the unit disk.
 $\circ: \lambda_1$ ($n=0$);  $\triangledown:\lambda_2=\lambda_3$  ($n=1$);  $\vartriangle:\lambda_4=\lambda_5$ ($n=2$); $\square: \lambda_6$  ($n=0$)/$\lambda_6=\lambda_7$ ($n=3$).
Dashed lines:  $y = \sigma K^{-2\beta_n}$.}
\label{fig:errPSMdisk}
\end{figure}
\begin{figure}[h!]
\hspace*{\fill}
\begin{minipage}[h]{0.35\textwidth}
\includegraphics[width=1\textwidth]{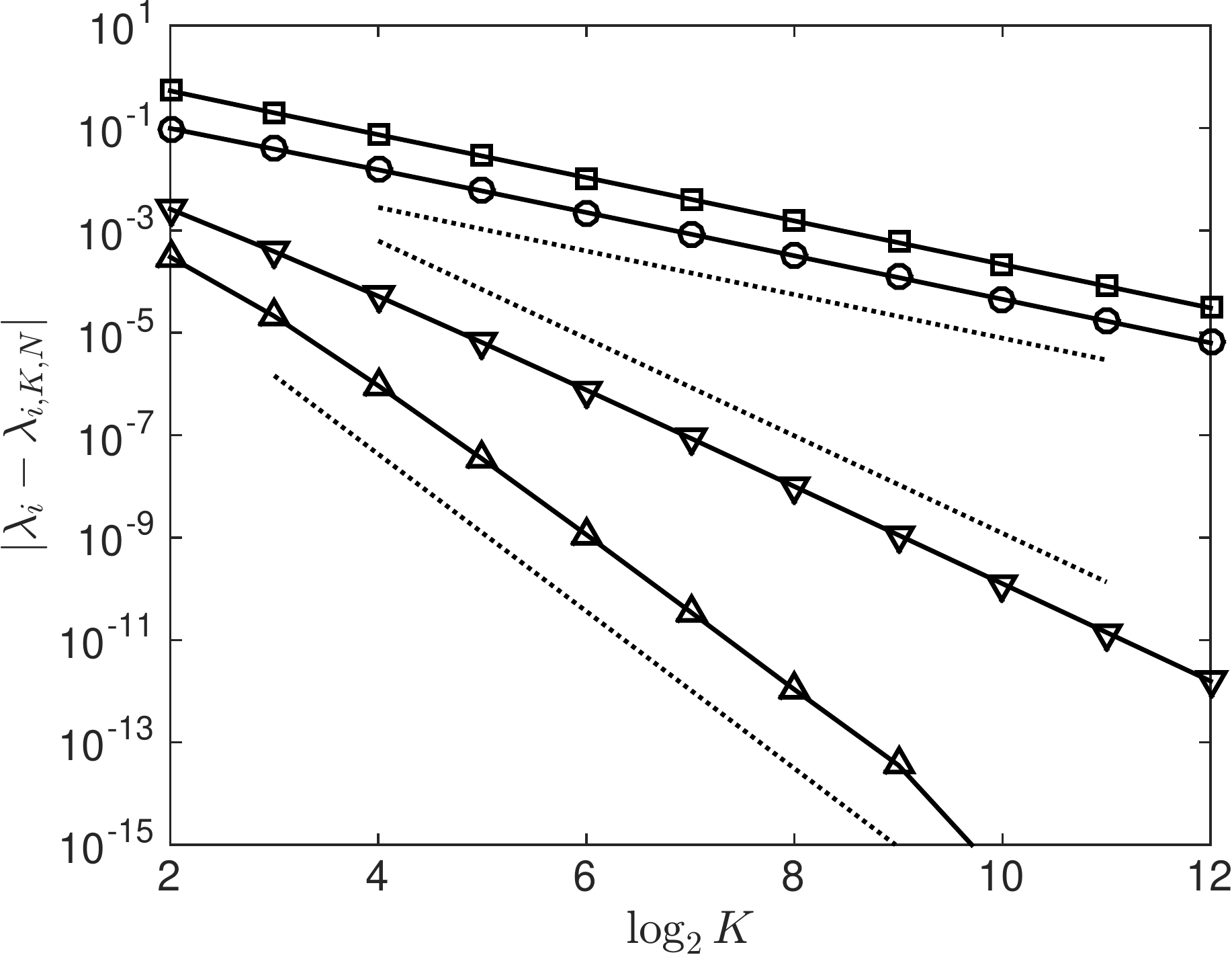}
\hspace*{\fill} (a). $c=1/2$. \hspace*{\fill}
\end{minipage}\hspace*{\fill}%
\begin{minipage}[h]{0.35\textwidth}
\includegraphics[width=1\textwidth]{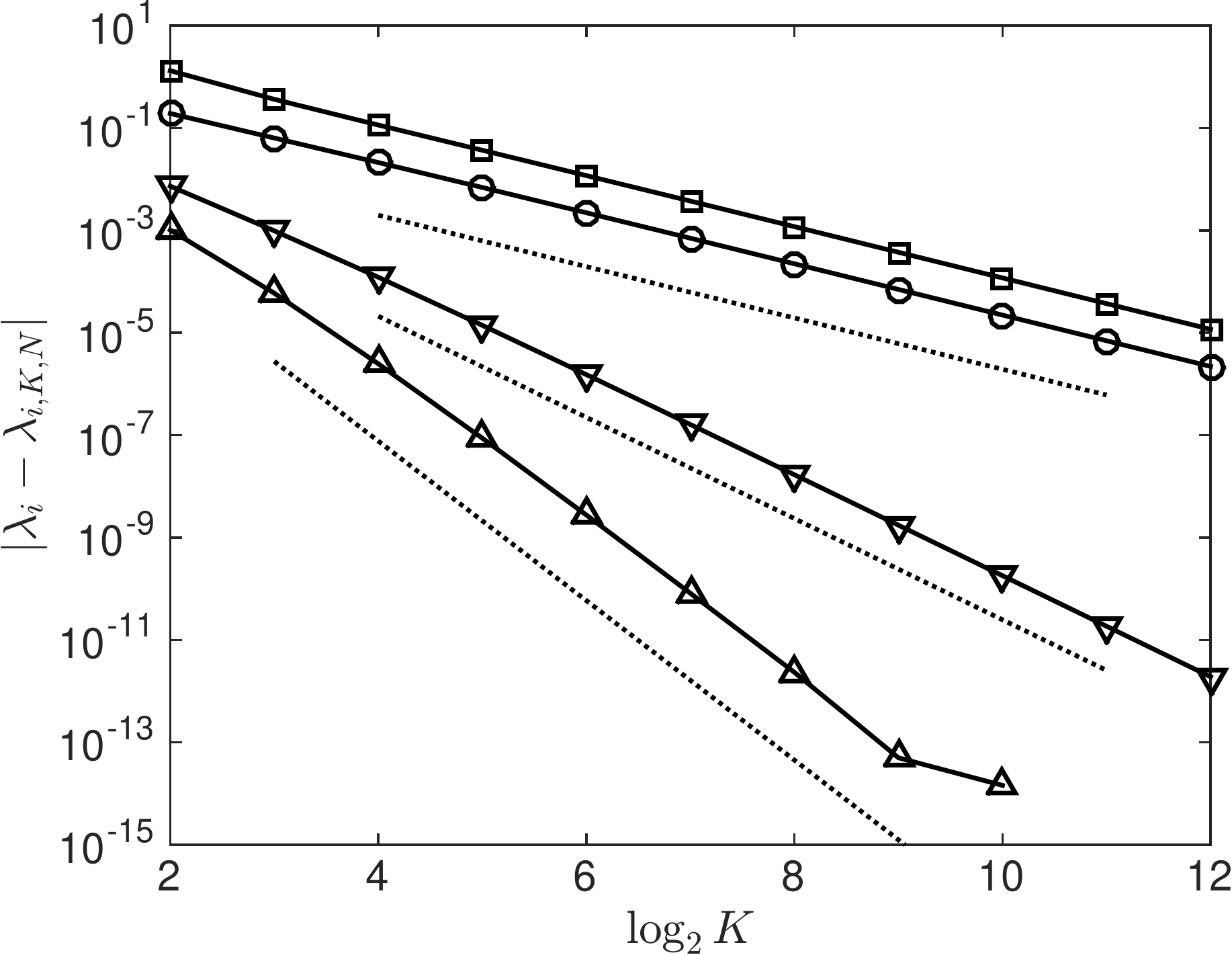}
\hspace*{\fill} (b). $c=2/3$.  \hspace*{\fill}
\end{minipage}\hspace*{\fill}
\caption{Approximation errors $|\lambda_i-\lambda_{i,K,N}|$  versus $K$ by the polynomial spectral method  on the unit ball.
 $\circ: \lambda_1$ ($n=0$),  $\triangledown:\lambda_2=\lambda_3=\lambda_4$  ($n=1$), $\vartriangle:\lambda_5=\dots=\lambda_9$  ($n=2$), $\square:  \lambda_{10} $  ($n=3$). Dashed lines:  $y = \sigma K^{-2\beta_n}$.}
\label{fig:errPSMball}
\end{figure}

On the contrary, exponential convergence rates of our novel spectral methods are readily observed
from Figures \ref{fig:SMdisk} and \ref{fig:SMball}. These results demonstrate  the effectiveness of
Method I and Method II. Interestingly, the convergence order of Method II  is  roughly
 twice as high as Method I. 

\begin{figure}[h!]
\hspace*{\fill}
\begin{minipage}[h]{0.35\textwidth}
\includegraphics[width=1\textwidth]{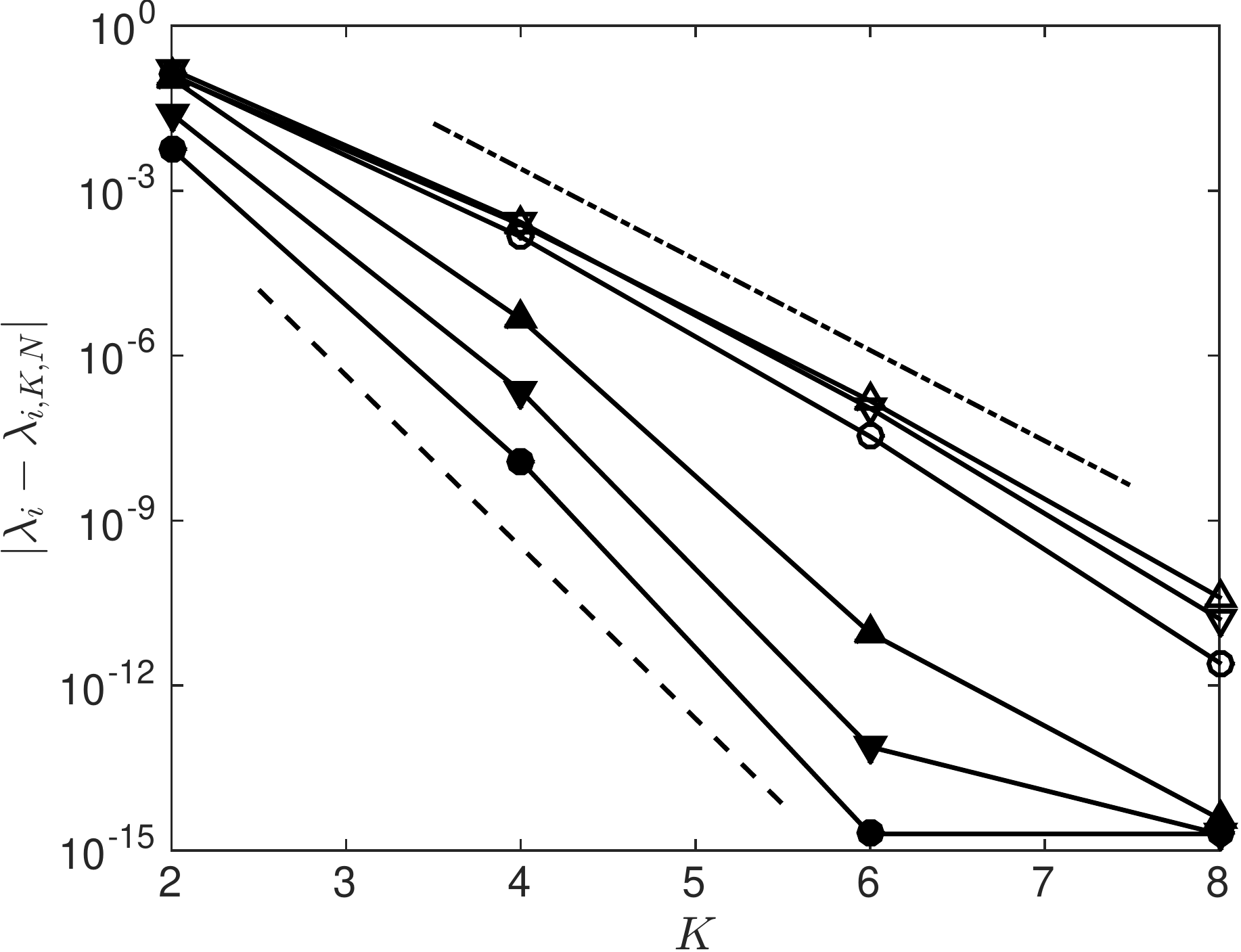}
\hspace*{\fill} (a). $c=1/2$. \hspace*{\fill}
\end{minipage}\hspace*{\fill}%
\begin{minipage}[h]{0.35\textwidth}
\includegraphics[width=1\textwidth]{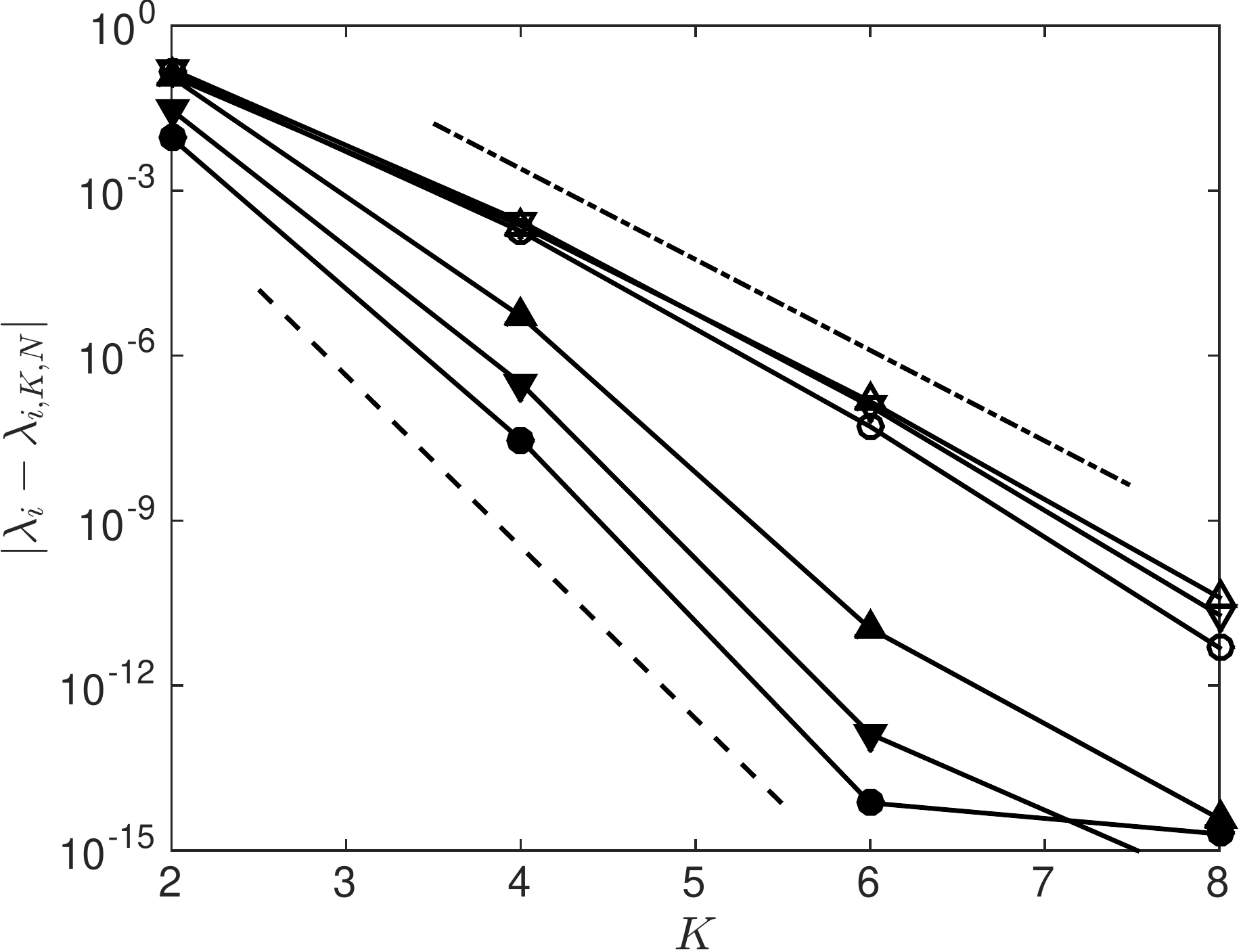}
\hspace*{\fill} (b). $c=2/3$.  \hspace*{\fill}
\end{minipage}\hspace*{\fill}
\caption{Approximation errors $|\lambda_i-\lambda_{i,K,N}|$  ($\circ: \lambda_1$,  $\triangledown:\lambda_2=\lambda_3$ and  $\vartriangle:\lambda_4=\lambda_5$) versus $K$ for
Method I (primitive markers) and Method II (filled markers) on the unit disk.
The dash-dot  and  dashed lines are the reference exponential $y = 10^{-1.65 K+4}$
and $y=10^{-3.12K + 3}$, respectively.}
\label{fig:SMdisk}
\end{figure}

\begin{figure}[h!]
\hspace*{\fill}
\begin{minipage}[h]{0.35\textwidth}
\includegraphics[width=1\textwidth]{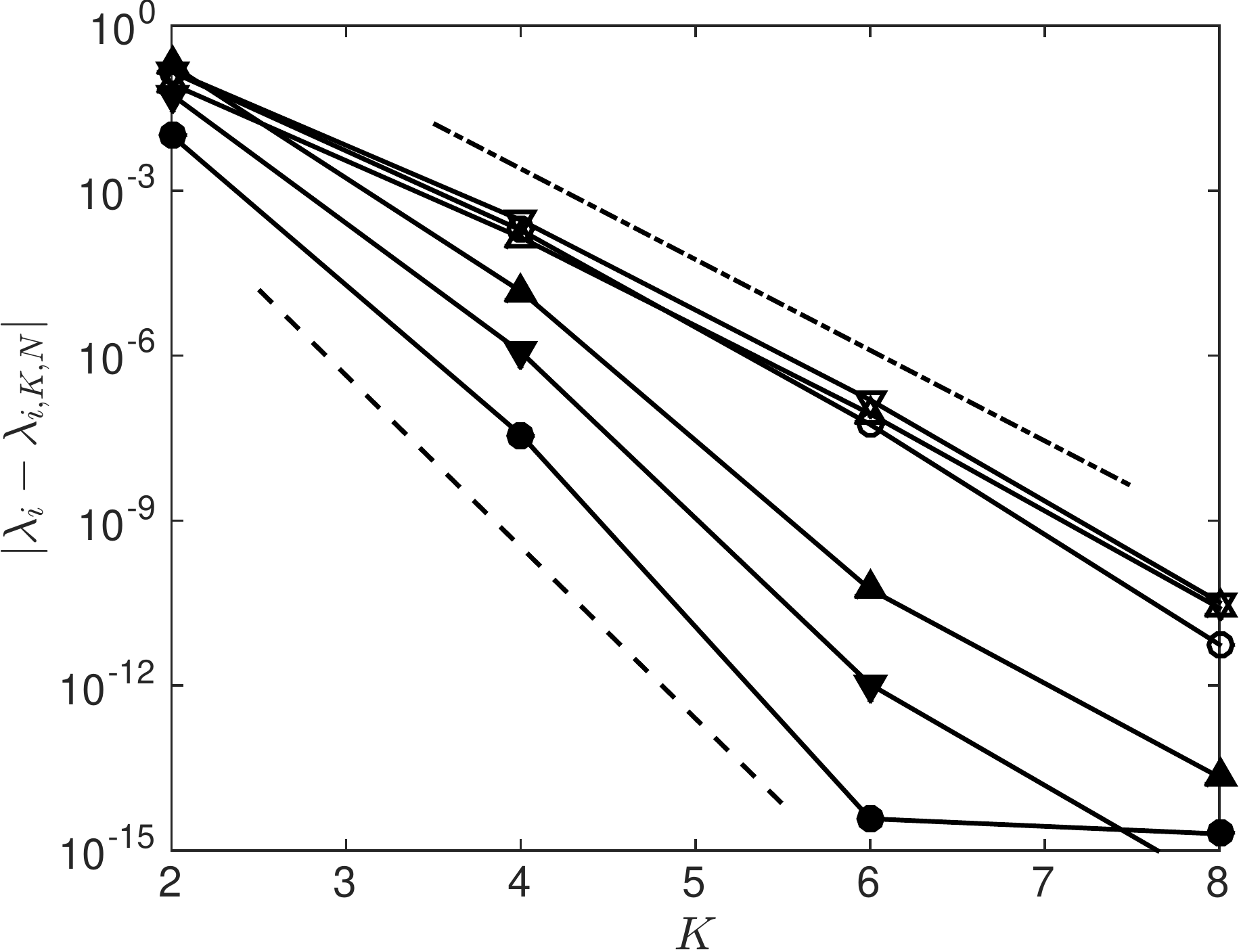}
\hspace*{\fill} (a). $c=1/2$. \hspace*{\fill}
\end{minipage}\hspace*{\fill}%
\begin{minipage}[h]{0.35\textwidth}
\includegraphics[width=1\textwidth]{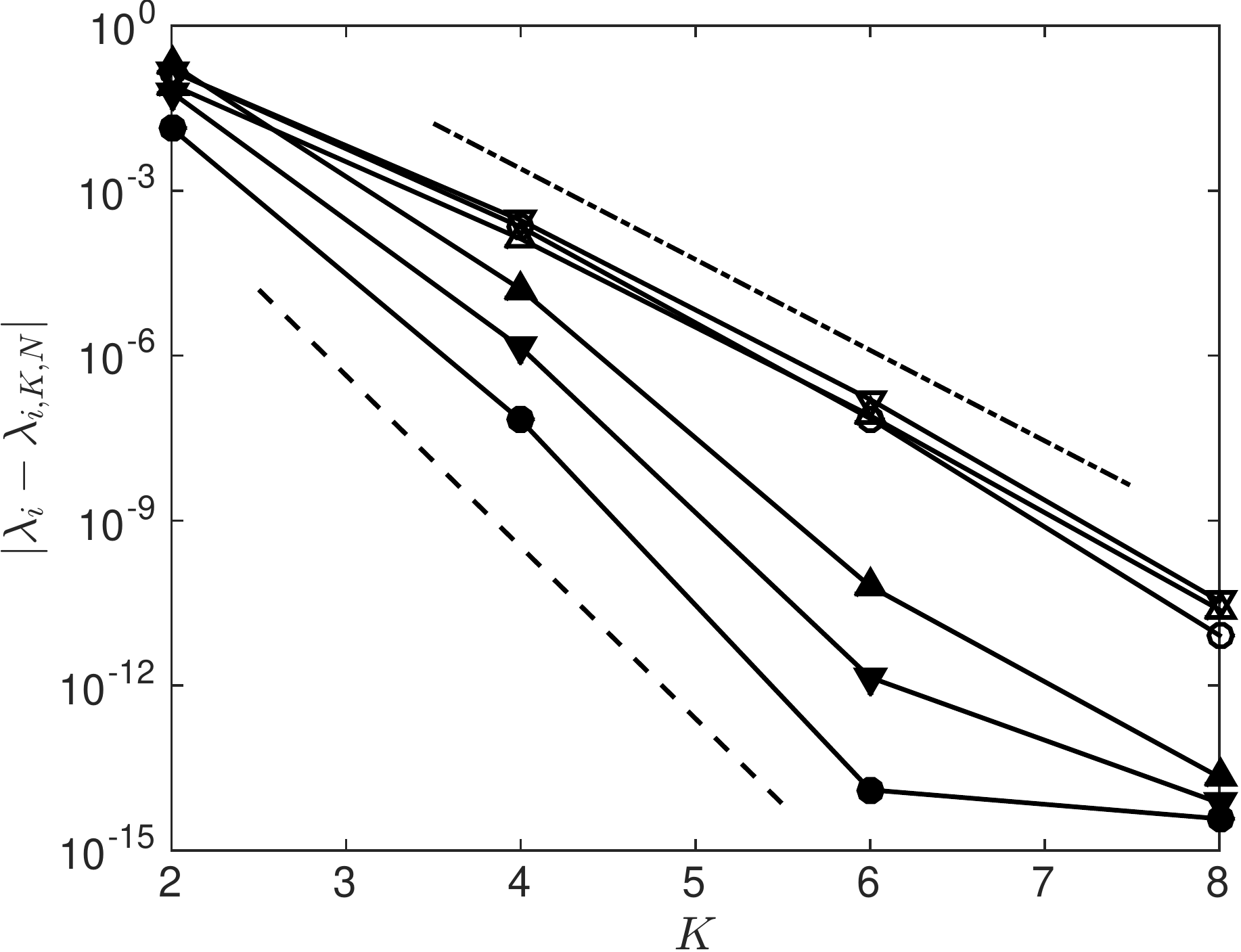}
\hspace*{\fill} (b). $c=2/3$.  \hspace*{\fill}
\end{minipage}\hspace*{\fill}
\caption{Approximation errors $|\lambda_i-\lambda_{i,K,N}|$  ($\circ: \lambda_1$,  $\triangledown:\lambda_2=\lambda_3=\lambda_4$ and  $\vartriangle:\lambda_5=\dots=\lambda_9$) versus $K$ on unit ball for Method I  (primitive markers) and Method II (filled markers).
The dash-dot  and  dashed lines are the reference exponential $y = 10^{-1.65 K+4}$
and $y=10^{-3.12K + 3}$, respectively.}
\label{fig:SMball}
\end{figure}

\subsection{Why and how do our methods work?}
\label{Why}
We first carry out a spectral analysis on the unit ball, where
the Schr\"{o}dinger equation \eqref{eigenP} can be reformulated, by using  \eqref{eq:Delta},  in
the spherical-polar coordinates as following,
\begin{align}
\label{eigenSP}
-\frac{1}{r^{d-1}}\partial_r\big(r^{d-1}\partial_ru) -\frac{1}{r^2}\Delta_0^2u +\frac{c^2}{r^2} u =\lambda u.
\end{align}
We now represent  the unknown eigenfunction $u$
as an expansion of spherical harmonic functions,
\begin{align*}
 u(x) = \sum_{n=0}^{\infty} \sum_{\ell=0}^{a_n^d} \wh u^n_{\ell}(r)  Y^n_{\ell}(\xi),
 \qquad x=r\xi, \, r =|x|,
\end{align*}
 and obtain an infinite system of second-order ordinary differential equations,
\begin{align}
\label{BeigenS}
-\frac{1}{r^{d-1}} \partial_r (r^{d-1}\partial_r) \wh u^{n}_{\ell}  + \frac{c^2+n(n+d-2)}{r^2} \wh u^{n}_{\ell}
= \lambda \wh u^{n}_{\ell}, \qquad 1\le \ell\le a_n^d,\,  n\ge 0.
\end{align}
Recall that $\beta_n= \beta(n,c,d)=\sqrt{ c^2+(n+d/2-1)^2}$, { the} system \eqref{BeigenS} is  then equivalent to
\begin{align*}
r^2\partial_r^2 [r^{d/2-1}\widehat u_{\ell}^{n}(r)] + r  \partial_r [r^{d/2-1}\widehat u_{\ell}^{n}(r)]  +\big[\lambda  r^2 -   \beta_n^2\big] r^{d/2-1}\widehat u_{\ell}^{n}(r)  =0,
\end{align*}
Making the variable transformation $\eta=\sqrt{\lambda}\, r$ and setting $\wh v^n_{\ell}(\eta)= r^{d/2-1} \wh u^n_{\ell}(r)$, one obtains
\begin{align*}
\eta^2\partial_{\eta}^2 \wh v_{\ell}^{n}(\eta) + \eta  \partial_{\eta} \wh v_{\ell}^{n}(\eta) +\big(\eta^2 -   \beta_n^2\big) \wh v_{\ell}^{n}(\eta)  =0,
\end{align*}
which is exactly the Sturm-Liouville equation for the first kind Bessel function,
hence  admits a unique solution  $\wh v^n_{\ell}(\eta)= J_{\beta_n}(\eta)$.
In return,
\begin{align}
\label{eigenF}
 \widehat u_{\ell}^{n}(r)= r^{1-d/2} \widehat v_{n}(\eta)= r^{1-d/2}  J_{\beta_n}(\sqrt{\lambda}\,r) =\sum_{m=0}^{\infty} \frac{(-1)^m}{m!\Gamma(m+\beta_n+1)} \Big(\frac{\sqrt{\lambda}\,{r}}{2}\Big)^{2m+\beta_n} r^{1-d/2}.
\end{align}

Since the homogeneous Dirichlet boundary condition in  \eqref{eq:Delta} implies
$\wh u_{\ell}^n(1)=0$,  one readily finds that
the eigenvalue $\lambda$ of \eqref{BeigenS} satisfies
\begin{align*}
J_{\beta_n}(\sqrt{\lambda})  = 0.
\end{align*}

Let us now shed light on the mechanism of our methods. The terms on the left-hand side of \eqref{BeigenS} can be merged into one, i.e.,
\begin{align*}
-\frac{1}{r^{d-1}}\partial_r\big(r^{d-1}\partial_r\wh u^n_{\ell}) +\frac{c^2+n (n+d-2)}{r^2}\wh u^n_{\ell}
=&\, -\partial_r^2 \wh u^n_{\ell} - \frac{d-1}{r} \partial_r \wh u^n_{\ell}  +\frac{c^2+n (n+d-2)}{r^2}\wh u^n_{\ell}
\\
= -\partial_r^2 \wh u^n_{\ell} - \frac{2\nu+\mu}{r} \partial_r \wh u^n_{\ell} -\frac{\nu(\nu+\mu-1)}{r^2}  \wh u^n_{\ell}
=&\,  -\frac{1}{r^{\mu+\nu}} \partial_r \Big[ r^{\mu} \partial_r \big( r^{\nu} \wh u^n_{\ell} \big) \Big],
\end{align*}
where  $\mu$ and $\nu$ are parameters to be determined by
\begin{align*}
2\nu+\mu=d-1, \qquad
-\nu(\nu+\mu-1) = c^2+n (n+d-2),
\end{align*}
or explicitly,
\begin{align*}
 \nu = d/2-1 \pm \beta_n, \qquad
 \mu =1\mp 2\beta_n.
\end{align*}
In particular,  taking
\begin{align*}
-\frac{1}{r^{d-1}}\partial_r\big(r^{d-1}\partial_r\wh u^n_{\ell}) +\frac{c^2+n (n+d-2)}{r^2}\wh u^n_{\ell}
= -r^{-\beta_n-d/2}\partial_r \Big[ r^{2\beta_n+1} \partial_r \big( r^{d/2-1-\beta_n} \wh u^n_{\ell} \big) \Big],
\end{align*}
{one eliminates}  the singularity of  the eigenfunction $ \wh u^n_{\ell}$
defined in \eqref{eigenF}  by multiplying $r^{d/2-1-\beta}$ such that
the analytic function
$r^{d/2-1-\beta} \wh u^n_{\ell}(r)$ can be well approximated by {the} Jacobi polynomials
in  $r$ on $[0,1]$
with an exponential rate of convergence. This {provides an explanation for the} effectiveness of Method I.

As for Method II, we note that, under the modified polar-spherical  coordinates $x=\sqrt{\rho}\,\xi$ with $\rho=r^2$,
\begin{align}
-\frac{1}{r^{d-1}}\partial_r\big(r^{d-1}\partial_r\wh u^n_{\ell}) +\frac{c^2+n (n+d-2)}{r^2}\wh u^n_{\ell}
= -4 r^{1-d/2-\beta_n}\partial_\rho \Big[ \rho^{\beta_n+1} \partial_\rho \big( r^{d/2-1-\beta_n} \wh u^n_{\ell} \big) \Big].
\label{Drho}
\end{align}
The right-side hand of \eqref{Drho} is self-adjoint with respect to the measure $r^{d-1}dr= \frac12 r^{d-2}d\rho$. {
According} to  \eqref{eigenF}, $r^{d/2-1-\beta_n} \wh u^n_{\ell} (r)$  can be  approximated by the
Jacobi polynomials in $\rho=r^2$  on $[0,1]$
with an exponential rate of convergence. This {explains} the effectiveness of Method II.

\def\CL{\mathcal{L}}

\section{Novel spectral method on a planar  sector}

In this section, we study two novel spectral methods
for the Schr\"{o}dinger equation on a planar circular  sector,
which is enclosed by the arc $\Gamma_2$ and two radii $\Gamma_1$ and $\Gamma_3$ (see the left of Figure \ref{Sector}):
\begin{align}
\label{sector}
\Omega = \Lambda:=& \{ (r,\theta): 0\le r<1, 0< \theta< \gamma^{-1} \pi   \} ,
\end{align}
where $\gamma \ge \frac12$, and $(r,\theta)$ is the polar coordinates satisfying
$x=(r\cos\theta, r\sin \theta) $.
\begin{figure}[h!]
\hfill\includegraphics[width=0.22\textwidth]{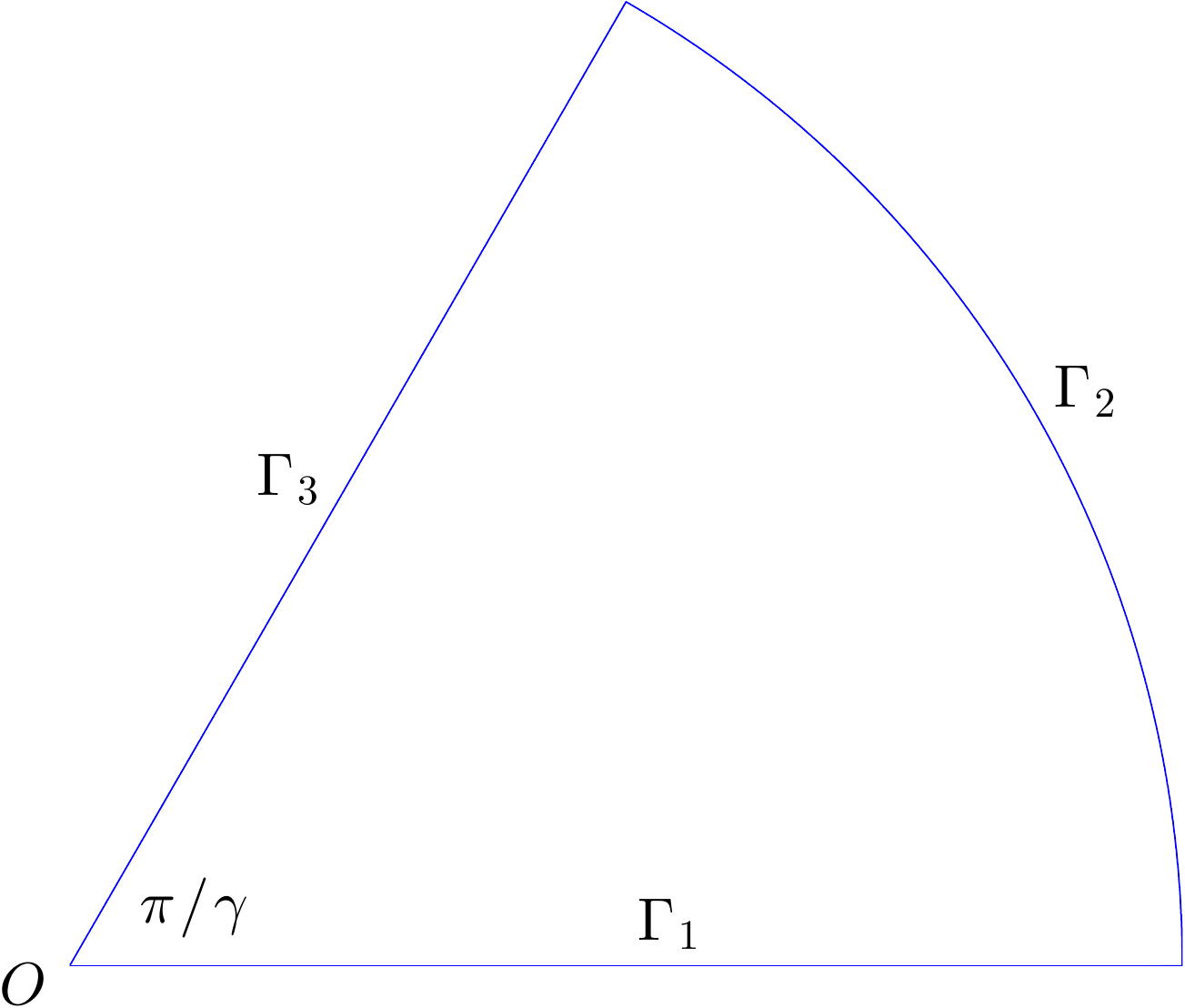}\hspace*{\fill}
\hfill\includegraphics[width=0.22\textwidth]{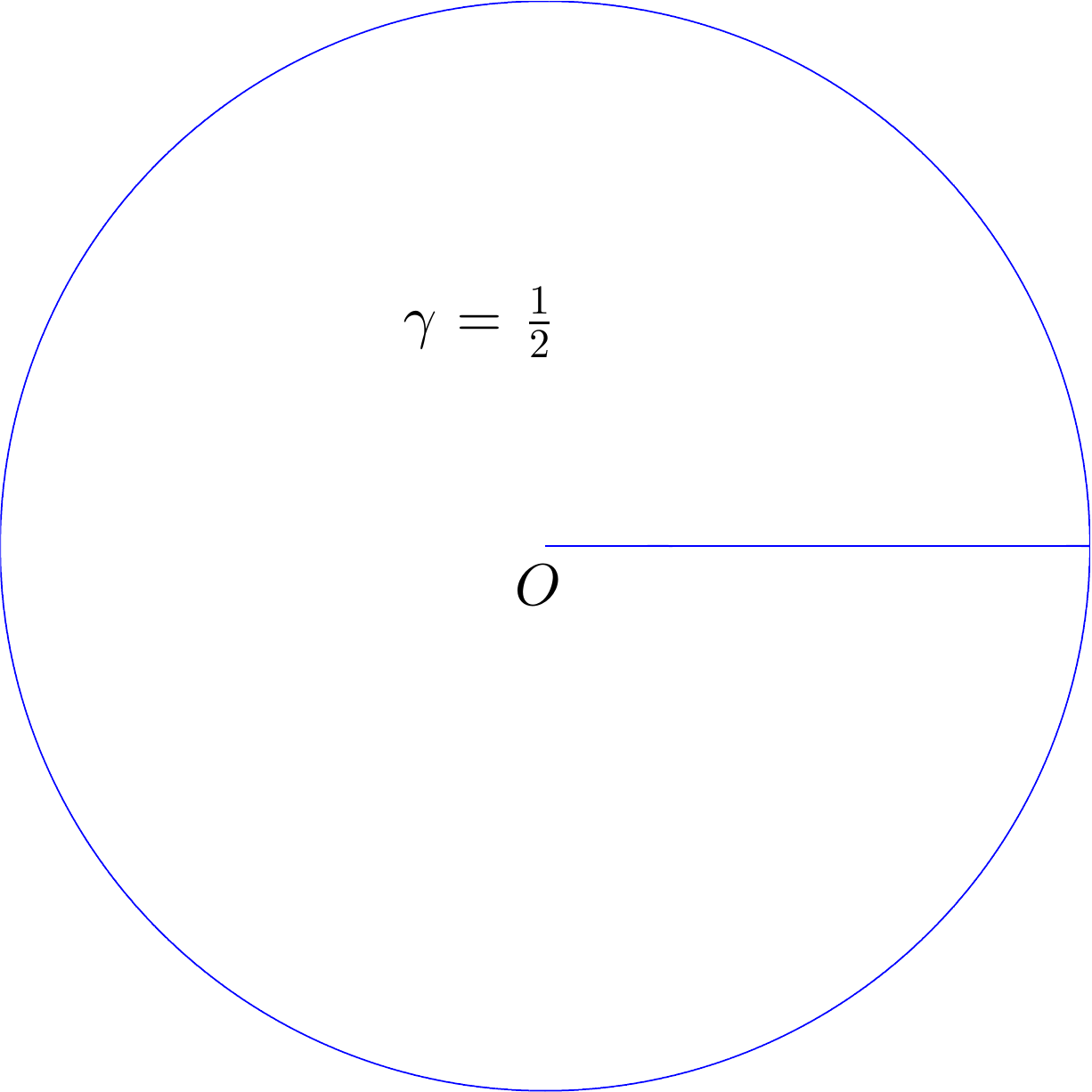}\hspace*{\fill}
\hfill\includegraphics[width=0.22\textwidth]{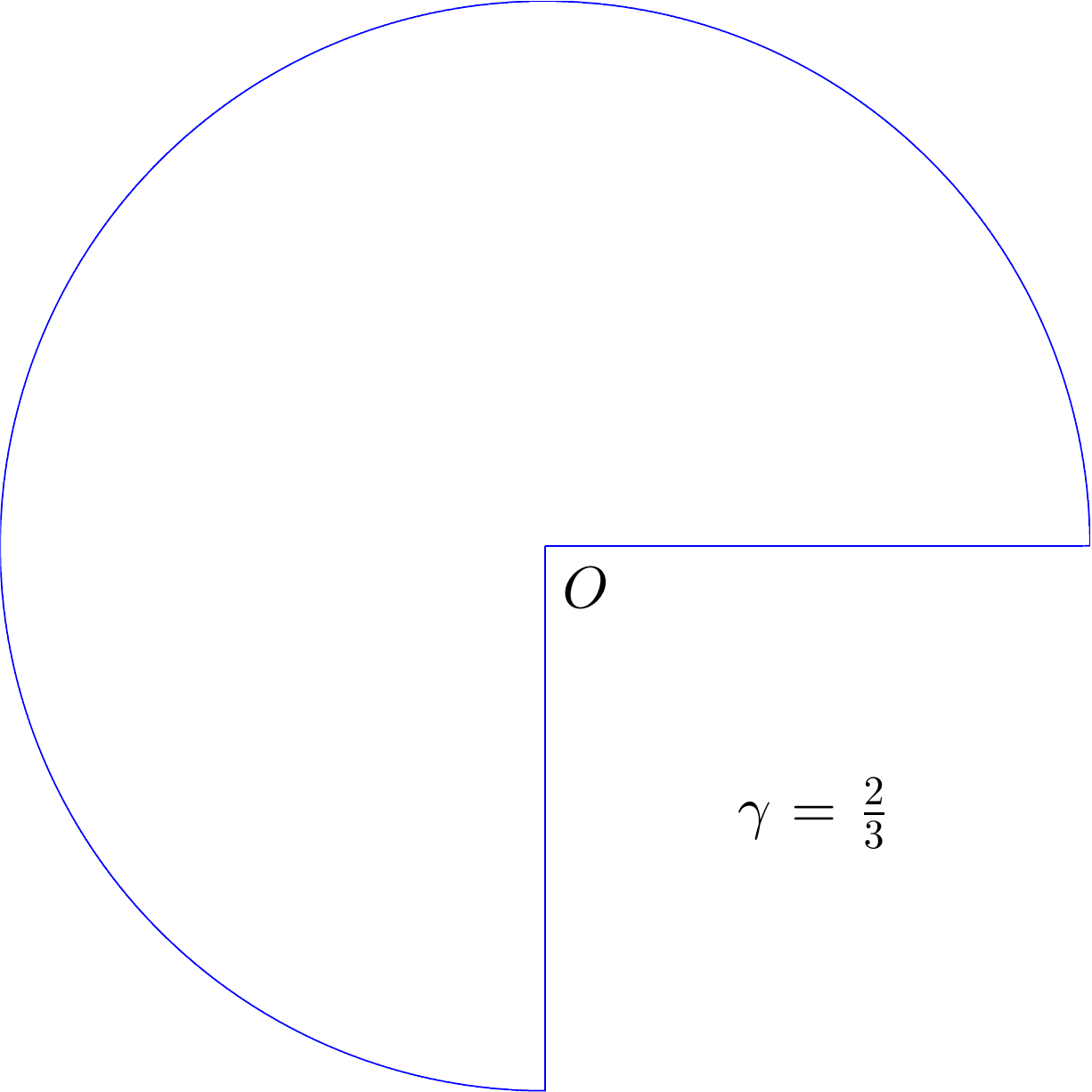}\hspace*{\fill}
\caption{Planar circular   sectors.}
\label{Sector}
\end{figure}

A heuristic idea is to expand the unknown eigenfunction $u$ by sine series 
\begin{align*}
  u(x) = \sum_{n=1}^{\infty} \wh u_n(r) \sin(n \gamma  \theta)
\end{align*}
in \eqref{eigenP} to  obtain
\begin{align*}
\big(-\Delta& + \frac{c^2}{r^2} \big)u(x) =\sum_{n=1}^{\infty}
\left[ -\frac{1}{r}\partial_r\big(r\partial_r\wh u_n(r)) +\frac{\beta^2_n}{r^2}\wh u_n(r) \right]
\sin(n \gamma\theta)
\\
=&-\sum_{n=1}^{\infty} r^{-\beta_n-1}\partial_r \Big[ r^{2\beta_n+1} \partial_r \big( r^{-\beta_n} \wh u_n(r) \big) \Big] \sin(n \gamma\theta),
\end{align*}
where $\beta_n:=\beta(c,n) =\sqrt{c^2+\gamma^2n^2} $.
This leads to the following equivalent eigen equations,
\begin{align*}
- r^{-\beta_n-1}\partial_r \Big[ r^{2\beta_n+1} \partial_r \big( r^{-\beta_n} \wh u_n(r) \big) \Big]
= \lambda \wh u_n(r), \qquad n\ge 0,
\end{align*}
which motivates us to propose two types of spectral methods for \eqref{eigenP} on a planar sector.

\subsection{Spectral method I}

{We} propose an approximation scheme on a circular sector  in analogue to Method I in the {previous} section.
Let us first introduce the Sobolev space
$${}_0W^1(\Lambda)
= \left\{ u\in W^1(\Lambda): u=0  \text{ on }  \Gamma_1 \cup \Gamma_3   \right\}.$$
\begin{lemma}
\label{SOSF}
Define
\begin{align*}
\Phi_{k}^n(x):=\dfrac{2k+2\beta_n}{k+2\beta_n}J_k^{-1,2\beta_n}(2r-1)  r^{\beta_n} \sin (n \gamma \theta ),
\quad  k\in\NN_0, n\in \NN.
\end{align*}
Then  $\Phi_{k}^n(x), \, k\in \NN_0,\, n\in \NN$ form a Sobolev orthogonal basis
in ${}_0W^1(\Lambda)$ in the following sense,
\begin{align*}
 (\nabla  \Phi_{k}^n, \nabla  \Phi_{j}^m)_{\Lambda} + c^2 ( \Phi_{k}^n, \Phi_{j}^m)_{r^{-2},\Lambda}
  = \frac{\pi}{2\gamma}  \delta_{m,n}\delta_{j,k}  (k+\beta_n) (2-\delta_{k,0}) .
\end{align*}
Moreover,
\begin{align*}
 (\Phi_{k}^n, \Phi_{j}^m)_{\Lambda} =\, & \frac{\pi}{2\gamma}\delta_{n,m}\times
         \begin{cases}
     \frac {( k+\beta_n )  ( k^2 +2 k\beta_n + 4 \beta_n^{2}-1) }{ ( k+\beta_n -1)  ( k+\beta_n+1) ( 2 k+2 \beta_n-1
 )  ( 2k+2 \beta_n+1) },
  & k=j\ge 1, \\[0.3em]
  \frac{1}{2(\beta_n+1)}, & k=j=0,\\[0.3em]
  - \frac { ( 2\beta_n-1) ( 2 \beta_n+1  ) }{   ( 2 k+2 \beta_n-1  )  ( 2 k+2 \beta_n+1  ) ( 2k+2\beta_n+3  )   },
     & j = k+1,\\[0.3em]
  - {\frac {    ( k+1  )  ( k+2 \beta_n+1  ) }{2 ( k+\beta_n+1  )   ( 2 k+2 \beta_n+1  )    ( 2 k+2 \beta_n+3  ) }},
   & j = k+2,\\[0.3em]
  - \frac { ( 2\beta_n-1) ( 2 \beta_n+1  ) }{   ( 2 j+2 \beta_n-1  )  ( 2 j+2 \beta_n+1  ) ( 2j+2\beta_n+3  )   },
     & k = j+1,\\[0.3em]
  - {\frac {    ( j+1  )  ( j+2 \beta_n+1  ) }{2 ( j+\beta_n+1  )   ( 2 j+2 \beta_n+1  )    ( 2 j+2 \beta_n+3  ) }},
   & k = j+2,\\[0.3em]
   0, & \text{otherwise}.
     \end{cases}
\end{align*}

\end{lemma}
Lemma \ref{SOSF} can be proved similarly as Lemma \ref{SOBF}, we omit the details.

Define the  approximation space
\begin{align*}
W_{K,N}^{\Lambda}=
\left\{ \Phi_{k}^n:  1 \le k\le K, \, 0\le  n \le N \right\}  \subset W^1_0(\Lambda).
\end{align*}
Then the spectral-Galerkin  approximation scheme is, to find $u_{K,N}\in W_{K,N}^{\Lambda}$ such that
\begin{align*}
a(u_{K,N},v)= (\nabla u_{K,N}, \nabla v)_{\Lambda} + c^2 (u_{K,N},v)_{r^{-2},{\Lambda}}
  = \lambda_{K,N} (u_{K,N},v)_{\Lambda}, \qquad v\in W_{K,N}^{\Lambda}.
\end{align*}
It is worthy to note that this approximation scheme leads to an
algebraic eigen system
with a diagonal stiffness matrix  and a penta-diagonal mass matrix,
which can be  easily decoupled  and  solved in parallel.

\subsection{Spectral method II}
Our second method for \eqref{eigenP} on a circular sector is based on
the following lemma, which is an analogue to Lemma \ref{SOBP}.

\begin{lemma}  Define
\label{SOSP} \begin{align*}
\Psi_{k}^n(x):=\dfrac{2k+\beta_n}{k+\beta_n}J_k^{-1,\beta_n}(2r^2-1)  r^{\beta_n} \sin (n \gamma \theta ),
\quad  k\in \NN_0,\, n\in \NN.
\end{align*}
Then $\Psi_{k}^n(x), \, k\in \NN_0,\, n\in \NN$ form a Sobolev orthogonal basis
in ${}_0W^1(\Lambda)$ in the following sense,
\begin{align*}
      (\nabla& \Psi_{k}^{n}, \nabla \Psi_{j}^{m})_{\Lambda} + c^2 (\Psi_{k}^{n},\Psi_{j}^{m})_{r^{-2},{\Lambda}}    =\frac{\pi}{2\gamma} \delta_{m,n}  \delta_{k,j} (2k+\beta_n) (2-\delta_{k,0})  .
\end{align*}
Moreover,
\begin{align*}
     (\Psi_{k}^{n},\Psi_{j}^{m})_{\Lambda} :=\, & \frac{\pi}{2\gamma}\delta_{n,m}\times
      \begin{cases}
 \frac {1}{ 2 k+\beta_n+1} +\frac{1-\delta_{k,0}}{ 2 k+\beta_n-1},
  & k=j, \\
  - \frac{1}{ 2(2k+\beta_n+1) },
     & j = k+1,\\
  - \frac{1}{2( 2j+\beta_n+1)},
     & k = j+1,\\
   0, & \text{otherwise}.
     \end{cases}
\end{align*}
\end{lemma}

We omit the proof, which is similar to that of Lemma \ref{SOBP} as in the Appendix \ref{AppA}.

Define the  approximation space
\begin{align*}
V_{K,N}^{\Lambda}=
\left\{ \Psi_{k}^n:  1 \le k\le K, \, 0\le  n \le N \right\}  \subset W^1_0(\Lambda).
\end{align*}
Then the spectral-Galerkin  approximation scheme is, to find $u_{K,N}\in V_{K,N}^{\Lambda}$ such that
\begin{align*}
 a(u_{K,N},v)=(\nabla u_{K,N}, \nabla v)_{\Lambda} + c^2 (u_{K,N},v)_{r^{-2},{\Lambda}}
  = \lambda_{K,N} (u_{K,N},v)_{\Lambda}, \qquad v\in V_{K,N}^{\Lambda}.
\end{align*}
This approximation scheme leads to an
algebraic eigen system
with a diagonal stiffness matrix  and a tridiagonal mass matrix,
which can also be   decoupled easily and  solved in parallel.

\subsection{Numerical experiments}
\label{SectNE}
Our novel spectral methods for \eqref{eigenP}  are examined   on the slit disk ($\gamma=\frac12$)
 and the circular sector with $\gamma=\frac{2}{3}$. The approximation errors of  the 3 smallest eigenvalues are reported in Figure \ref{fig:SMsect} in semi-log scale.
The exponential convergence is observed for both Method I and Method II. Furthermore, a comparison of Figure \ref{fig:SMsect} with Figures \ref{fig:SMdisk}-\ref{fig:SMball} reveals that 
both  Method I and Method II converge at a fixed order of their own regardless of $\Omega$ being a  ball or a circular sector.
Finally,  since the radial component of the eigenfunction
 $r^{-\beta_n} \wh u_n(r)$ is analytic in $\rho=r^2$,   Method II  nearly converges
 twice as fast as Method I, just as shown in Figure \ref{fig:SMsect}.

\begin{figure}[h!]
\hspace*{\fill}
\begin{minipage}[h]{0.35\textwidth}
\includegraphics[width=1\textwidth]{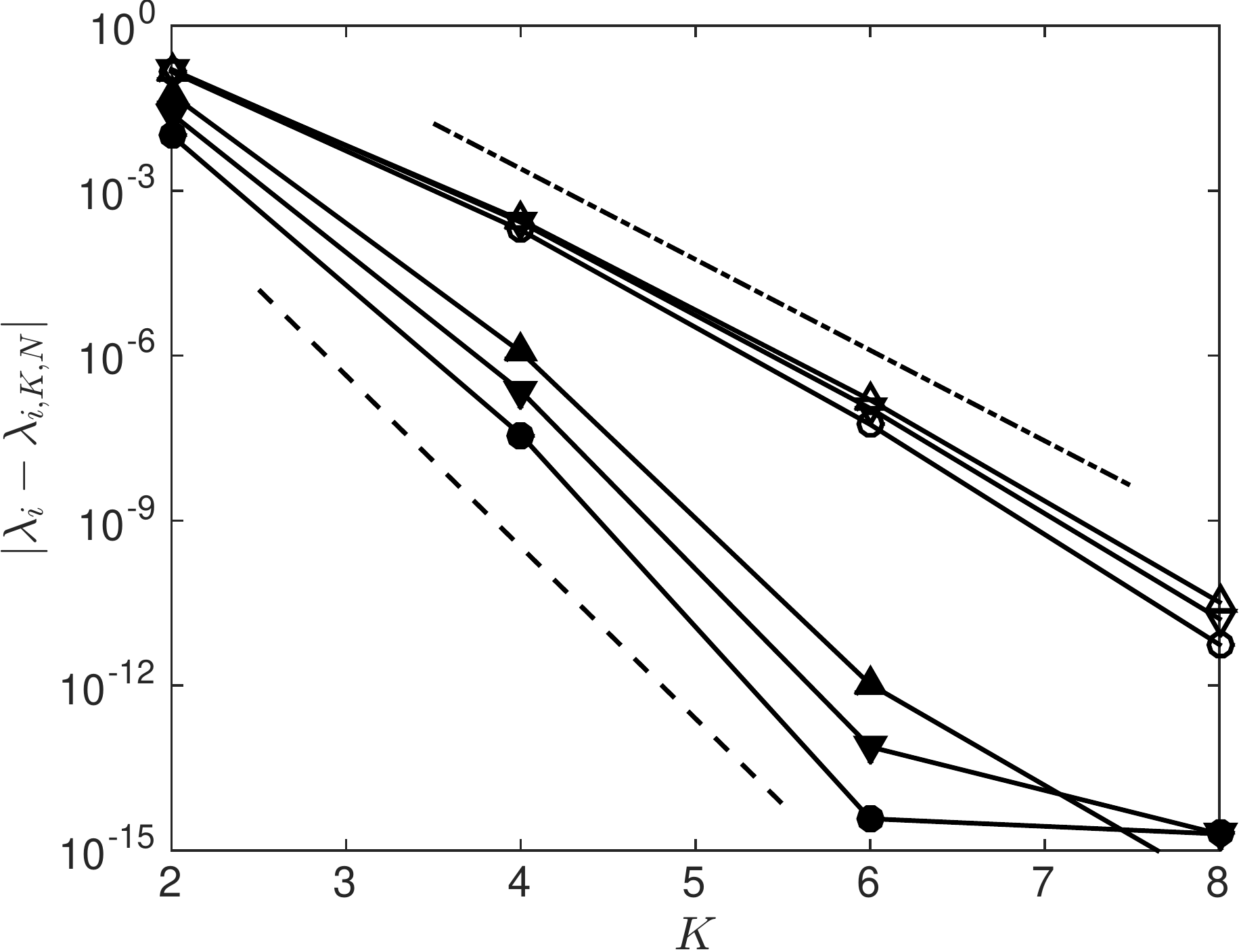}
\hspace*{\fill} (a). $\gamma=1/2,\,c=1/2$. \hspace*{\fill}
\end{minipage}\hspace*{\fill}%
\begin{minipage}[h]{0.35\textwidth}
\includegraphics[width=1\textwidth]{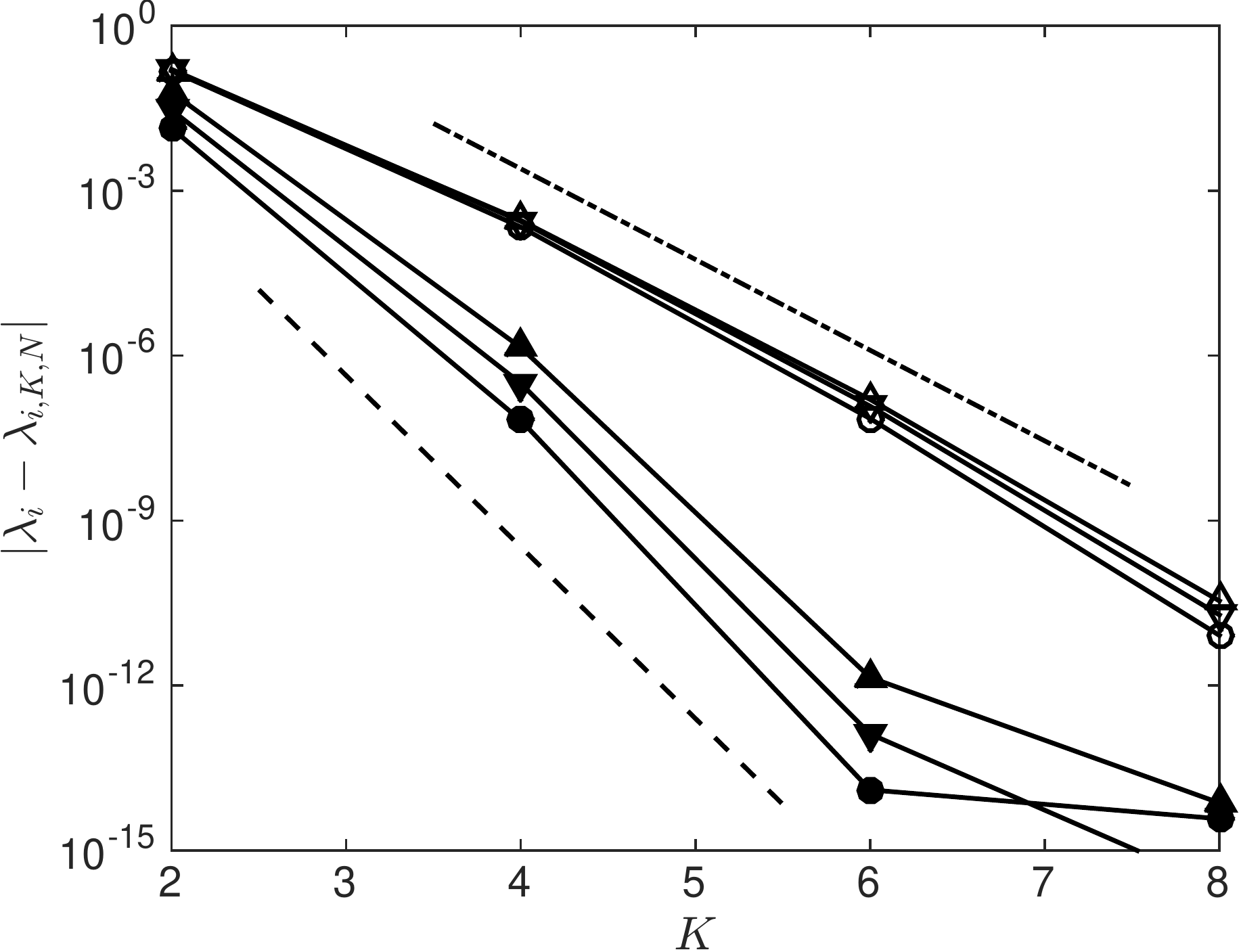}
\hspace*{\fill} (b). $\gamma=1/2,\,c=2/3$.  \hspace*{\fill}
\end{minipage}\hspace*{\fill}

\vspace*{0.5em}
\hspace*{\fill}
\begin{minipage}[h]{0.35\textwidth}
\includegraphics[width=1\textwidth]{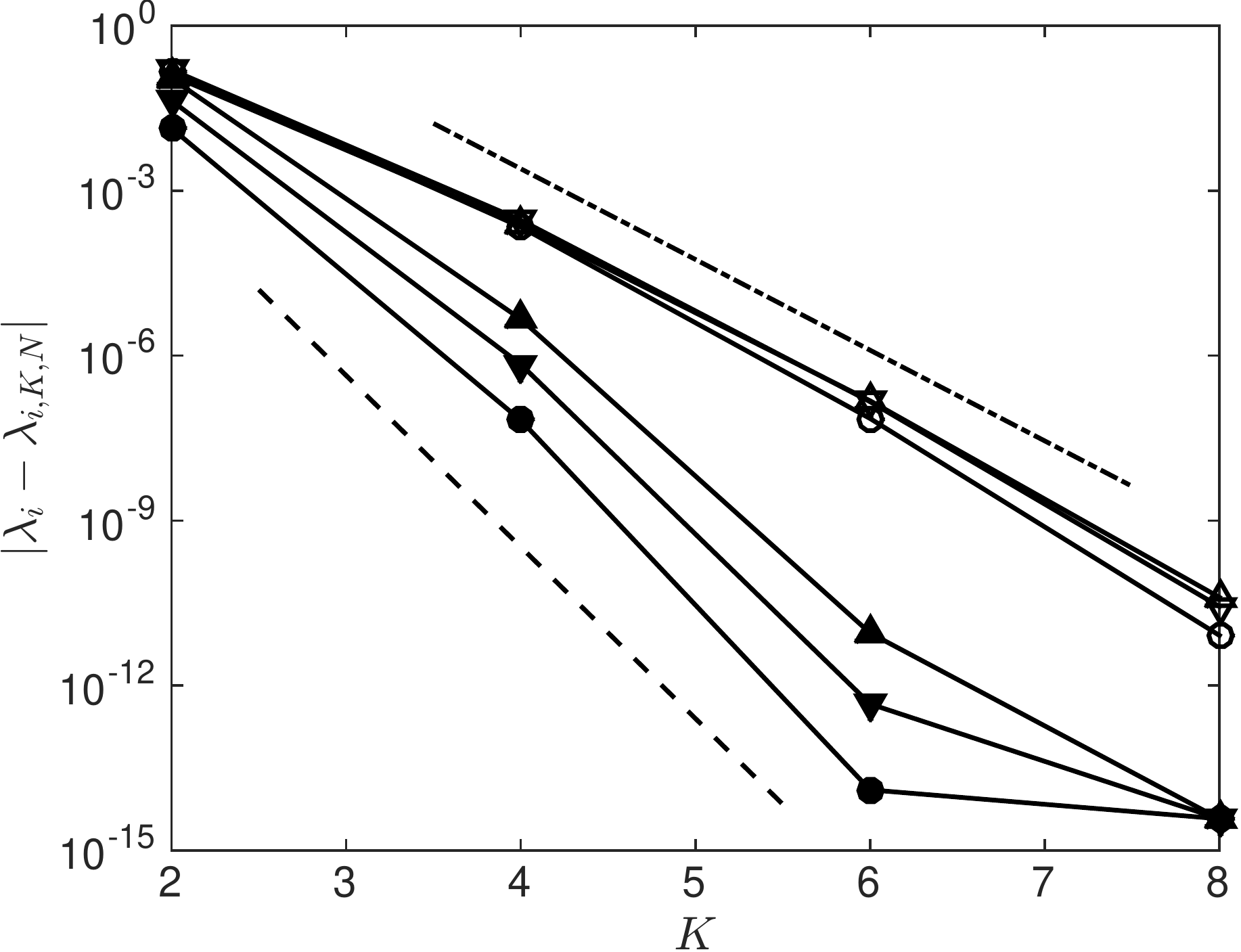}
\hspace*{\fill} (c). $\gamma=2/3,\,c=1/2$. \hspace*{\fill}
\end{minipage}\hspace*{\fill}%
\begin{minipage}[h]{0.35\textwidth}
\includegraphics[width=1\textwidth]{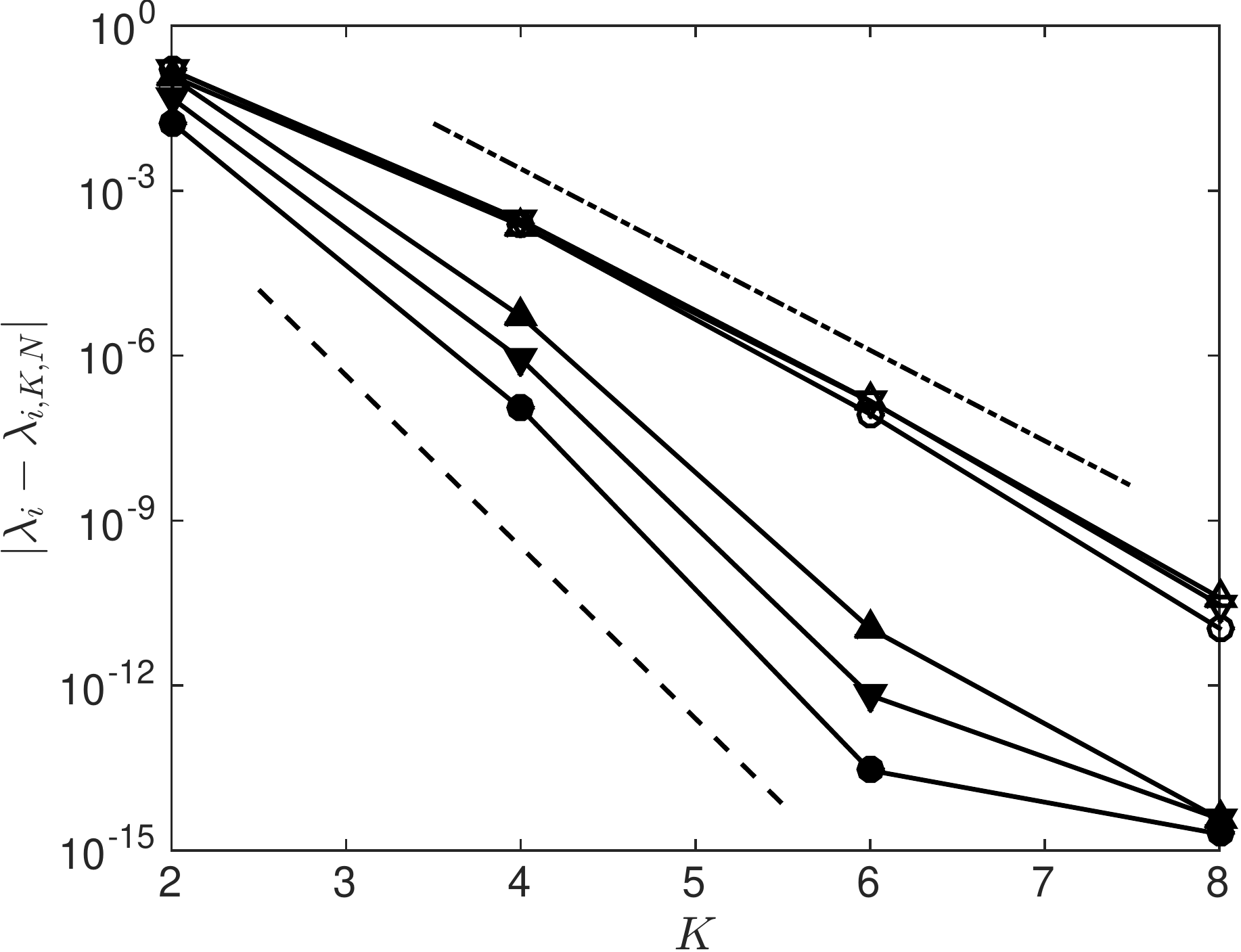}
\hspace*{\fill} (d). $\gamma=2/3,\,c=2/3$.  \hspace*{\fill}
\end{minipage}\hspace*{\fill}

\caption{Approximation errors $|\lambda_i-\lambda_{i,K,N}|$  ($\circ: \lambda_1$,  $\triangledown:\lambda_2$ and  $\vartriangle:\lambda_3$) versus $K$ on the circular sectors.
Method I (primitive markers) and Method II (filled markers).
The dash-dot  and  dashed lines are the reference exponential $y = 10^{-1.65 K+4}$
and $y=10^{-3.12 K + 3}$, respectively.}
\label{fig:SMsect}
\end{figure}

\section{Mortar spectral element methods}

The mortar element method uses nonconforming domain decomposition technique,
which allows to choose independently the discretization method on each sub-domain
 to adapt to the local behavior of the partial differential equation
\cite{BMR,CHQZ2007}.
For simplicity, we consider only mortar spectral element methods on  planar domains.

\subsection{Mortar spectral elements on a regular domain}
Let $\Omega$ be a bounded domain  $\RR^2$ such that
$(0,0)\in \Omega$ and  $(0,0) \notin \partial \Omega$.
We first use the
circle centered at the original with the radius  $R>0$,
$$\Gamma_{R} = \left\{ (x,y) = (R\cos\theta, R\sin\theta):  0\le \theta\le 2\pi  \right\}\subset \Omega\setminus\partial\Omega, $$
 to decompose  $\Omega$ into two subdomains
$$\Omega_0=\{(x,y)\in \RR^2:  x^2+y^2<R^2 \},
\quad
\Omega_1=\Omega\setminus\Omega_0.$$  For discretization, our main idea is to use a novel spectral
method on the disk $\Omega_0$  while use  the standard spectral element method on $\Omega_1$.

Let us take $\Omega=[-1,1]^2$ as an example to explain our idea.
For simplicity, we further  decompose $\Omega_1$ into four curvilinear quadrilaterals by the two diagonals of the square  as indicated in the left side of Figure \ref{fig:NCM},
\begin{align*}
&\Omega_1^{(1)} = \{(x,y)\in \RR^2:  x^2+y^2>R^2,\,  |y|< x < 1 \},
\\
&\Omega_1^{(2)} = \{(x,y)\in \RR^2:  x^2+y^2>R^2,\,  |x|< y < 1 \},
\\
&\Omega_1^{(3)} = \{(x,y)\in \RR^2:  x^2+y^2>R^2,\,  |y|< -x < 1 \},
\\
&\Omega_1^{(4)} = \{(x,y)\in \RR^2:  x^2+y^2>R^2,\,  |x|< -y < 1 \}.
\end{align*}
In such a way,  we have the non-overlapping partition $\Omega=\Omega_0\cup  \Omega_1^{(1)} \cup \Omega_1^{(2)}  \cup
\Omega_1^{(3)}\cup \Omega_1^{(4)}$, and  the interior edges of each $\Omega_1^{(\kappa)}$ ($1\le \kappa\le 4$)  are then perpendicular to the circle.

\begin{figure}[h!]
\hspace*{\fill}\begin{minipage}[h]{0.35\textwidth}
\includegraphics[width=1\textwidth]{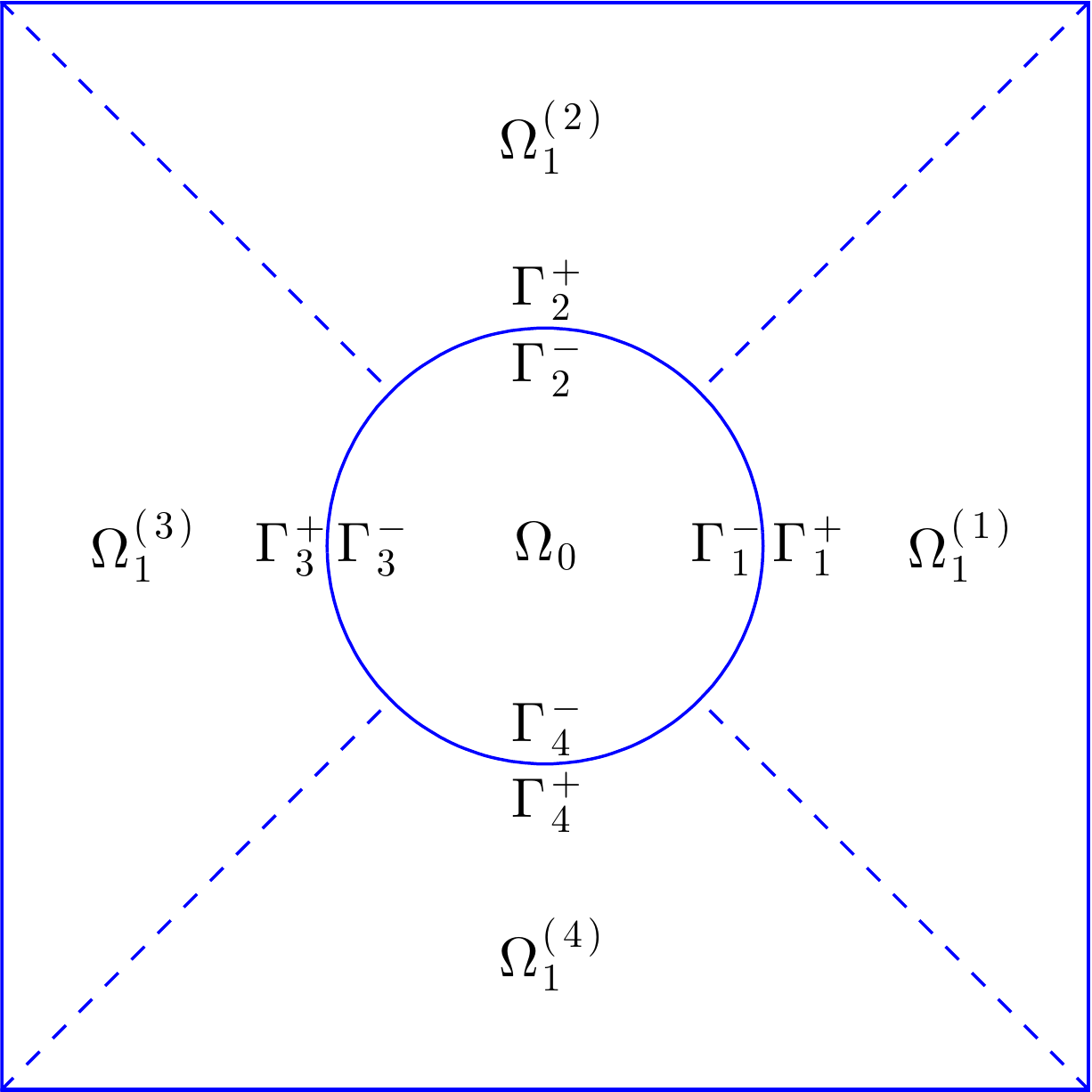}
\end{minipage}\hspace*{\fill}%
\hspace*{\fill}\begin{minipage}[h]{0.35\textwidth}
\includegraphics[width=1\textwidth]{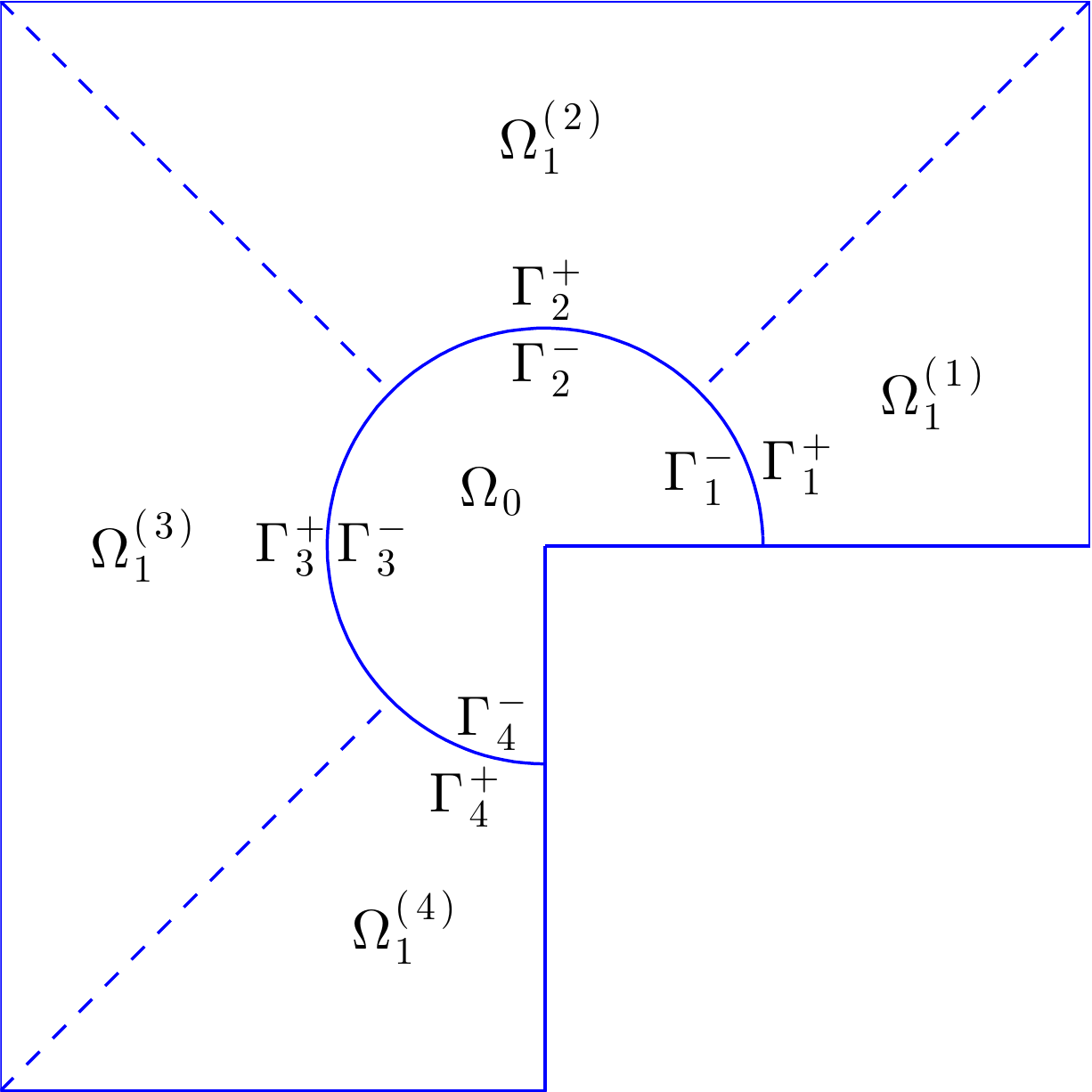}
\end{minipage}\hspace*{\fill}%

\caption{Non-conforming meshes for the mortar spectral method on a square (left) and
an L-shape domain (right).}
\label{fig:NCM}
\end{figure}

We now introduce the approximation space.
Denote $\delta=(K_0,N_0, K_1,N_1)$.
Let  $T_0: \mathbb{B}^2\mapsto \Omega_0$ such that $(x,y)=T_0(\xi,\eta)=(R\xi, R\eta)$, and define
the approximation space on $\Omega_0$,
\begin{align*}
X_{\delta,0}=
\spn\big\{ P_{k,\ell}^{n} \circ T_0^{-1}:  0 \le k\le K_0, \,  1\le \ell\le a_n^2,\, 0\le n \le N_0 \big\}.
\end{align*}
To introduce the approximation space on $\Omega_1$,
we first make the Gordon-Hall transformations $T_{\kappa}: [-1,1]^2\mapsto \Omega_1^{(\kappa)}$, $\kappa=1,2,3,4$, such that
\begin{align*}
(x+\i y)  =\i^{\kappa-1}\left[ \frac{1+\xi}2 (1+\i\eta) + \frac{1-\xi}{2} R \e^{\frac{\i\pi \eta }{4}}\right],
\end{align*}
or equivalently,
\begin{align}
\label{GHT0}
\begin{cases}
x = \dfrac{1+\xi}{2} \left[ \sin \dfrac{\pi \kappa}{2} + \eta \cos \dfrac{\pi \kappa}{2} \right]
+ \dfrac{1-\xi}{2}  R \sin \dfrac{\pi (\eta+2\kappa)}{4} ,
\\[0.5em]
y = \dfrac{1+\xi}{2} \left[\eta \sin \dfrac{\pi \kappa}{2} - \cos \dfrac{\pi \kappa}{2} \right]
- \dfrac{1-\xi}{2}  R \cos \dfrac{\pi (\eta+2\kappa)}{4}.
\end{cases}
\end{align}
Then we define the conforming approximation space on $\Omega_1$ as following,
\begin{align*}
&X_{\delta,1}=\spn\big\{
v\in H^{1}(\Omega_{1}):
  v|_{\Omega_{1}^{(\kappa)}} \circ T_{\kappa}\in X_{K_{1},N_{1}}, \kappa=1,2,3,4  \big\},
  \\
&X_{K_1,N_1}=\spn\left\{  v \in \mathbb{P}_{K_1}\times \mathbb{P}_{N_1}:   v(1,\eta)=0 \right\}.
\end{align*}
Finally, our mortar
approximation space on $\Omega$ is defined by
\begin{align*}
X_{\delta}=\spn\left\{ v: v|_{\Omega_{i}} \in X_{\delta,i}, i=0,1; \quad
( \gamma^{-}v - \gamma^{+}v,  \phi   )_{\Gamma_R} = 0,  \quad  \phi \in V_{\delta}
  \right\},
\end{align*}
where
the non mortar space $V_{\delta}$ is defined through  either    the trace  operator $\gamma^{+}: X_{\delta,1}\mapsto L^2(\Gamma_R)$   or $\gamma^{-}: X_{\delta,0}\mapsto L^2(\Gamma_R)$, i.e., $V_{\delta}\subseteq \{  \gamma^{+}v_{\delta}: v_{\delta}\in  X_{\delta,1}  \}$ or
 $ V_{\delta}\subseteq \{  \gamma^{-}v_{\delta}: v_{\delta}\in  X_{\delta,0}  \}$.
     Note that the matching condition $( \gamma^{-}v - \gamma^{+}v,  \phi   )_{\Gamma_R} = 0, \  \phi \in V_{\delta}$ is enforced to guarantee  information interchange between  $X_{\delta,1}$ and $X_{\delta,0}$,
     and here  we simply use
     $$ V_{\delta}= \spn\{  \gamma^{-}v_{\delta}: v_{\delta}\in  X_{\delta,0}  \}
     = \spn\left\{ \e^{\i n\theta }:    -N_0 \le n \le N_0 \right\}.$$

The mortar spectral element approximation scheme reads:  Find the eigenpairs
 $( \lambda_{\delta}, u_{\delta}) \in  \RR_+\times  X_{\delta}$ such that
\begin{align}
   \label{VarMSEM}
  & (\nabla u_{\delta}, \nabla v_{\delta})_{\Omega} +  (u_{\delta}, v_{\delta})_{r^{-2},\Omega} = \lambda_{\delta} (u_{\delta}, v_{\delta})_{\Omega}, \quad  v_{\delta}\in X_{\delta}.
\end{align}
If we remove the  matching condition in the definition of  $X_{\delta}$ to obtain
\begin{align*}
X_{\delta}^{*}=\left\{ v: v|_{\Omega_{i}} \in X_{\delta,i}, i=0,1  \right\},
\end{align*}
then \eqref{VarMSEM} is equivalent to the following one:
Find $( \lambda_{\delta}, u_{\delta},\psi_{\delta}) \in  \RR_+\times X_{\delta}^{*}\times V_{\delta}$ such that
\begin{align}
\label{SadMSEM}
\begin{split}
  &  (\nabla u_{\delta}, \nabla v_{\delta})_{\Omega} +  (u_{\delta}, v_{\delta})_{r^{-2},\Omega}
  +( \gamma^{-}v_{\delta} - \gamma^{+}v_{\delta},  \psi_{\delta}   )_{\Gamma_R}  = \lambda_{\delta} (u_{\delta}, v_{\delta})_{\Omega} , \quad  v_{\delta}\in X_{\delta}^{*},
   \\
&( \gamma^{-}u_{\delta} - \gamma^{+}u_{\delta},  \phi_{\delta}   )_{\Gamma_R} = 0,  \quad  \phi_{\delta} \in V_{\delta}.
\end{split}
\end{align}

Before concluding this  subsection,  we give some remarks on
the evaluation of the matrices of the reduced algebraic problems  in our mortar spectral element method.
The local stiffness matrix  associating $ (\nabla \cdot, \nabla \cdot)_{\Omega_0} +  (\cdot, \cdot)_{r^{-2},\Omega_0}$  and the  local  mass matrix associating
$(\cdot, \cdot)_{\Omega_0}$ can be easily obtained from
Lemma \ref{SOBP}.  While for the local stiffness and mass matrices on $\Omega_1^{(\kappa)}, \, (1\le \kappa\le 4)$,
we only need to consider the local matrices on $\Omega_1^{(1)}$ owing to the parity
of our symmetric partition.
We first use  the following basis functions for  $X_{K_1,N_1}$,
\begin{align*}
X_{K_1,N_1}=\left\{\phi_{k,n}(\xi,\eta):= \phi_k(\xi) \phi_n(\eta) :   0\le  k \le K_1,    1\le n \le N_1  \right\},
\end{align*}
where
\begin{align*}
   \phi_0(\zeta) = \frac{1+\zeta}{2}, \  \phi_1(\zeta) = \frac{1-\zeta}{2}, \quad  \phi_k(\zeta)
   = J^{-1,-1}_k(\zeta), \ k\ge 2.
\end{align*}
Next, we note that the following differentiation relation under the Gordon-Hall mapping $T_1$,
\begin{align*}
 \widetilde\nabla:= \begin{bmatrix}
 \partial_{\xi}\\
 \partial_{\eta}
 \end{bmatrix}  = J_1 \begin{bmatrix}
 \partial_{x}\\
 \partial_{y}
 \end{bmatrix} = J_1 \nabla
 ,\qquad
\nabla = \begin{bmatrix}
 \partial_{x}\\
 \partial_{y}
 \end{bmatrix}  = J_1^{-1}   \begin{bmatrix}
 \partial_{\xi}\\
 \partial_{\eta}
 \end{bmatrix} = J_1^{-1}   \widetilde \nabla,
\end{align*}
where
\begin{align*}
&J_{1} =  \begin{bmatrix}  \frac12-\frac{R}2 \cos\frac{\pi \eta}{4} &    \frac{\eta}2-\frac{R}2 \sin\frac{\pi \eta}{4} \\[0.5em]
\frac{\pi}{8}(\xi-1)  R\sin \frac{\pi \eta}{4} & \frac{\xi+1}{2} -\frac{\pi}{8} (\xi-1) R\cos\frac{\pi \eta}{4}
\end{bmatrix},
\\
&|J_{1}| = \det(J_1)=-\frac{(\pi+4)\xi+(4-\pi)}{16}R  \cos\frac{\pi \eta}{4}  -\frac{\pi \eta (\xi-1)}{16}R  \sin\frac{\pi \eta}{4}
+ \frac{\xi+1}{4} +\frac{R^2\pi (\xi-1)}{16}.
\end{align*}
Then  the  local mass matrix on $\Omega_1^{(1)}$ can be precisely expressed
by Bessel functions, while each entry of local stiffness matrix can be easily evaluated
through the  Gaussian quadratures on $[-1,1]^2$.
Moreover, since $\phi_{k,n}(-1, \eta) = \delta_{k,0} \phi_{n}(\eta)$, one readily finds that the matrix relating to
the mortar elements on $\Gamma_{R}$ can also be  formulated explicitly using  Bessel functions.

\subsection{Mortar spectral elements
 on a domain with reentrant corners}
 For simplicity, let us consider the L-shape domain and  suppose $\Omega=[-1,1]^2\setminus ( [0,1]\times[-1,0] )$
 such that the  point of the potential singularity is also the vertex of  the  reentrant corner. The computational domain $\Omega$
 is  partitioned using the same technique as in the previous subsection. More precisely,
 $\Omega= \Omega_0\cup \Omega_1$, $\Omega_1=\Omega_1^{(1)}\cup \Omega_1^{(2)} \cup  \Omega_1^{(3)} \cup \Omega_1^{(4)}$, and
 \begin{align*}
 & \Gamma_R = \{(x,y)= (R\cos\theta, R\sin\theta ):  0\le \theta \le \frac{3\pi}{2} \}\subset \Omega,
 \\
 & \Omega_0 =  \{(x,y)= (r\cos\theta, r\sin\theta ): 0\le r<R, \,   0\le \theta \le \frac{3\pi}{2} \},
 \\
&\Omega_1^{(1)} = \{(x,y)\in \RR^2:  x^2+y^2>R^2,\,  0<y< x < 1 \},
\\
&\Omega_1^{(2)} = \{(x,y)\in \RR^2:  x^2+y^2>R^2,\,  |x|< y < 1 \},
\\
&\Omega_1^{(3)} = \{(x,y)\in \RR^2:  x^2+y^2>R^2,\,  -1<x<-|y| \},
\\
&\Omega_1^{(4)} = \{(x,y)\in \RR^2:  x^2+y^2>R^2,\,  -1<y<x<0 \}.
\end{align*}
The Gordon-Hall mappings $T_{\kappa}: [-1,1]^{2}\mapsto \Omega_{1}^{(\kappa)}$
are the same as in \eqref{GHT0} for $\kappa=2,3$ and are defined as following for $\kappa=1,4$,
\begin{align*}
F_{1}:\quad
\begin{cases}
 x  =\tfrac{1+\xi}{2} + \tfrac{1-\xi}{2} R \cos \frac{\pi (\eta+1)}{8},
 \\
 y  =\tfrac{1+\xi}{2} \tfrac{1+\eta}{2}  + \tfrac{1-\xi}{2} R \sin \frac{\pi (\eta+1)}{8},
 \end{cases}
 \\
 F_{4}:\quad
\begin{cases}
x  =-\tfrac{1+\xi}{2} \tfrac{1+\eta}{2}  - \tfrac{1-\xi}{2} R \sin \frac{\pi (\eta+1)}{8},
\\
y  =-\tfrac{1+\xi}{2} - \tfrac{1-\xi}{2} R \cos \frac{\pi (\eta+1)}{8}.
 \end{cases}
\end{align*}

Denote $\delta=(K_0,N_0, K_1,N_1, K_2,N_2, K_3,N_4, K_4,N_4)$ and let $T_0: \mathbb{B}^2\mapsto \Omega_0$ be the same as in the former subsection.  We now define
the approximation spaces on $\Omega_0$ and $\Omega_1$,
\begin{align*}
X_{\delta,0}= \left\{ v:  v \circ  T_{0} \in V_{K_0,N_0}^{\Lambda} \right\} =
\spn\big\{ \Psi_{k}^{n}  \circ  T_{0}^{-1}:  1 \le k\le K_0, \, 0\le n \le N_0 \big\}.
\end{align*}
and
\begin{align*}
&X_{\delta,1}=\spn\big\{
v\in H^{1}(\Omega_{1}):
  v|_{\Omega_{1}^{(\kappa)}} \circ T^{-1}_{\kappa}\in X_{K_{\kappa},N_{\kappa}}, \kappa=1,2,3,4  \big\},
  \\
&X_{K_{\kappa},N_{\kappa}}=\spn\left\{  v \in \mathbb{P}_{K_{\kappa}}\times \mathbb{P}_{N_{\kappa}}:   v(1,\eta)=0 \right\}.
\end{align*}
Our mortar
approximation space on $\Omega$ is defined again  by
\begin{align*}
X_{\delta}=\spn\left\{ v: v|_{\Omega_{i}} \in X_{\delta,i}, i=0,1; \quad
( \gamma^{-}v - \gamma^{+}v,  \phi   )_{\Gamma_R} = 0,  \quad  \phi \in V_{\delta}
  \right\},
\end{align*}
where the non mortar space $V_{\delta}$ is chosen as
     $$ V_{\delta}= \spn\{  \gamma^{+}v_{\delta}: v_{\delta}\in  X_{\delta,1}  \}$$
     to ensure a spectrally high approximation accuracy on the $\Gamma_R$.

Now the mortar spectral element approximation scheme for \eqref{eigenP} on the L-shape domain
with a  singular corner exactly follows the formulas \eqref{VarMSEM} and \eqref{SadMSEM}.

At last, we conclude this subsection with the remark that one can readily extend our mortar spectral element
method  to solve   \eqref{eigenP}  with multiple singular potentials and  reentrant/obtuse  corners.

\subsection{Numerical experiments}
In this subsection, we shall show some numerical results on the mortar spectral element method (MSEM) for
\eqref{eigenP}.  To evaluate our method, we first introduce the $hp$-finite element method  using
geometric mesh (GFEM) by Gui, Guo and Babu\v{s}ka \cite{GuiBab86b,BabGuo96}  to handle corner/polar singularity
 of type  $r^{\alpha}$  in numerical PDE. This method
offers the best (exponential) convergence rate  among traditional methods for handling such kind of singularities.
The geometric mesh is characterized  by $n$ layers of conforming elements, in which the size
of  elements in the $i$-th layer, $h_i  =c q^{n-i}, \, 0<q<1,\,   1\le i\le n$,
and the  polynomial degree of  elements in the $i$-th layer,  $p_i$, is proportional to its layer number $i$.
In such a way, the mesh is refined as $n$ increases and
simultaneously the degrees of elements are increased too.

\begin{figure}[h!]
\hspace*{\fill}\begin{minipage}[h]{0.3\textwidth}
\includegraphics[width=1\textwidth]{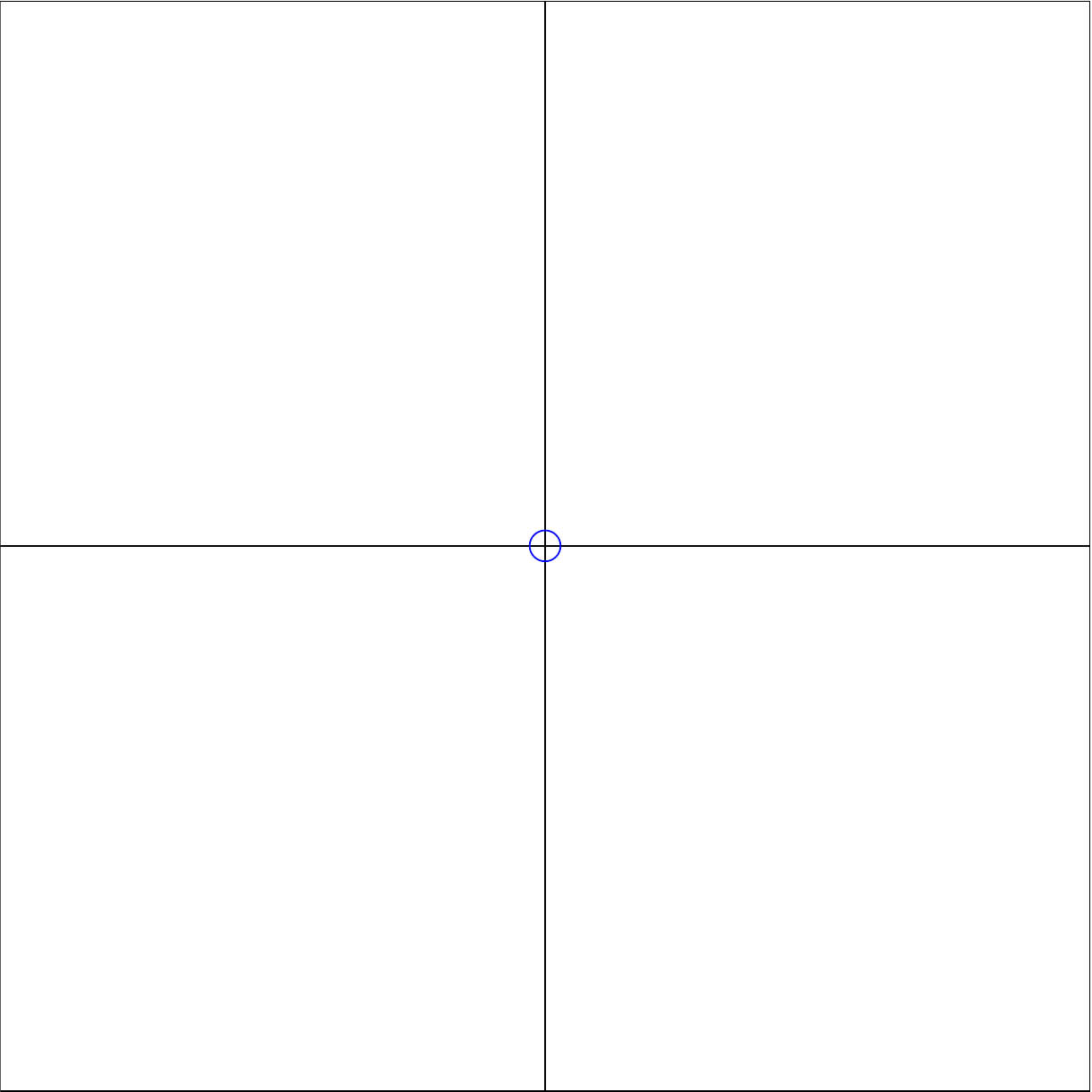}
\end{minipage}\hspace*{\fill}%
\hspace*{\fill}\begin{minipage}[h]{0.3\textwidth}
\includegraphics[width=1\textwidth]{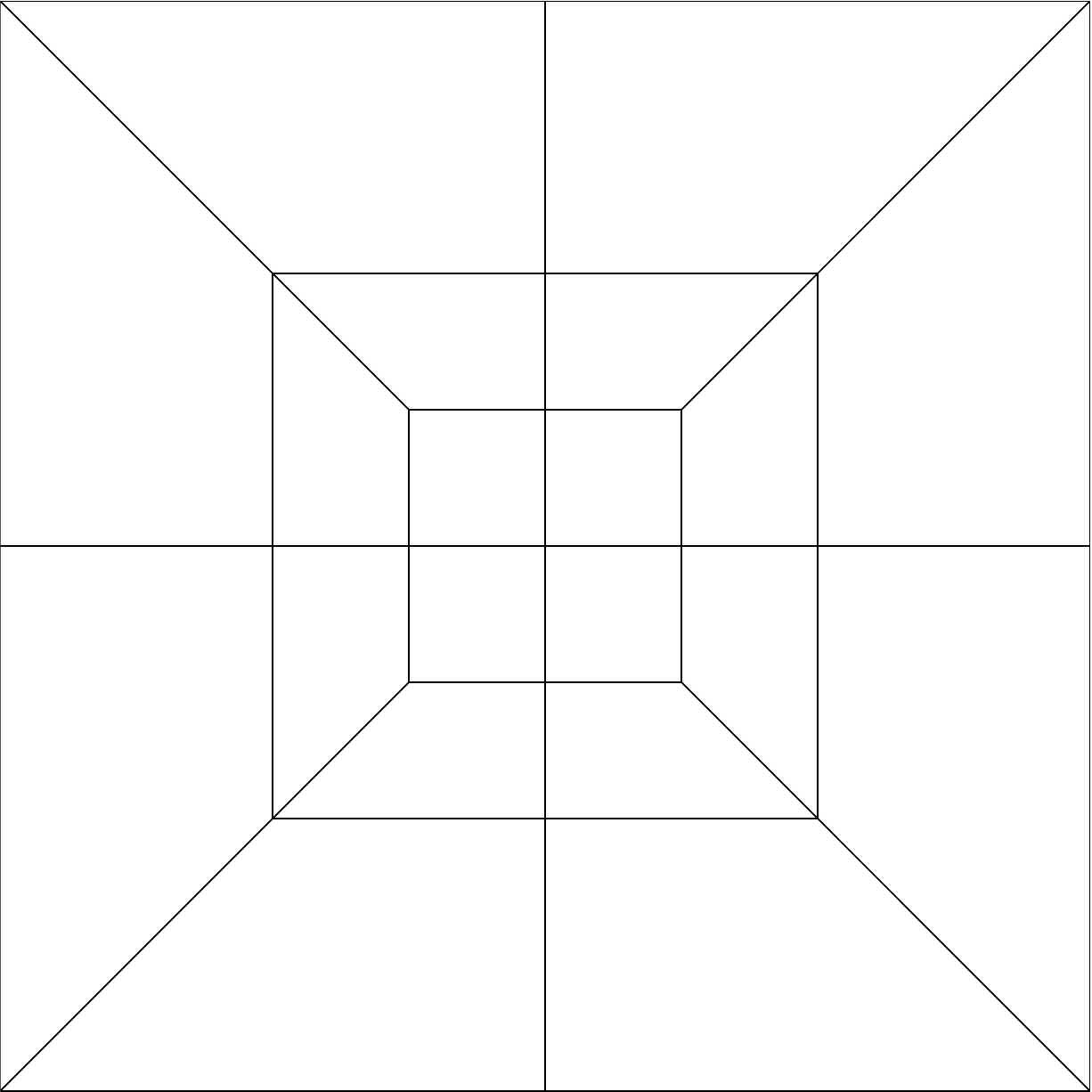}
\end{minipage}\hspace*{\fill}%
\hspace*{\fill}\begin{minipage}[h]{0.3\textwidth}
\includegraphics[width=1\textwidth]{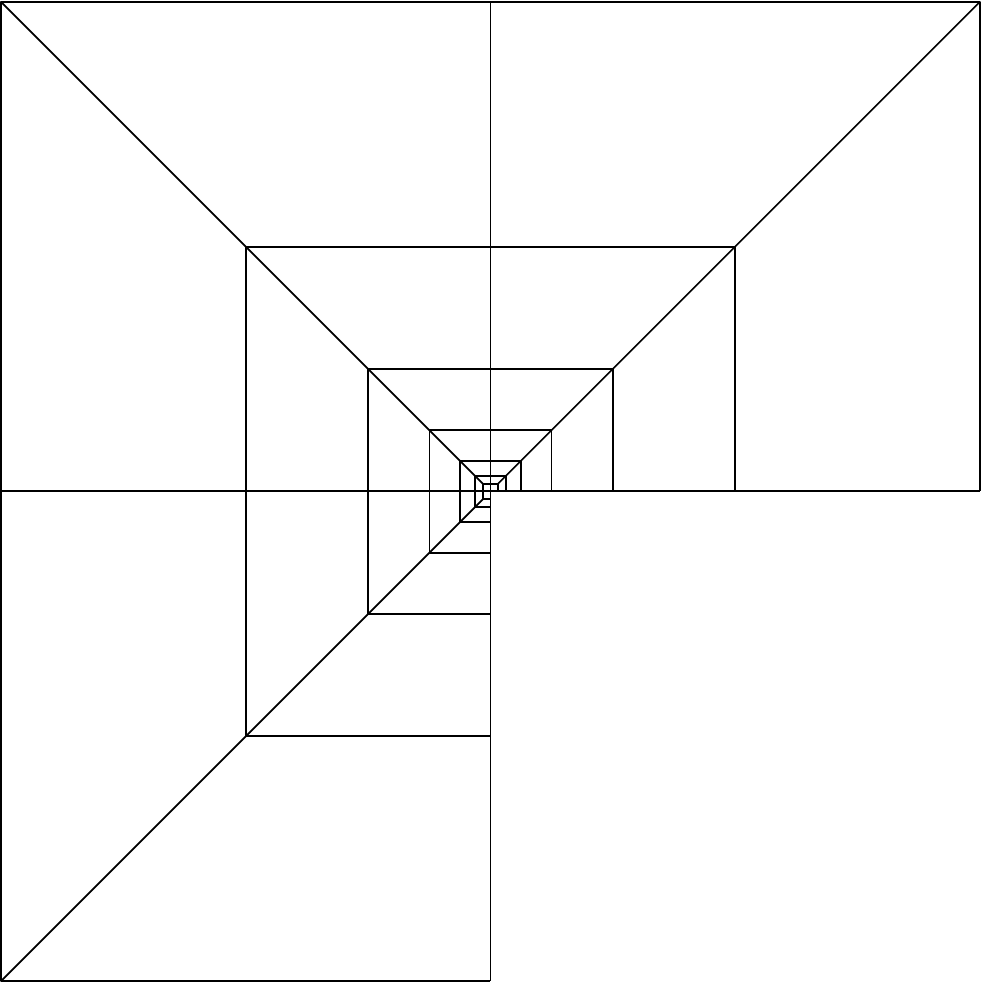}
\end{minipage}\hspace*{\fill}
\caption{The geometric meshes around the original on the square $\Omega=[-1,1]^2$ with  the mesh level
$n=1$ (left), $n=3$ (center) and on the L-shape domain $\Omega=[-1,1]^2\setminus([0,1]\times [-1,0])$ with  the mesh level  $n=7$ (right).}
\label{fig:GMSQ}
\end{figure}

\paragraph{\bf Example 1}
We first  examine the Schr\"{o}dinger equation \eqref{eigenP} on $[-1,1]^2$  with $c=1/2$ and $c=2/3$. The reference eigenvalues are evaluated by  GFEM with  $n=17$, $q=(\sqrt{2}-1)^2$ and  $p_i=i, \, 1\le i\le n$,
to obtain a 15-digit precision with an optimal convergence order \cite{GuiBab86b}.
Numerical eigenvalues of  MSEM are computed with the parameters $R=0.3$ and $\delta=(10,14,17,18)$,
whose absolute errors are reported in the fourth columns of Table \ref{tab:MSEM1} and   Table \ref{tab:MSEM2}.
It can be easily observed that the error of the MSEM with the total degrees of freedom (DoF) $1539$ is  close to the machine precision, which reflect the  spectral accuracy of our MSEM method.

In comparison, we also introduce a spectral element method (SEM) with four rectangular subdomains with the
common vertex at the original.
This method is proposed from the standard spectral element method by enforcing the vanishing of all the basis functions
at the original, hence  is equivalent to the reduced GFEM with $n=1$.  Errors with the (separate) polynomial degree $p=61$ are then listed in the fifth columns of Table \ref{tab:MSEM1} and   Table \ref{tab:MSEM2}, where an obviously low accuracy is found, especially
for the first and the fifth smallest eigenvalues.
This indicates only a limited low order of convergence rate can  be obtained  in a classic spectral element discretization.

To take an insight of  the superiority of  MSEM to GFEM, we  further present
in the last columns of Table \ref{tab:MSEM1} and   Table \ref{tab:MSEM2}  the approximation errors of  GFEM
 using  a slightly larger degrees of freedom.  It is obvious that
our MSEM acquires an accuracy at least  6-digit higher than GFEM. Quantitatively,
let us compare the error plots of the 4 smallest MSEM eigenvalues (i.e., $|\lambda_i-\lambda_{i,\delta}|$ versus $\sqrt{DoF}$  in a semi logarithm scale) in Figure \ref{fig:QMSEM}   with  the error plots of  the GFEM eigenvalues (i.e., $|\lambda_i-\lambda_{i,\delta}|$ versus $\sqrt[3]{DoF}$) in Figure \ref{fig:QGFEM}. The plots reveal that MSEM converges asymptotically in  $\mathcal{O}\big(\exp({-\sigma_1  \sqrt{DoF}})\big)$ while
GFEM converges only in  $\mathcal{O}\big(\exp({-\sigma_2  \sqrt[3]{DoF}})\big)$ with $\sigma_2$ varying from case to case.  More importantly, using generic local basis functions in the function space which the eigenfunctions belong to,  MSEM characterizes the underlying  singularities perfectly  and  approximates consecutive eigenvalues (together with the corresponding eigenfunctions) with an almost uniform convergence rate.
While GFEM mimics the singular solution by a balance between local mesh sizes and local polynomial degrees,
and can only remove part of the singularities. Thus it approximates  consecutive eigenvalues
with different convergence rates whenever their associate eigenfunctions possess  different orders of singularities.

\begin{table}[h!]
\caption{The  reference values of the first 6 eigenvalues on $[-1,1]^2$ for $c=1/2$;  the errors  of
MSEM  with  $DoF=1539$,
 of  SEM with $DoF=14640$, and of the GFEM  with $n=8$ and  $DoF=1624$.
}

\hspace*{\fill}
\begin{tabular}{c||c|c|c|c|c}
 No. &    Ref.   &  Mul.  &  MSEM &  SEM  &  GFEM \\ \hline
 1 &    8.37681498711058 & 1 & 5.3291e-15 &    1.1417e-03  &  7.9985e-6 \\ \hline
 2 &  13.35313963139164 & 2 & 8.8818e-15 &    2.7979e-08  & 6.4234e-9 \\ \hline
 3 &  20.33106215893244 & 1 &  3.5527e-15   &    2.4869e-14 & 3.3054e-8 \\ \hline
 4  & 25.42501776089188 & 1 &   4.9738e-14 &    1.8474e-13  & 1.8657e-8 \\ \hline
 5 &  30.86901223422695 & 1 &   3.3040e-13 &    3.3095e-03 & 2.2988e-5
   \\ \hline
 6 & 32.83995595781530 & 2 & 3.5527e-14 &    2.6943e-08 & 1.2435e-5
\end{tabular}
\hspace*{\fill}
\vspace*{0.5em}
\label{tab:MSEM1}
\end{table}

\begin{table}[h!]
\caption{The  reference values of the first 6 eigenvalues on $[-1,1]^2$ for $c=2/3$;  the errors  of
MSEM  with  $DoF=1539$,
 of  SEM with $DoF=14640$, and of  GFEM  with $n=8$ and  $DoF=1624$.
}
\hspace*{\fill}
\begin{tabular}{c||c|c|c|c|c}
 No. &   Ref.  &  Mul.  &  MSEM & SEM & GFEM \\ \hline
 1 &   9.65231567885163 & 1 & 1.7764e-15  &   8.9349e-5  & 2.4365e-7\\ \hline
 2 &  14.0914338712714 & 2 & 1.0658e-14   &   3.1005e-8  & 1.1318e-8 \\ \hline
 3 &  20.7838715370525 & 1 &  2.4869e-14  &   7.8160e-14  & 3.5321e-8 \\ \hline
 4 &  25.9999831911128   &1   &  7.1054e-14 &  6.7502e-14 & 2.1666e-8\\ \hline
 5 & 32.8581767543383 &  1 &   7.8160e-14 &   2.7518e-04 & 8.0361e-7
   \\ \hline
 6 & 33.3937111616692 & 2 & 1.4211e-14  &   2.9763e-08 & 1.5920e-6
\end{tabular}
\hspace*{\fill}
\vspace*{0.5em}
\label{tab:MSEM2}
\end{table}

\begin{figure}[h!]
\hspace*{\fill}
\begin{minipage}[h]{0.35\textwidth}
\includegraphics[width=1\textwidth]{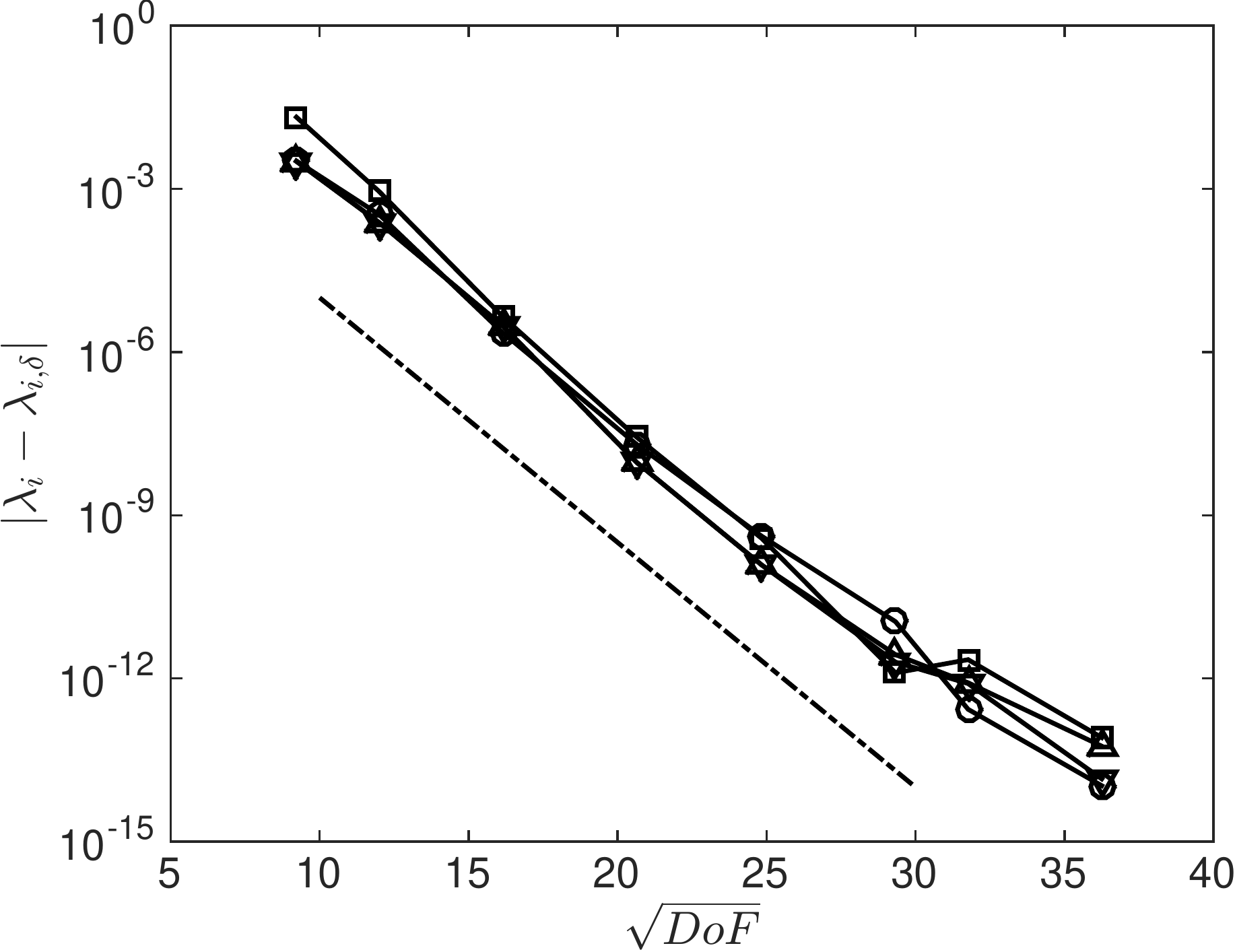}
\hspace*{\fill} (a). $c=1/2$. \hspace*{\fill}
\end{minipage}\hspace*{\fill}%
\begin{minipage}[h]{0.35\textwidth}
\includegraphics[width=1\textwidth]{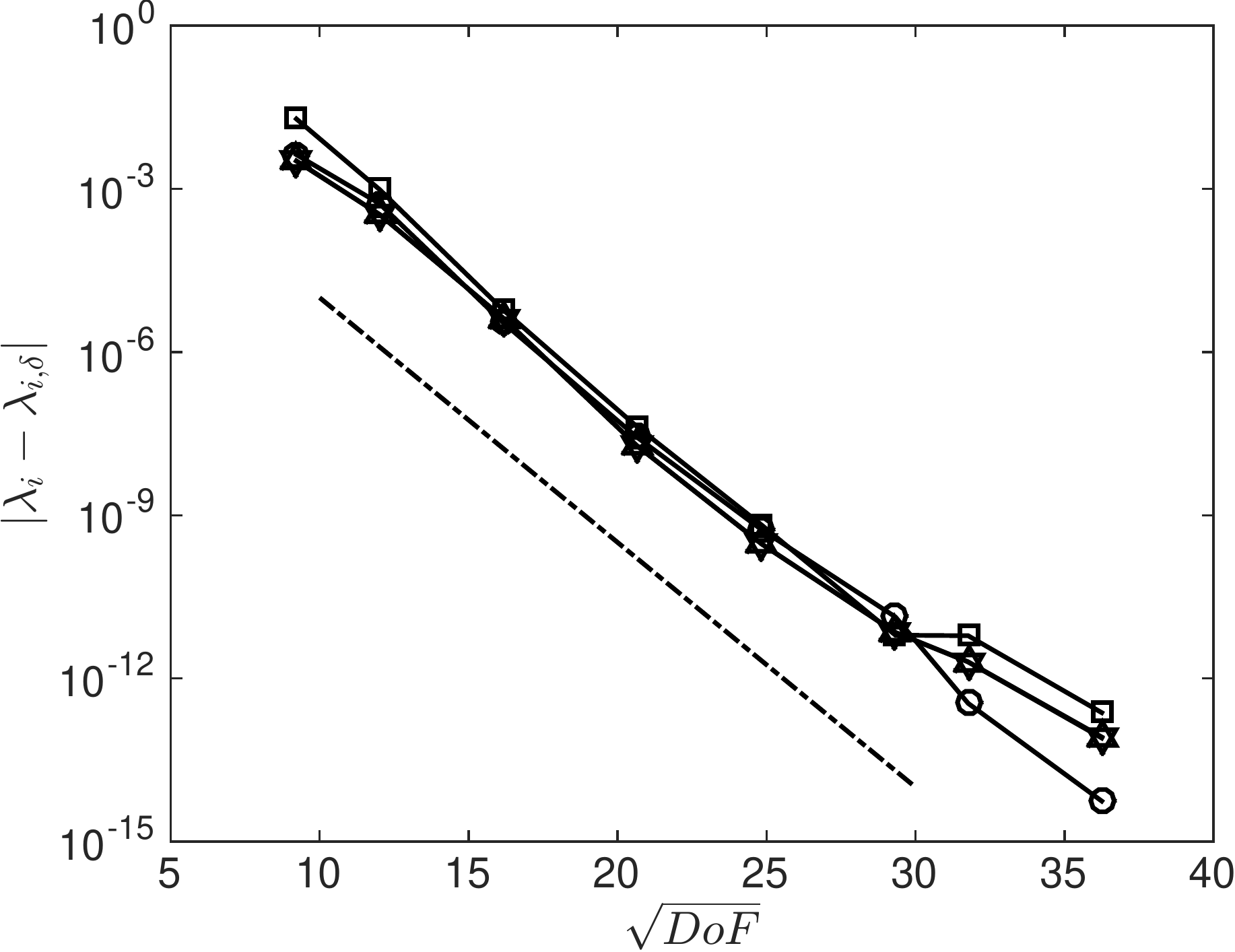}
\hspace*{\fill} (b). $c=2/3$.  \hspace*{\fill}
\end{minipage}\hspace*{\fill}%
\caption{MSEM approximation errors $|\lambda_i-\lambda_{i,\delta}|$  ($\circ: \lambda_1$,  $\triangledown:\lambda_2$, $\vartriangle:\lambda_3$ and  $ \square: \lambda_4$)  versus $\sqrt{DoF}$ on the square $[-1,1]^2$.
The dash-dot lines are the reference exponential $y = 10^{-0.45 \sqrt{DoF}-0.5 }$.}
\label{fig:QMSEM}
\end{figure}

\begin{figure}[h!]
\hspace*{\fill}
\begin{minipage}[h]{0.35\textwidth}
\includegraphics[width=1\textwidth]{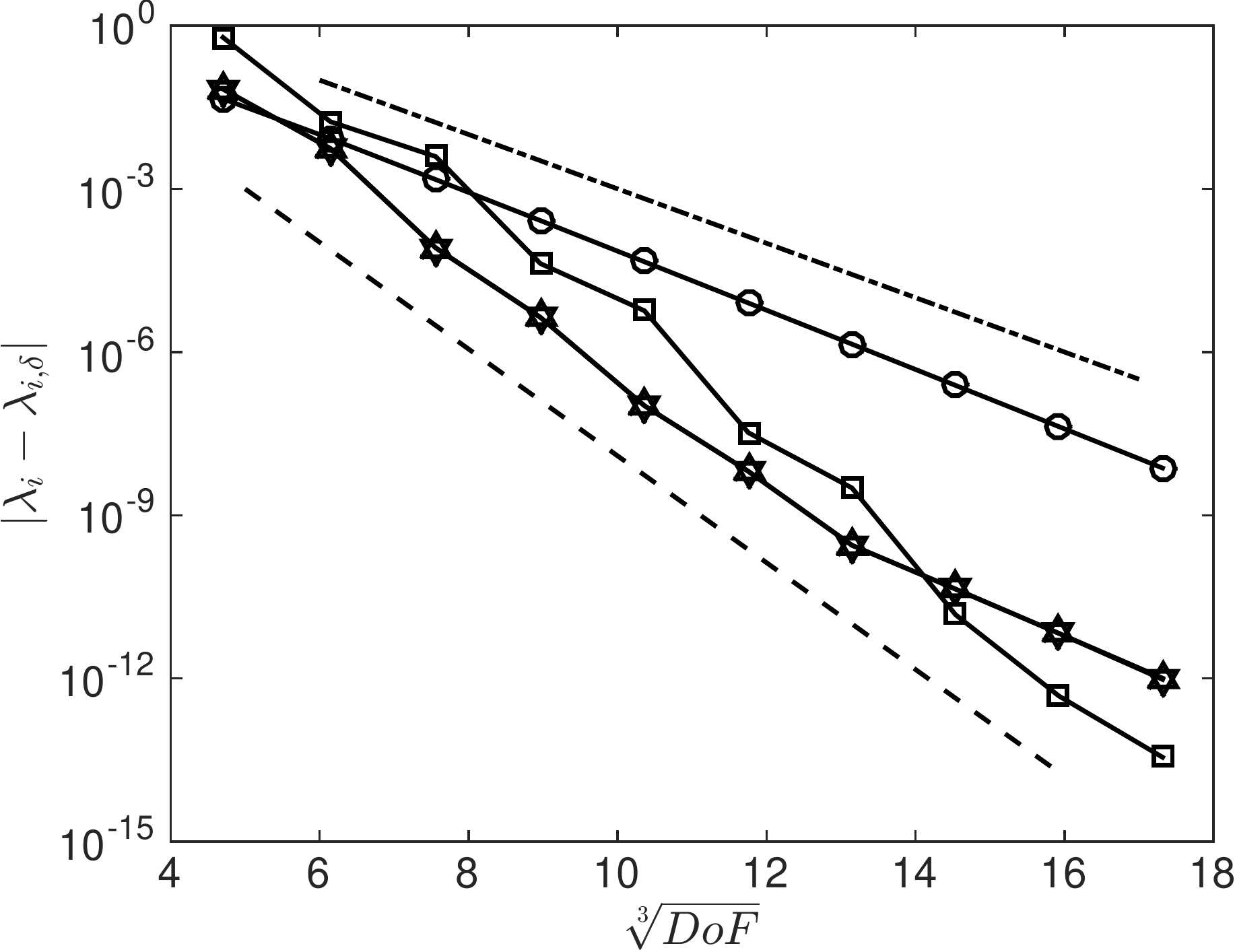}
\hspace*{\fill} (a). $c=1/2$. \hspace*{\fill}
\end{minipage}\hspace*{\fill}%
\begin{minipage}[h]{0.35\textwidth}
\includegraphics[width=1\textwidth]{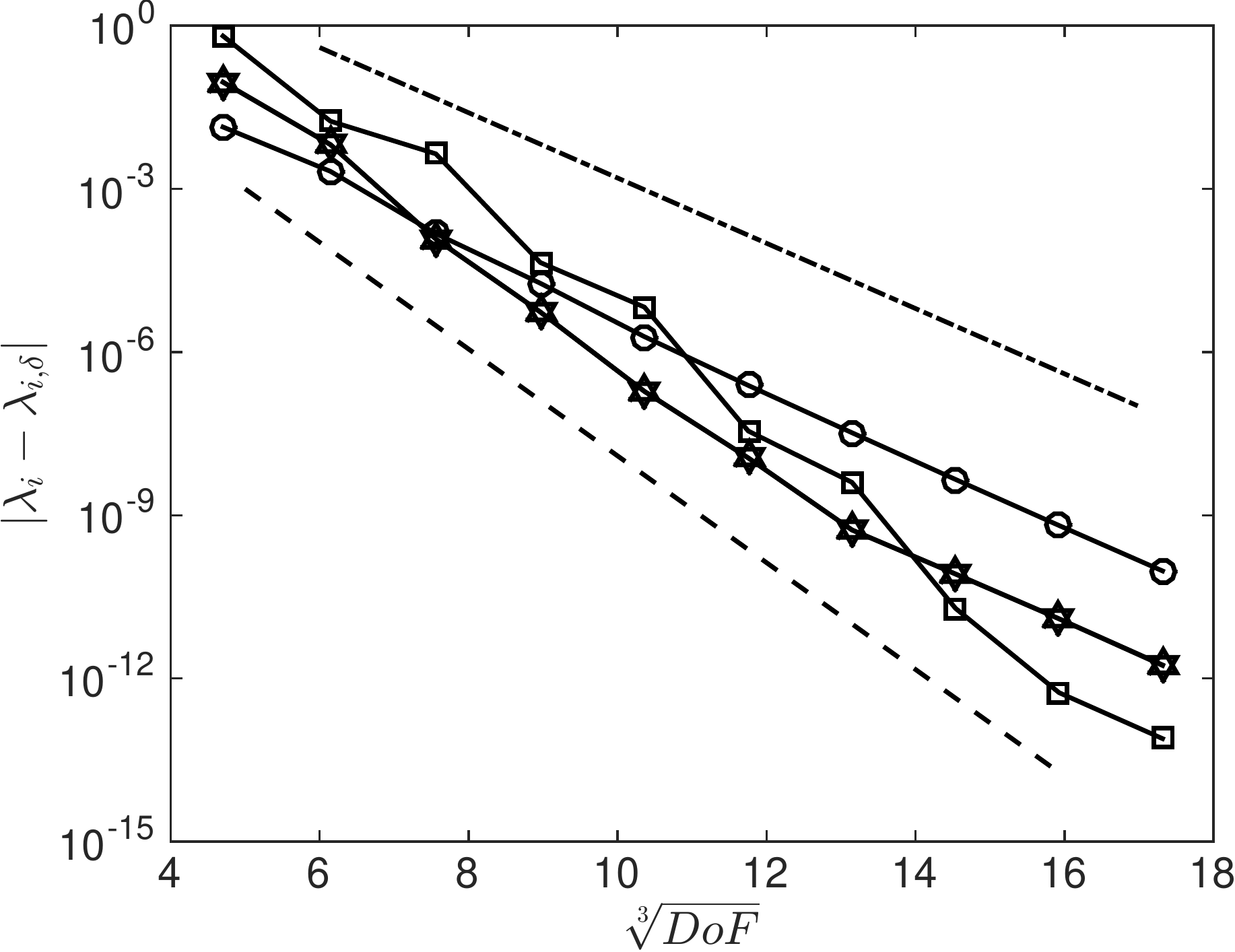}
\hspace*{\fill} (b). $c=2/3$. \hspace*{\fill}
\end{minipage}\hspace*{\fill}%
\caption{GFEM approximation errors $|\lambda_i-\lambda_{i,\delta}|$  ($\circ: \lambda_1$,  $\triangledown:\lambda_2$, $\vartriangle:\lambda_3$ and  $ \square: \lambda_4$) versus $\sqrt[3]{DoF}$ on the square $[-1,1]^2$.
 The  dashed  lines are $y = 10^{-0.98 \sqrt[3]{DoF}+1.9 }$;  the dash-dot lines are  (a). $y = 10^{-0.5 \sqrt[3]{DoF}+2 }$
and (b). $y = 10^{-0.6 \sqrt[3]{DoF}+3 }$.}
 \label{fig:QGFEM}
\end{figure}

\vspace*{-1.5em}
\paragraph{\bf Example 2}
Further, let us examine the MSEM for solving \eqref{eigenP}  on the L-shape domain.
Once again, the reference eigenvalues are given by the GFEM using a $17$-level geometric mesh with
 $q=(\sqrt{2}-1)^2$ and  $p_i=i$ on each subdomain of level $i$, which amounts to
a high degrees of freedom $10552$.
In Table \ref{tab:LLap}, absolute errors  (for $c=0$) are given in the third column for  the MSEM eigenvalues, which  are evaluated using the discretization parameters $R=0.5$
and $\delta=(\{17,20\}, \{15,9\}, \{15,18\}^{2}, \{15,9\})$ with the total  degrees of freedom $1152$.
As a comparison, approximation errors for SEM (GFEM with $n=1$ and $p=64$) with the degrees of freedom $12033$ and   GFEM ($n=8$)  with the degree of freedom $1182$
are listed in the subsequent columns, respectively.
The defect of the SEM  for \eqref{eigenP} on the L-shape domain  is as obvious as before;
while the  difference in error between the GFEM and our MSEM  is astonishing, and MSEM
is several-digit superior to GFEM in accuracy.

Moreover, we utilize  the MATLAB code of the modified method of particular solution (MMPS)
provided by Betcke and Trefethen \cite{BT05} for solving the Laplacian eigenvalue problem on the L-shape domain
with $N=72$ particular solutions (trial functions), $2N$ equally distributed boundary points and $2N$ randomly  distributed interior points. Approximation errors  are then  reported  in the last column.
This shows that our MSEM  is also superior to MMPS for small eigenvalues, although the latter is
specifically  designed for evaluating  Laplacian eigenvalues on  a polygon domain with reentrant corners.

To have a fair comparison, a few issues need to be addressed.
The method of particular solution (MPS) starts with various solutions of the eigenvalue equation for a given $\lambda$, and then vary $\lambda$ until one can find a linear combination of such solutions that satisfies the boundary condition at a number of sample points along the boundary.
 improves/revives  MPS (mainly in stability) by restricting the set of admissible functions to functions that are bounded away from zero in the interior. 
In practice this idea is realized by minimizing the angle $\phi(\lambda)$ between the space of functions that satisfy the eigenvalue equation and the space of functions that are zero on the boundary.  The advantage of  MMPS  is  that  it can  acquire an  exponential rate of convergence  by  using a small number of trial functions.
However, its convergence  can not always be  guaranteed due to  its 
nearly singular matrix resulted.  Futhermore,   MMPS is not capable of  distinguishing  multiple eigenvalues from simple eigenvalues. More  importantly, targetting  at  eigenvalue problems, MMPS  can not be directly  applicable for solving source problems as freely as our  variationaly formulated methods.

Once again, we plot the absolute errors  in Figure \ref{fig:LMSEM} and Figure \ref{fig:LGFEM} which clearly show that MSEM converges
asymptotically in  $\mathcal{O}\big(\exp({-\sigma_1  \sqrt{DoF}})\big)$ while
GFEM  converges only in  $\mathcal{O}\big(\exp({-\sigma_2  \sqrt[3]{DoF}})\big)$  with $\sigma_2$ varying possibly  from case to case. One readily observes that the convergence rates of MSEM for consecutive eigenvalues are  almost uniform, while the convergence rate of GFEM  varies depending on the singularities of the corresponding eigenfunctions. These phenomenon confirm the superiority of MSEM to GFEM.

\begin{table}[h!]
\caption{Reference values (column 2)  of the first 10 Dirichlet Laplacian  eigenvalues on $[-1,1]^{2}\setminus  ([0,1]\times [-1,0])$, and the approximation  errors of MSEM
with $DoF= 1152$, of SEM with $DoF=12033$, of GFEM with $n=8$ and $ DoF=1182$, and of MMPS.}
\hspace*{\fill}
\begin{tabular}{c||c|c|c|c|c}
 No. &   Ref. & MSEM & SEM  & GFEM  & MMPS\\ \hline
 1 & 9.639723844021988 &  1.7763e-14 &    3.9237e-05 &2.3513e-7 & 1.2189e-11 \\ \hline
 2 & 15.197251926454335 & 7.9936e-14 &    1.8547e-09 & 1.6818e-8 & 1.8474e-13  \\ \hline
 3 & 19.739208802178716 & 2.6645e-13   &    2.1316e-14 &  3.0391e-8 & 5.6843e-14 \\ \hline
 4 & 29.521481114144805 & 6.6791e-13 &    7.3931e-10  & 2.6937e-7 & 8.6407e-08 \\ \hline
 5 & 31.912635957137759 & 7.0663e-12 &    9.5822e-05 & 7.1421e-7 & 6.7748e-12 \\ \hline
 6 & 41.474509890214925 & 3.5782e-10 &    7.2059e-05  & 8.1881e-6 &5.5848e-10 \\ \hline
 7 & 44.948487781351275 & 1.1535e-09 &    9.5454e-09 & 2.5653e-5 &7.1765e-13  \\ \hline
 8 & 49.348022005446765 & 1.2818e-09 &    4.9738e-14  &1.3954e-5  &3.3040e-12  \\ \hline
 9 & 49.348022005446765 &  1.6727e-09 &    1.2079e-13  &1.4287e-4 &3.4888e-12 \\ \hline
 10 & 56.709609887385042 & 4.0229e-09 &    8.0445e-05  &5.1652e-4 &4.6896e-13
\end{tabular} \hspace*{\fill}
\label{tab:LLap}
\end{table}

\begin{figure}[h!]
\hspace*{\fill}
\begin{minipage}[h]{0.32\textwidth}
\includegraphics[width=1\textwidth]{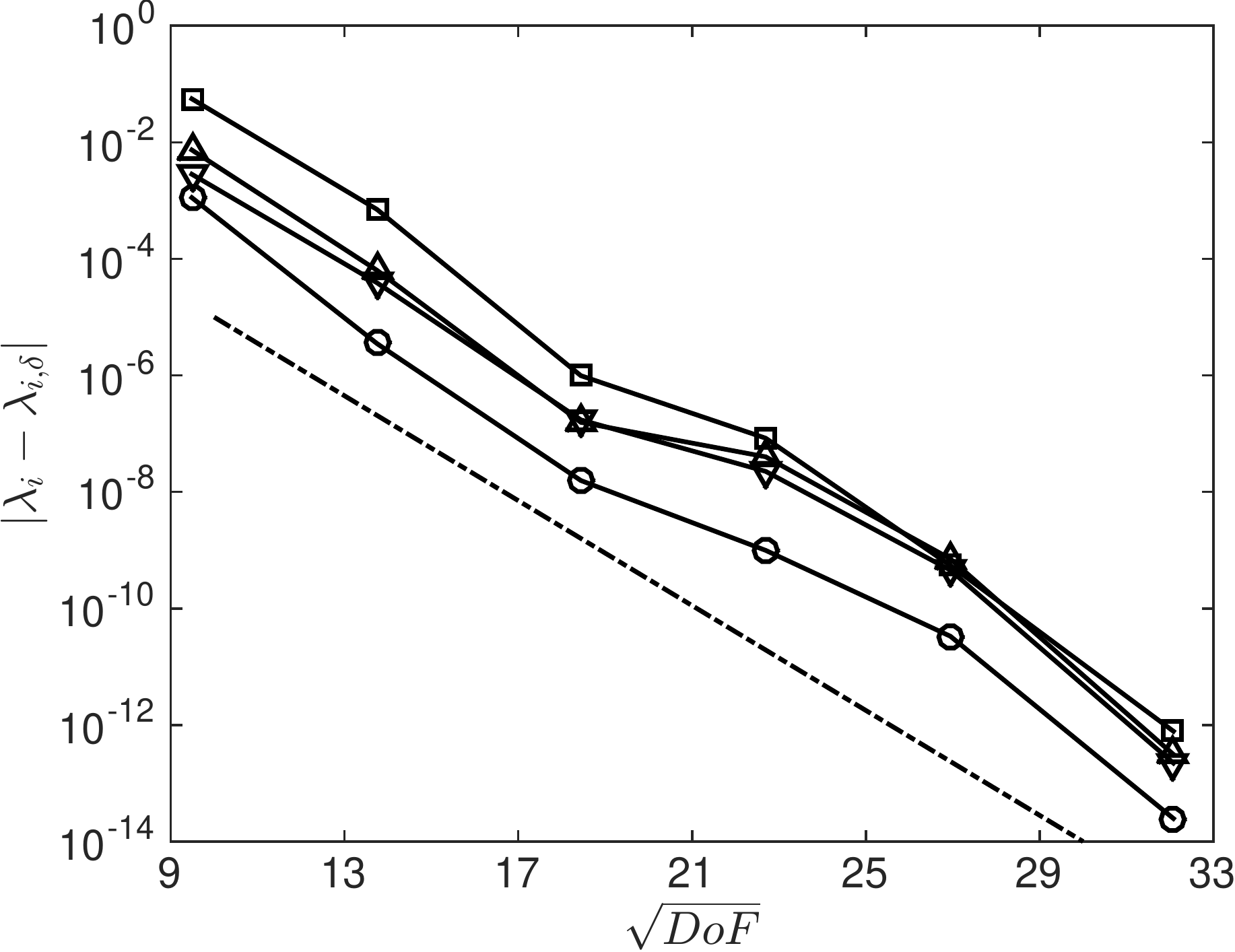}
\hspace*{\fill} (a).  $c=0$.\hspace*{\fill}
\end{minipage}\hspace*{\fill}%
\begin{minipage}[h]{0.32\textwidth}
\includegraphics[width=1\textwidth]{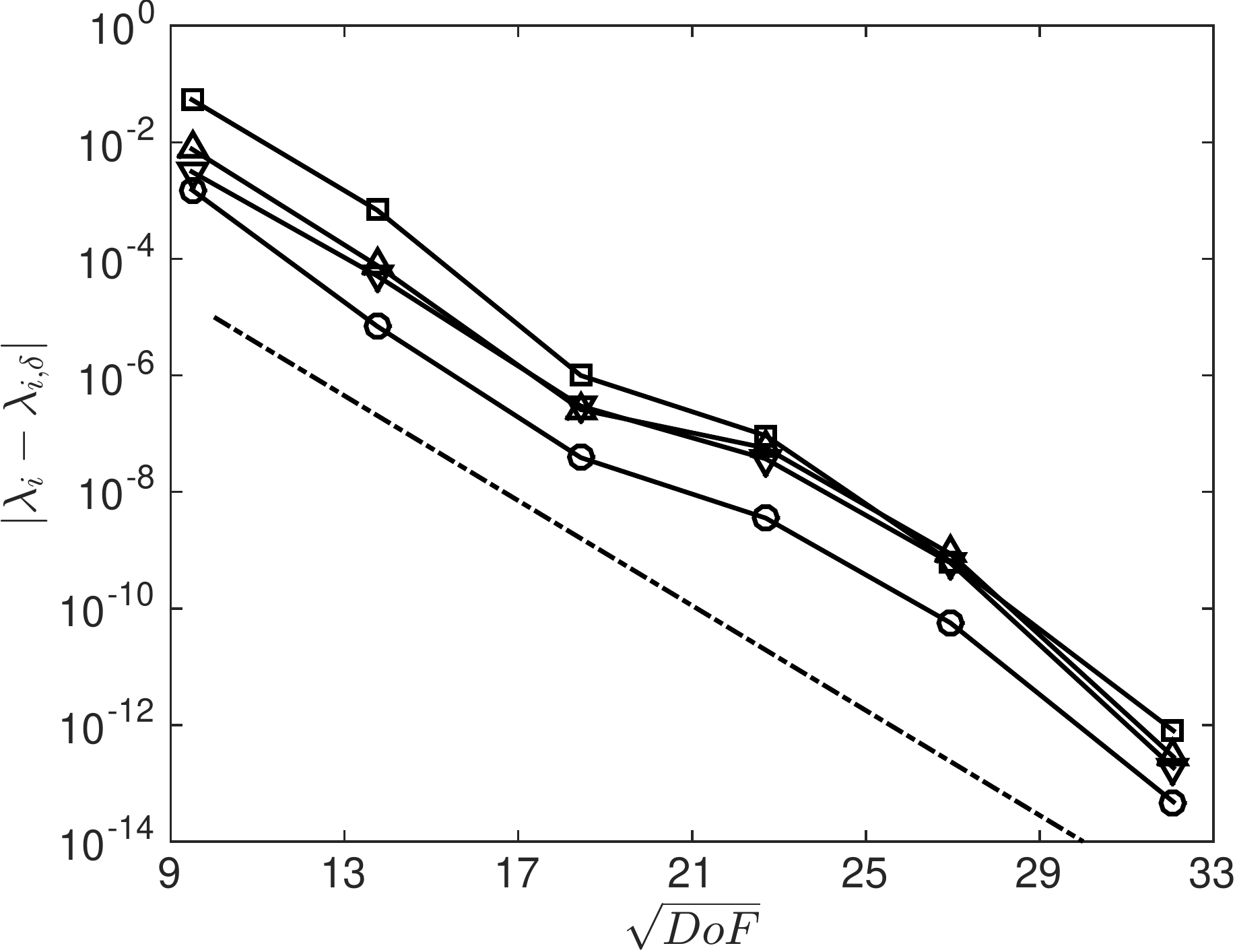}
\hspace*{\fill} (b).  $c=1/2$.\hspace*{\fill}
\end{minipage}\hspace*{\fill}%
\begin{minipage}[h]{0.32\textwidth}
\includegraphics[width=1\textwidth]{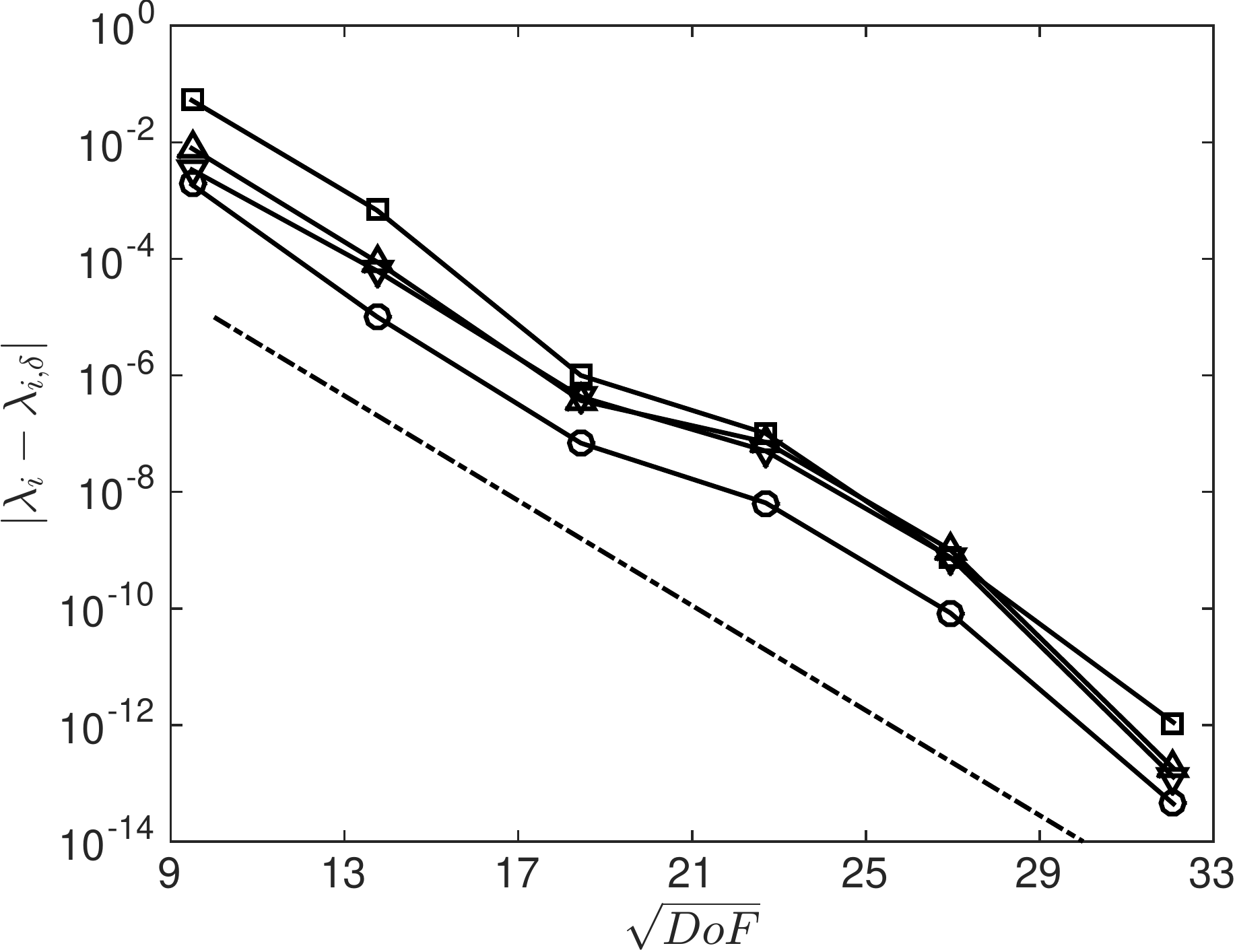}
\hspace*{\fill} (c).  $c=2/3$.\hspace*{\fill}
\end{minipage}\hspace*{\fill}%
\caption{MSEM approximation errors $|\lambda_i-\lambda_{i,\delta}|$  ($\circ: \lambda_1$,  $\triangledown:\lambda_2$, $\vartriangle:\lambda_3$ and  $ \square: \lambda_4$)  versus $\sqrt{DoF}$ on the L-shape domain $[-1,1]^2\setminus ( [0,1]\times[-1,0] )$:
The  dash-dot  lines are the exponential $y = 10^{-0.45 \sqrt{DoF}-0.5 }$.}
\label{fig:LMSEM}
\end{figure}

\begin{figure}[h!]
\hspace*{\fill}
\begin{minipage}[h]{0.32\textwidth}
\includegraphics[width=1\textwidth]{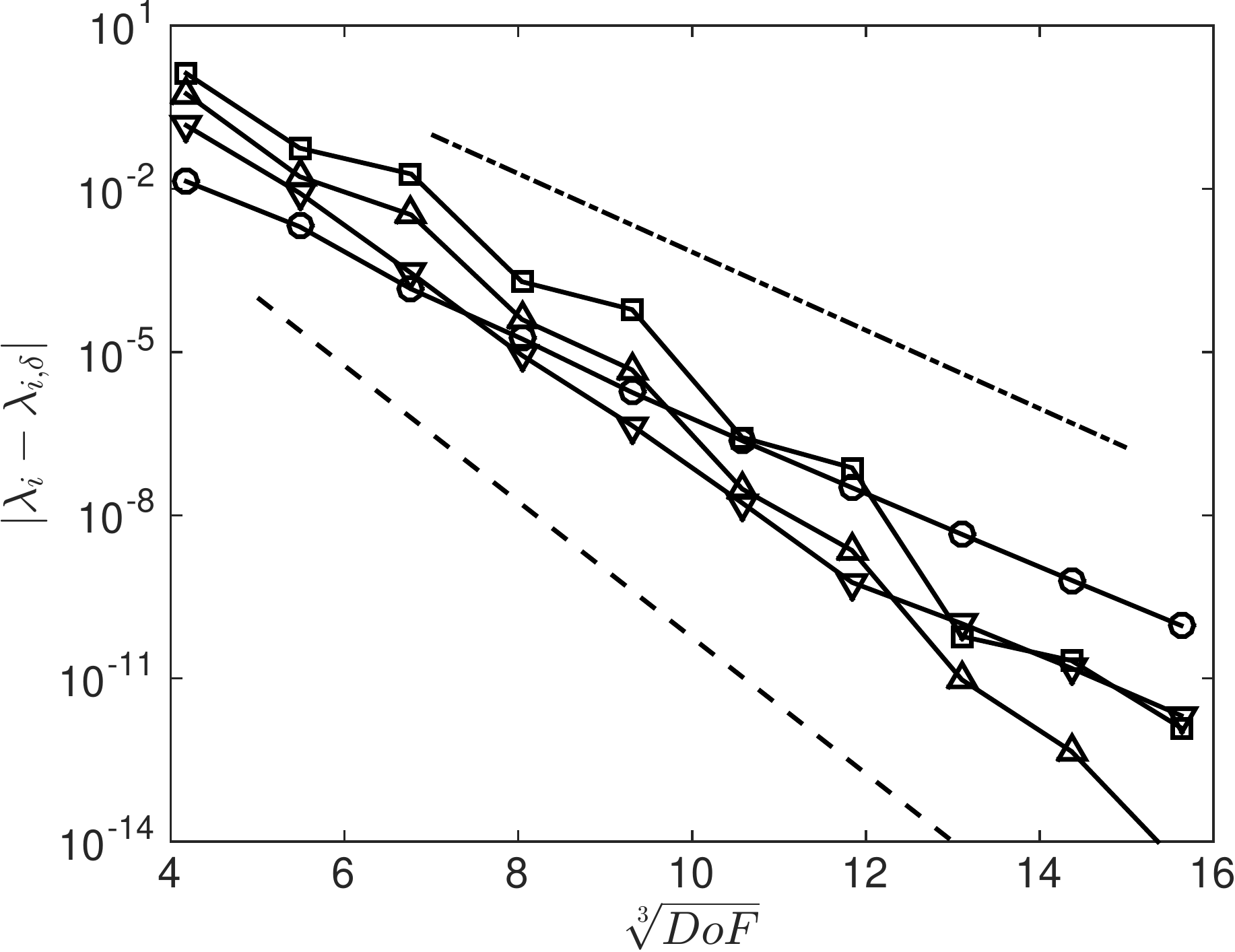}
\hspace*{\fill} (a).  $c=0$.\hspace*{\fill}
\end{minipage}\hspace*{\fill}%
\begin{minipage}[h]{0.32\textwidth}
\includegraphics[width=1\textwidth]{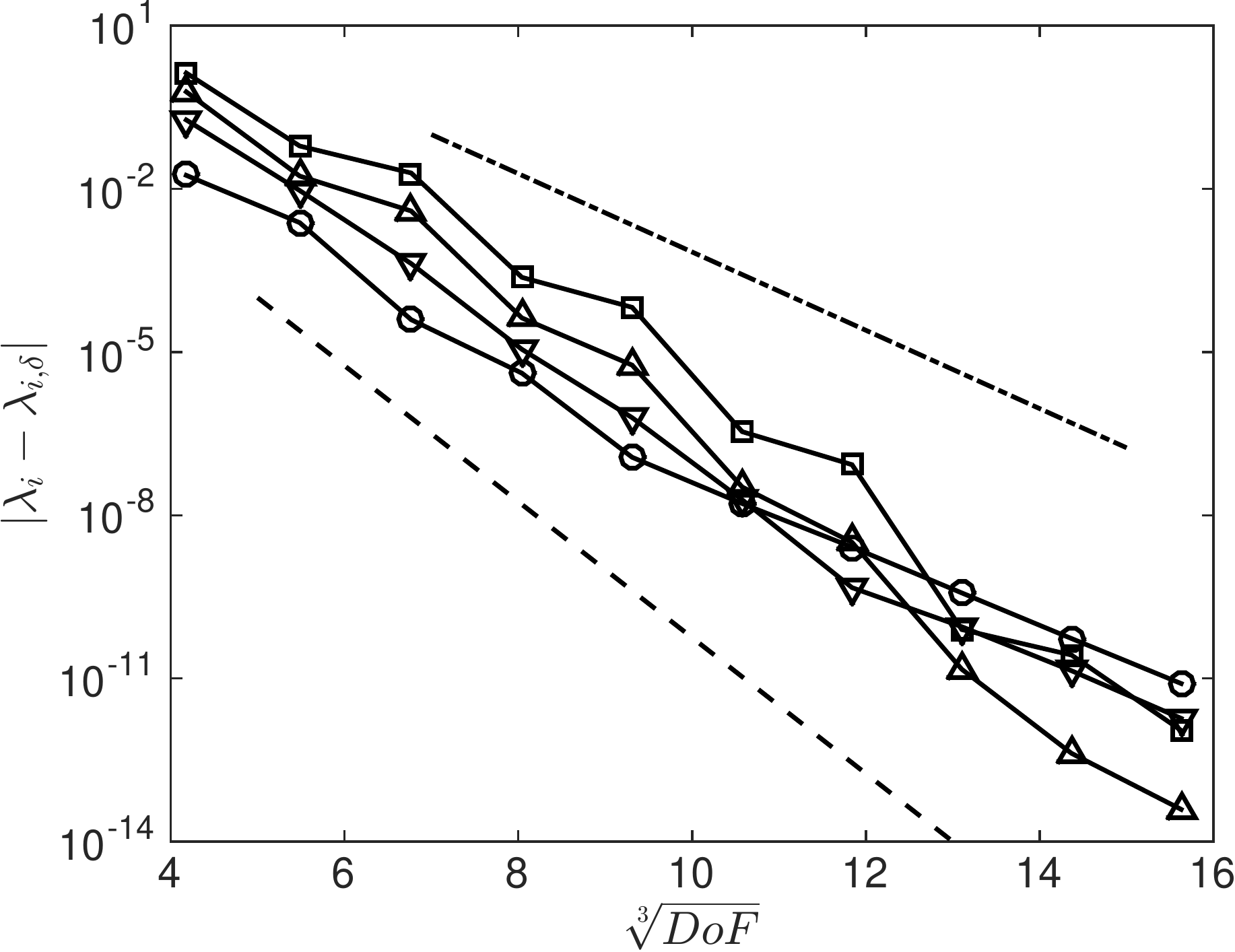}
\hspace*{\fill} (b).  $c=1/2$.\hspace*{\fill}
\end{minipage}\hspace*{\fill}%
\begin{minipage}[h]{0.32\textwidth}
\includegraphics[width=1\textwidth]{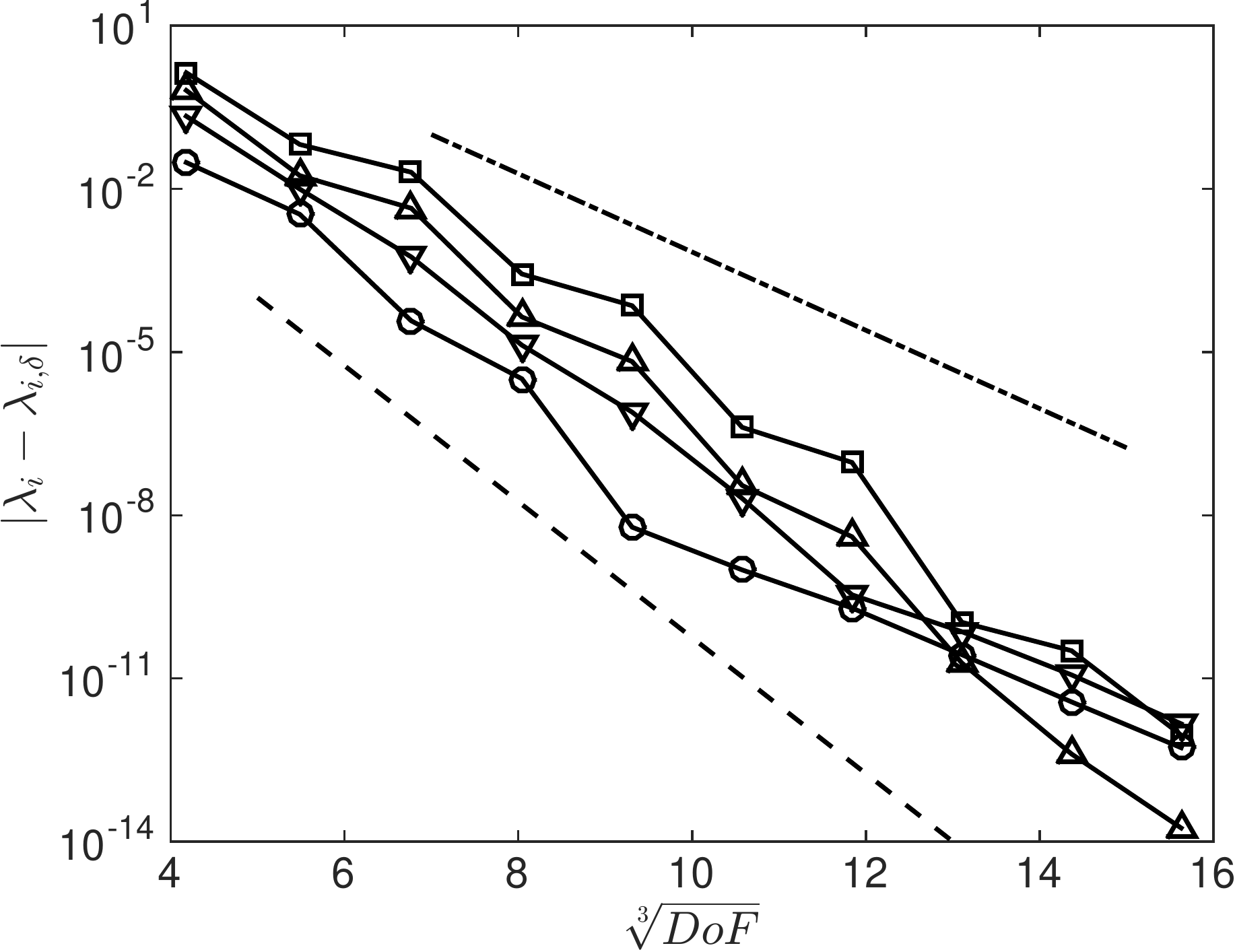}
\hspace*{\fill} (c).  $c=2/3$.\hspace*{\fill}
\end{minipage}\hspace*{\fill}%

\vspace*{0.5em}
\caption{GFEM approximation errors $|\lambda_i-\lambda_{i,\delta}|$  ($\circ: \lambda_1$,  $\triangledown:\lambda_2$, $\vartriangle:\lambda_3$ and  $ \square: \lambda_4$)  versus $\sqrt[3]{DoF}$ on the L-shape domain $[-1,1]^2\setminus ( [0,1]\times[-1,0] )$.
The dashed lines are the exponential $y = 10^{-1.25 \sqrt[3]{DoF}+2 }$, the dash-dot lines are the exponential $y = 10^{-0.72 \sqrt[3]{DoF}+3.32 }$.}
\label{fig:LGFEM}
\end{figure}

\vspace*{-1.5em}
\paragraph{\bf Example 3} At last, let us consider the numerical verification of two isospectral geometries, which possess  the same Laplacian eigenvalues \cite{BCDS94,Kac96,Driscoll97},
see Figure \ref{fig:isospectral}.
To carry out the MSEM experiments, we first partition both geometries with four sectors  $\Omega_i$,  $1\le i\le 4$,
with the radius $R=1/3$,
which are located at the
vertices of the four reentrant/obtuse corners, respectively. Then $\Omega_5:=\Omega \setminus (\cup_{i=1}^4\Omega_i)$
is further decomposed into 10 right triangles (of size $\sqrt{2}$) and 18 curvilinear quadrilaterals, on which the standard
$C^0$-conforming spectral elements are recommended.

\begin{figure}[h!]
\hspace*{\fill}
\begin{minipage}[h]{0.35\textwidth}
\includegraphics[width=1\textwidth]{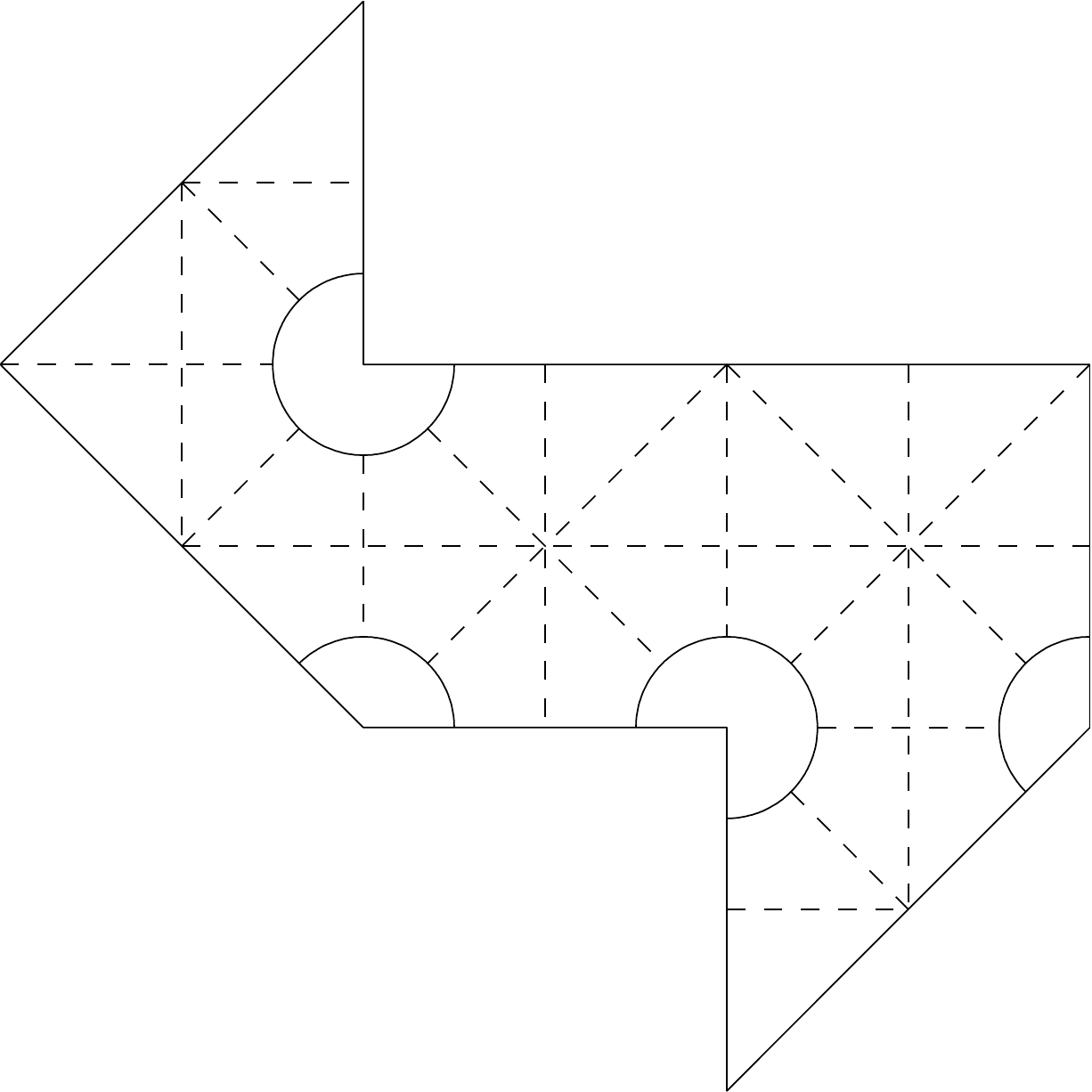}
\hspace*{\fill} (a) \hspace*{\fill}
\end{minipage}\hspace*{\fill}%
\begin{minipage}[h]{0.35\textwidth}
\includegraphics[width=1\textwidth]{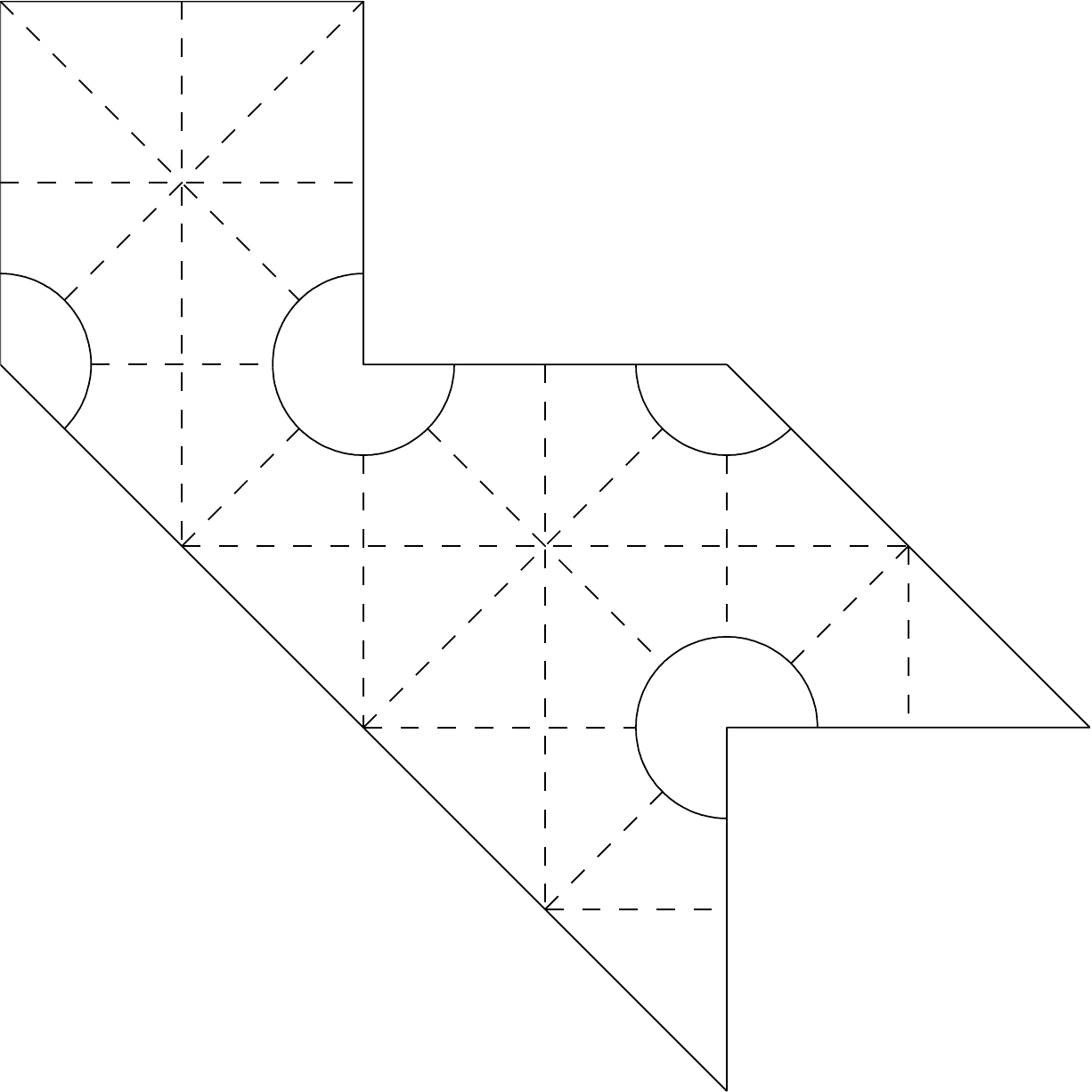}
\hspace*{\fill} (b) \hspace*{\fill}
\end{minipage}\hspace*{\fill}%
\caption{A pair of isospectral geometries and their nonconforming meshes for MSEM.}
\label{fig:isospectral}
\end{figure}

We now tabulate the first 25 computed eigenvalues  in Table \ref{tab:isospec} with $13$ decimal digits
by MSEM with the total degrees of freedom $4457$.  The computed eigenvalues are compared with those in 12 digits  by Driscoll \cite{Driscoll97}   (also by Betcke and Trefethen \cite{BT05,TB06}).
We note that the excerpted eigenvalues in the last column perfectly  match
the leftmost 12 rounded digits of
the approximate eigenvalues on both  the geometry (a) and the geometry (b),
and the second and the third columns differ only in the rightmost digit.
These partially illustrate the isospectral property of the two geometries together
with the effectiveness and efficiency of the MSEM proposed in the current paper.

\begin{table}[h!]
\caption{The 25 smallest numerical Laplacian eigenvalues on the isospectral geometries (a) and (b), together with those excerpted
 from \cite{Driscoll97}.}
\hfill%
\begin{tabular}{c|c|c|c}
No. & MSEM (a)            & MSEM (b)      & Driscoll\\   \hline
1 &   2.5379439997986 &   2.5379439997986 & 2.53794399980 \\ \hline
2 &   3.6555097135244 &   3.6555097135244 & 3.65550971352 \\ \hline
3 &   5.1755593562245 &   5.1755593562245 & 5.17555935622 \\ \hline
4 &   6.5375574437644 &   6.5375574437644 & 6.53755744376 \\ \hline
5 &   7.2480778625641 &   7.2480778625641 & 7.24807786256 \\ \hline
6 &   9.2092949984032 &   9.2092949984031 & 9.20929499840 \\ \hline
7 &  10.5969856913332 &  10.5969856913331 & 10.5969856913 \\ \hline
8 &  11.5413953955859 &  11.5413953955859 & 11.5413953956 \\ \hline
9 &  12.3370055013616 &  12.3370055013617 & 12.3370055014 \\ \hline
10&  13.0536540557280 &  13.0536540557280 & 13.0536540557 \\ \hline
11&  14.3138624642910 &  14.3138624642910 & 14.3138624643 \\ \hline
12&  15.8713026200093 &  15.8713026200093 & 15.8713026200 \\ \hline
13&  16.9417516879721 &  16.9417516879721 & 16.9417516880 \\ \hline
14&  17.6651184368431 &  17.6651184368430 & 17.6651184368 \\ \hline
15&  18.9810673876525 &  18.9810673876526 & 18.9810673877 \\ \hline
16&  20.8823950432823 &  20.8823950432823 & 20.8823950433 \\ \hline
17&  21.2480051773729 &  21.2480051773729 & 21.2480051774 \\ \hline
18&  22.2328517929733 &  22.2328517929735 & 22.2328517930 \\ \hline
19&  23.7112974848240 &  23.7112974848240 & 23.7112974848 \\ \hline
20&  24.4792340692739 &  24.4792340692739 & 24.4792340693 \\ \hline
21&  24.6740110027234 &  24.6740110027235 & 24.6740110027 \\ \hline
22&  26.0802400996599 &  26.0802400996599 & 26.0802400997 \\ \hline
23&  27.3040189211259 &  27.3040189211260 & 27.3040189211 \\ \hline
24&  28.1751285814531 &  28.1751285814533 & 28.1751285815 \\ \hline
25&  29.5697729132392 &  29.5697729132393 & 29.5697729132 \\
\end{tabular}
\hfill
\vspace*{0.3em}
 \label{tab:isospec}
\end{table}

 {\bf Conclusion Remarks.} In this work, we present a novel and effective way to handle operator singularity of the  inverse square potential
as well as the domain corner singularities. Although we do not provide a full convergence analysis, numerical evidences indicate
that our new methods are superior to existing methods including the $hp$ finite element  method using geometric meshes which is known for  the best convergence rate  with the presence of corner singularities.   Therefore, our approach can
serve as a better alternative to solve such kind of problems with singularity.

\appendix
\section{The proof of Lemma  \ref{SOBF} and  Lemma \ref{SOBP}}
\label{AppA}
Let $\nabla_0$ be the spherical gradient, which is the spherical part of $\nabla$ and involves
only derivatives in $\xi$, i.e.,
\begin{align*}
&  \nabla_0 = r( \nabla -  \xi \partial_r), \qquad  x=r\xi, \, \xi\in \sph.
\end{align*}
As a result,  $\xi \cdot \nabla_0=0$ and
\begin{align}
\label{InnerG}
 (\nabla u, \nabla v)_{\Omega}  
 = \big(\nabla_0 u, \nabla_0v)_{r^{-2},\Omega} + \big(\partial_r u, \partial_r v)_{\Omega}, \quad u,v\in H^1(\Omega).
 \end{align}
Moreover,
it  holds that (\cite[p.\,16 and p.\, 26]{DaiXu2013})
\begin{align}
\label{BLSG}
&\Delta_0 = \nabla_0 \cdot \nabla_0,
\\
\label{InnerSG}
&( \nabla_0 u,  \nabla_0v)_{\sph} = - (\Delta_0 u,  v)_{\sph}, \qquad
u \in H^2(\sph),\, v\in H^1(\sph).
\end{align}

We next prove that for any $u,v\in H^1(0,1)$,
\begin{align}
\label{II0}
\begin{split}
\int_0^1  \partial_r \big(&r^{d/2-1-\beta} u\big) \partial_r \big(r^{d/2-1-\beta} v\big)
r^{2\beta+1} dr
= \int^1_0\Big[ r^{2} \partial_r u \partial_r v
+     (\beta^2-(d/2-1)^2)  u v\Big]  r^{d-3}  dr
\\
&+  (d/2-1-\beta) [u(1)v(1)- \delta_{d,2} u(0)v(0) ]
.
\end{split}
\end{align}
Actually, a technical reduction leads to
\begin{align*}
\int_0^1 & \partial_r \big[r^{d/2-1-\beta} u\big] \partial_r \big[r^{d/2-1-\beta} v\big]
r^{2\beta+1} dr
\\
= & \int_0^1 \Big[r^{d-1}\partial_r u \partial_r u  +  r^{d-3}(d/2-1-\beta)^2 u v
+ (d/2-1-\beta) r^{d-2}\partial_r(uv)\Big]
  dr\\
=& \int_0^1 \Big[r^{d-1}\partial_r u \partial_r u  +  r^{d-3}(d/2-1-\beta)^2 u v
- (d-2) (d/2-1-\beta) r^{d-3} uv \Big]
  dr
  \\& +
   (d/2-1-\beta) [u(1)v(1)- 0^{d-2} u(0)v(0) ] \\
= & \int^1_0\Big[ r^{d-1} \partial_r u \partial_r v
+  \big(\beta^2-(d/2-1)^2)\big)  r^{d-3} u v\Big] dr + (d/2-1-\beta) [u(1)v(1)- \delta_{d,2} u(0)v(0) ],
\end{align*}
where the second equality sign was derived by integration by part.

\vspace*{0.8em}

We now concentrate on the proofs of Lemma \ref{SOBF} and \ref{SOBP} in the main body of this paper.
\begin{proof}[Proof of Lemma \ref{SOBF}] We first note that
\begin{align*}
\int_{\ball} u(x) dx  =  \int_{0}^{1}r^{d-1}dr \int_{\sph} u(r\xi) d\sigma(\xi).
\end{align*}
Then, by \eqref{InnerG}, \eqref{InnerSG} and \eqref{eq:LaplaceBeltrami}, we obtain that
\begin{align*}
      \big(\nabla & Q_{k,\ell}^{n}, \nabla Q_{j,\iota}^{m}\big)_{\ball} + c^2 \big(Q_{k,\ell}^{n},Q_{j,\iota}^{m}\big)_{r^{-2},{\ball}}
\\
      =&\, \big(\partial_r  Q_{k,\ell}^{n}, \partial_r Q_{j,\iota}^{m}\big)_{\ball}  + \big(\nabla_0  Q_{k,\ell}^{n}, \nabla_0 Q_{j,\iota}^{m}\big)_{r^{-2},\ball} + c^2 \big(Q_{k,\ell}^{n},Q_{j,\iota}^{m}\big)_{r^{-2},{\ball}}
\\
      =&\, \big(\partial_r  Q_{k,\ell}^{n}, \partial_r Q_{j,\iota}^{m}\big)_{\ball}  - \big(\Delta_0  Q_{k,\ell}^{n}, Q_{j,\iota}^{m}\big)_{r^{-2},\ball} + c^2 \big(Q_{k,\ell}^{n},Q_{j,\iota}^{m}\big)_{r^{-2},{\ball}}
\\
      =&\, \big(\partial_r  Q_{k,\ell}^{n}, \partial_r Q_{j,\iota}^{m}\big)_{\ball}  + \big( c^2+ n(n+d-2)\big) \big(Q_{k,\ell}^{n},Q_{j,\iota}^{m}\big)_{r^{-2},{\ball}}  .
\end{align*}
In the sequel, we temporarily set $q_{k,n}(r)= \frac{2k+2\beta_n}{k+2\beta_n} J^{-1,2\beta_n}_k(2r-1) r^{\beta_n+1-d/2}$
and  get further from \eqref{II0}, \eqref{GJacobir} , \eqref{diff} and \eqref{Jorth} that
\begin{align}
\label{SOBF2}
\begin{split}
      \big(\nabla & Q_{k,\ell}^{n}, \nabla Q_{j,\iota}^{m}\big)_{\ball} + c^2 \big(Q_{k,\ell}^{n},Q_{j,\iota}^{m}\big)_{r^{-2},{\ball}}
\\
 =&\,  \int_{\SS^{d-1}}  Y^n_{\ell}(\xi) Y^m_{\iota}(\xi)  d\sigma(\xi)   \int_{0}^1  \Big[r^{d-1}  \partial_r q_{k,n}   \partial_r q_{j,n}
   +  \big( c^2+ n(n+d-2)\big) r^{d-3} q_{k,n}   q_{j,n}
   \Big] dr
   \\
   =&   \, \omega_d  \delta_{m,n} \delta_{\ell,\iota}      \int_{0}^1  \partial_r \big[r^{d/2-1-\beta_n}   q_{k,n}   \big]
   \partial_r \big[r^{d/2-1-\beta_n} q_{j,n}   \big]
r^{2\beta_n+1}   dr
\\
&\, + \, \omega_d  \delta_{m,n} \delta_{\ell,\iota}   (\beta_n+1-d/2) [q_{k,n}(1)q_{j,n}(1)- \delta_{d,2} q_{k,n}(0)q_{j,n}(0) ]
\\
=&   \, \omega_d  \delta_{m,n} \delta_{\ell,\iota}   \frac{(2k+2\beta_n)(2j+2\beta_n)}{(k+2\beta_n)(j+2\beta_n)}   \int_{0}^1  \partial_rJ^{-1,2\beta_n}_k(2r-1)  \partial_rJ^{-1,2\beta_n}_j(2r-1)r^{2\beta_n+1} dr
\\
&\, +  \, \omega_d  \delta_{m,n} \delta_{\ell,\iota} (\beta_n+1-d/2)  \frac{(2k+2\beta_n)(2j+2\beta_n)}{(k+2\beta_n)(j+2\beta_n)}  [\delta_{k,0}\delta_{j,k}
-\delta_{\beta_n+1-d/2,0} \delta_{d,2}  ]
\\
=&   \, \omega_d  \delta_{m,n} \delta_{\ell,\iota}  (2k+2\beta_n) (2j+2\beta_n)    \int_{0}^1  J^{0,2\beta_n+1}_{k-1}(2r-1)  J^{0,2\beta_n+1}_{j-1}(2r-1)r^{2\beta_n+1} dr
\\
&\, +  \, \omega_d  \delta_{m,n} \delta_{\ell,\iota} (\beta_n+1-d/2)  \delta_{k,0}\delta_{j,k}
\\
=&   \, \omega_d  \delta_{m,n} \delta_{\ell,\iota} \delta_{k,j}
\left[ (2k+2\beta_n) (1-\delta_{k,0})  + (\beta_n-d/2+1) \delta_{k,0}   \right],
\end{split}
\end{align}
which gives \eqref{Qortha}.

Next, it is easy to see that
 \begin{align}
 \label{InnerQ}
 \begin{split}
 (Q_{k,\ell}^{n},&Q_{j,\iota}^{m})_{\ball}
 =\, \int_{\sph} Y_{\ell}^n(\xi) Y_{\ell}^n(\xi) d\sigma(\xi)
 \\
 & \times  \frac{(2k+2\beta_n)(2j+2\beta_m)}{(k+2\beta_n)(j+2\beta_m)} \int_{0}^1  J^{-1,2\beta_n}_k(2r-1) J^{-1,2\beta_m}_j(2r-1) r^{\beta_n+\beta_m+1}  dr \,
 \\
 =&\, \omega_d\delta_{n,m} \delta_{k,j}  \frac{(2k+2\beta_n)(2j+2\beta_m)}{(k+2\beta_n)(j+2\beta_m)}  \int_{0}^1  J^{-1,2\beta_n}_k(2r-1) J^{-1,2\beta_m}_j(2r-1) r^{2\beta_n+1}  dr.
 \end{split}
 \end{align}
To proceed the proof of \eqref{Qorthb},
we  shall resort to the following identity on generalized Jacobi polynomials,
\begin{align}
\label{Jac10}
J^{\alpha,\beta}_k(\zeta) = \frac{k+\alpha+\beta+1}{2k+\alpha+\beta+1}J^{\alpha+1,\beta}_k(\zeta)
-\frac{k+\beta}{2k+\alpha+\beta+1}J^{\alpha+1,\beta}_{k-1}(\zeta),
\end{align}
which is stemmed from \cite[p.\,304]{AAR1999}  by extension.
Using \eqref{Jac10} twice together with \eqref{symmJ} yields
\begin{align*}
\begin{split}
&J^{\alpha,\beta}_k(\zeta) = \frac{(k+\alpha+\beta+1)(k+\alpha+\beta+2) }{(2k+\alpha+\beta+1)(2k+\alpha+\beta+2)}
 J^{\alpha+1,\beta+1}_k(\zeta)
 \\
&\quad+\frac{(\alpha-\beta)(k+\alpha+\beta+1)}{(2k+\alpha+\beta)(2k+\alpha+\beta+2)} J^{\alpha+1,\beta+1}_{k-1}(\zeta)
-\frac{(k+\alpha)(k+\beta)}{(2k+\alpha+\beta)(2k+\alpha+\beta+1)}J^{\alpha+1,\beta+1}_{k-2}(\zeta).
\end{split}
\end{align*}
In particular,
\begin{align}
\label{Jac11}
\frac{2k+\beta}{k+\beta} J_{k}^{-1,\beta} = \frac{k+\beta+1}{2k+\beta+1} J_{k}^{0,\beta+1}
-\frac{(1+\beta)(2k+\beta)}{(2k+\beta-1) (2k+\beta+1)} J_{k-1}^{0,\beta+1}
-\frac{k-1}{2k+\beta-1}J_{k-2}^{0,\beta+1} .
 \end{align}
 Then a combination of  \eqref{InnerQ}, \eqref{Jac11} and \eqref{Jorth}  immediately yields \eqref{Qorthb}.
This completes the proof of Lemma \ref{SOBP}.
\end{proof}

\begin{proof}[Proof of Lemma \ref{SOBP}]
Let us temporarily set $p_{k,n}(r)= \frac{2k+\beta_n}{k+\beta_n} J^{-1,\beta_n}_k(2r^2-1) r^{\beta_n+1-d/2}$.
Then a  similar reduction as in \eqref{SOBF2} yields
\begin{align*}
      \big(\nabla & P_{k,\ell}^{n}, \nabla P_{j,\iota}^{m}\big)_{\ball} + c^2 \big(P_{k,\ell}^{n},P_{j,\iota}^{m}\big)_{r^{-2},{\ball}}
   \\
   =&   \, \omega_d  \delta_{m,n} \delta_{\ell,\iota}      \int_{0}^1  \partial_r \big[r^{d/2-1-\beta_n}   p_{k,n}   \big]
   \partial_r \big[r^{d/2-1-\beta_n} p_{j,n}   \big]
r^{2\beta_n+1}   dr
\\
&\, + \, \omega_d  \delta_{m,n} \delta_{\ell,\iota}   (\beta+1-d/2) [p_{k,n}(1)p_{j,n}(1)- \delta_{d,2} p_{k,n}(0)p_{j,n}(0) ]
\\
=&   \, \omega_d  \delta_{m,n} \delta_{\ell,\iota}   \frac{(2k+\beta_n)(2j+\beta_n)}{(k+\beta_n)(j+\beta_n)}   \int_{0}^1  \partial_rJ^{-1,\beta_n}_k(2r^2-1)  \partial_rJ^{-1,\beta_n}_j(2r^2-1)r^{2\beta_n+1} dr
\\
&\, +  \, \omega_d  \delta_{m,n} \delta_{\ell,\iota} (\beta_n+1-d/2)   \frac{(2k+\beta_n)(2j+\beta_n)}{(k+\beta_n)(j+\beta_n)}[\delta_{k,0}\delta_{j,k}
-\delta_{\beta_n+1-d/2,0} \delta_{d,2}  ]
\\
=&   \, \omega_d  \delta_{m,n} \delta_{\ell,\iota}  2 (2k+\beta_n) (2j+\beta_n)    \int_{0}^1  J^{0,\beta_n+1}_{k-1}(2r^2-1)  J^{0,\beta_n+1}_{j-1}(2r^2-1)r^{2\beta_n+2} dr^2
\\
&\, +  \, \omega_d  \delta_{m,n} \delta_{\ell,\iota} (\beta_n+1-d/2)  \delta_{k,0}\delta_{j,k}
\\
=&   \, \omega_d  \delta_{m,n} \delta_{\ell,\iota} \delta_{k,j}
\left[ 2(2k+\beta_n) (1-\delta_{k,0})  + (\beta_n-d/2+1) \delta_{k,0}   \right],
\end{align*}
which gives \eqref{Portha}.

Further, we note that
 \begin{align*}
 \begin{split}
 (P_{k,\ell}^{n},P_{j,\iota}^{m})_{\ball}
  =&\, \omega_d\delta_{n,m} \delta_{k,j}  \frac{(2k+\beta_n)(2j+\beta_m)}{(k+\beta_n)(j+\beta_m)}  \int_{0}^1  J^{-1,\beta_n}_k(2r^2-1) J^{-1,\beta_m}_j(2r^2-1) r^{2\beta_n+1}  dr.
 \end{split}
 \end{align*}
 Then \eqref{Porthb} is an immediate consequence of   \eqref{Jac10} and \eqref{Jorth}.
 This proof of Lemma \ref{SOBP} is now  completed.
\end{proof}

\bibliographystyle{siam}
\bibliography{SchRef}
\end{document}